\documentclass{amsart}

\theoremstyle{definition}

\theoremstyle{remark}

\numberwithin{equation}{section}
\usepackage{graphicx}
\usepackage{latexsym}
\usepackage{subfigure}
\usepackage{amsmath}
\usepackage{amssymb}
\usepackage{amsfonts}
\usepackage{verbatim}
\usepackage{mathrsfs}
\usepackage[latin1]{inputenc}
\usepackage{color}
\usepackage[colorlinks,citecolor=blue,urlcolor=blue]{hyperref}
\usepackage{caption}

\usepackage{datetime}

\newtheorem{tm}{Theorem}[section]
\newtheorem{rk}{Remark}[section]

\newtheorem{prop}{Proposition}[section]
\newtheorem{lm}{Lemma}[section]
\newtheorem{cor}{Corollary}[section]
\newtheorem{ex}{Example}[section]
\usepackage{enumitem}

\newcommand{\N}{\mathbb N}
\newcommand{\R}{\mathbb R}

\newcommand{\Log}{\text{Log}}
\newcommand{\Arg}{\text{Arg}}

\newcommand{\<}{\langle}
\renewcommand{\>}{\rangle}
\numberwithin{figure}{section}

\newcommand{\cui}[1]{{\color{blue} [cui: #1]}}

\begin{document}

\title[Time Discretizations of Wasserstein-Hamiltonian Flows]
{Time discretizations of Wasserstein-Hamiltonian flows}

\author{Jianbo Cui}
\address{School of Mathematics, Georgia Tech, Atlanta, GA 30332, USA}
\curraddr{}
\email{jcui82@math.gatech.edu}
\thanks{}
\author{Luca Dieci}
\address{School of Mathematics, Georgia Tech, Atlanta, GA 30332, USA}
\curraddr{}
\email{dieci@math.gatech.edu}
\thanks{}
\author{Haomin Zhou}
\address{School of Mathematics, Georgia Tech, Atlanta, GA 30332, USA}
\curraddr{}
\email{hmzhou@math.gatech.edu}
\thanks{}




\dedicatory{}

\begin{abstract}
We study discretizations of Hamiltonian systems on the probability density manifold equipped with the $L^2$-Wasserstein metric.
Based on discrete optimal transport theory, several Hamiltonian systems on graph (lattice) with different 
weights are derived, which can be viewed as spatial discretizations to the original Hamiltonian systems. 
We prove the consistency and provide the approximate orders for those discretizations. By regularizing the system using Fisher information, we deduce
an explicit lower bound for the density function, which guarantees that symplectic schemes can be used to discretize in time. 
Moreover, we show desirable long time behavior of these schemes, and demonstrate their performance 
on several numerical examples. 
\end{abstract}


\maketitle

\section{Introduction}
\noindent
In recent years, there has been a lot of interest in studying
Hamiltonian systems defined on the probability space endowed with the $L^2$-Wasserstein metric, 
also known as Wasserstein manifold, and several authors have been concerned with
%
their connections to some well-known partial differential equations (PDEs); e.g., see \cite{AG08, GKP11, Vil03}. 

Our present study is influenced by the point of view in \cite{CLZ20}, where the authors showed 
that the push-forward density of a 
classical Hamiltonian vector field in phase space is a Hamiltonian flow on the Wasserstein manifold.
To be more precise, consider a Hamiltonian system subject to initial condition $(q_0, v_0)$: 
\begin{equation}\label{HamIVP}\begin{split}
dv &=-\frac {\partial H}{\partial q}(v,q),\; v(0)=v_0,\\
dq &=\frac {\partial H}{\partial v}(v,q),\; q(0)=q_0,
\end{split}\end{equation}
where the position $q \in \mathbb R^d $, the conjugate momenta $v \in \mathbb R^d$, and the real valued Hamiltonian 
$H\in \mathcal C^2(\mathbb R^d\times\mathbb R^d)$, $d\in \N^+$.  
Let $q(t)$, $v(t)$ denote the solution of \eqref{HamIVP}.
If we assume that the initial position $q_0$ is a random vector associated to a joint probability density $\rho_0$,
then the density $\rho$ of $q(t)$ satisfies
\begin{equation}\label{FP}\begin{split}
\partial_t \rho+\nabla \cdot (\frac {\partial H}{\partial v}\rho )&=0,\\
\partial_t v+\nabla v \cdot v + \nabla \cdot \frac {\partial H}{\partial q}&=0.
\end{split}\end{equation}
By introducing $v =\nabla S $, one can rewrite this system as the {\emph{Wasserstein-Hamiltonian}} system
\begin{equation}\label{WassHam1}\begin{split}
\partial_t \rho + \nabla \cdot (\frac {\partial H}{\partial v} \rho )&=0,\\
\partial_t S +\frac 12  |\nabla S|^2 + \frac {\partial H}{\partial q} &= C(t),
\end{split}\end{equation}
where $C(t)$ is a function depending only on $t$ and $|\nabla S|^2=\nabla S \cdot \nabla S$.

The formulation \eqref{WassHam1} is remarkably powerful and general.  Indeed, 
with different choices of the Hamiltonian $H$, the Wasserstein-Hamiltonian system \eqref{WassHam1}
leads to differential equations arising in many different applications. For example,
by taking $H(v,q) = \frac{1}{2} |v|^2$, one obtains the well-known geodesic equations between
two densities $\rho^0$ and $\rho^1$ on the Wasserstein manifold:
\begin{equation}\label{GeodEqn1}\begin{split}
\partial_t \rho + \nabla \cdot ({\rho }\nabla S)&=0, \\
\partial_t S +\frac 12 |\nabla S|^2 &=0,
\end{split}\end{equation}
with $\rho(0)=\rho^0, \rho(1)=\rho^1$.  In the seminal paper \cite{BB00},
it has been proven that the solution of \eqref{GeodEqn1} is a minimizer of the following variational problem,
commonly known as the {\emph{Benamou-Brenier formula}}:
\begin{equation}\label{min-e}\begin{split}
g_W(\rho_0,\rho_1)^2&=\inf_{v_t}\{\int_0^1\<v ,v \>_{\rho}dt\ : \, 
\partial_t \rho  + \nabla\cdot(\rho  v )=0, \rho_0=\rho^0, \rho_1=\rho^1\},
\end{split}\end{equation}
where $ \<v,v\>_{\rho}:=\int_{\mathbb R^d} |v|^2 \rho dx$.  As shown in \cite{BB00}, the optimal value $g_W(\rho^0, \rho^1)$ is the $L^2$-Wasserstein distance between $\rho^0$ and $\rho^1$.

%

Similarly, a problem known as the {\emph{Schr\"odinger Bridge Problem}} can be stated as 
\begin{equation}\label{SchrBr}
\inf_{v}\Big\{\int_0^1\frac 12 \<v ,v \>_{\rho}+\frac {\hbar^2}{8} I (\rho)dt\ :\,  \partial_t \rho  + \nabla\cdot(\rho  v )=0, \rho_0=\rho^0, \rho_1=\rho^1 \Big\},
\end{equation} 
where $\hbar>0$ and $I(\rho):=\<\nabla \log(\rho), \nabla \log(\rho)\>_{\rho}$ is the {\emph{Fisher information}}. 
The minimizer of \eqref{SchrBr} satisfies the Wasserstein-Hamiltonian system \eqref{WassHam1}
with the energy $\mathcal H(v,\rho)=\frac 12\int_{\mathbb R^d} |v|^2\rho dx-\frac {\hbar^2} 8 I(\rho)$ in density space.
Although the Schr\"odinger Bridge problem is nearly 100 years old, it has recently received attention in control theory
and machine learning, see \cite{Sch31,Leo14,Pav03}.
 
If we change the sign of the Fisher information term in \eqref{SchrBr}, we get 
\begin{equation}\label{Madelung1}
\inf_{v}\Big\{\int_0^1\frac 12\<v ,v \>_{\rho}-\frac {\hbar^2}{8} I (\rho)dt\ :\,  \partial_t \rho  + \nabla\cdot(\rho  v )=0, \rho_0=\rho^0, \rho_1=\rho^1 \Big\},
\end{equation} 
and this is the variational formula that Nelson used to derive the Schr\"odinger equation \cite{Nel66}.  
Its reformulation as Wasserstein-Hamiltonian system becomes the well known 
Madelung system \cite{Mad27}. 

\begin{rk}
The Benamou-Brenier formula \eqref{min-e} has been extensively used to study Wasserstein
gradient flows; e.g., see \cite{Laf88, Ott01, Vil03, Vil09}.
%
However, unlike the variational formulations from \eqref{min-e} that use 2-point boundary values, 
much less is known for Wasserstein-Hamiltonian flows, hence for solutions of \eqref{WassHam1} for given
initial values. The problem is subtle, for once because --depending on the initial condition-- the solution of
\eqref{WassHam1} may develop singularities. Moreover, there are several important properties of the
Wasserstein-Hamiltonian flow, such as preservation of symplectic structure and other quantities, 
which make the numerical approximation of Wasserstein-Hamiltonian flows quite challenging.
These considerations have motivated us to carry out the present numerical study. 
\end{rk}

To the best of our knowledge, prior to our work, there are no numerical analysis results
on the full  (i.e., space and time) discretization of Wasserstein-Hamiltonian systems.  The way we approach
this problem is by first using discrete optimal transport techniques to obtain  
Wasserstein-Hamiltonian systems on a graph, and view these as spatial discretizations of the original 
Wasserstein-Hamiltonian system.  We explicitly show the consistency of the semi-discretizations,  
and derive lower bounds for the probability density function on different graphs. 
Then, we combine ideas from discrete optimal transport and symplectic integration to 
construct fully discrete numerical schemes for the solution of the Wasserstein-Hamiltonian system. 

We would like to emphasize the crucial role of Fisher information in our study. Fisher information is widely 
used in many areas in statistics, physics and biology (see e.g. \cite{Fri04}). It appears naturally in some  
Wasserstein-Hamiltonian systems, such as \eqref{SchrBr}, and it has recently been
used as a regularization term in computations of optimal transport and Wasserstein gradient flows 
(see \cite{LYO18, LLW19} and references therein).
Our analysis in this paper indicates that there are clear benefits to using Fisher information
as a regularization term for the approximation of  Wasserstein-Hamiltonian flows: 
it leads to maintaining positivity of the density function, it is conducive to having schemes
that are time reversible and gauge invariant, that preserve mass and  
symplectic structure, and that almost preserve energy for very long times (of ${\mathcal{O}}(\tau^{-r})$,
where $r$ is the order of the numerical scheme and $\tau$ is the time step-size).

This paper is organized as follows. In Section \ref{sec-ode}, we introduce the  Wasserstein-Hamiltonian vector field
on graphs and study its properties. 
In Section \ref{sec-low}, we give an explicit lower bound of the probability density for the discrete  Wasserstein-Hamiltonian 
flow on different graphs; the proofs of the technical results in this Section are in the Appendix at the end of the paper.
Section \ref{sec-disc} is devoted to constructing and analyzing time discretizations, and in particular we develop and analyze
symplectic schemes.  To compare with the results we obtain using Fisher information as regularization device, in this Section
\ref{sec-disc} we also analyze regularized schemes obtained by adding a viscosity term.
%
Several numerical examples are given in Section \ref{sec-test}.

\section{Wasserstein-Hamiltonian Vector Field and Flow on a Finite Graph}\label{sec-ode}

Our goal in this Section is three-fold: to introduce a special vector field (the Wasserstein-Hamiltonian vector field) on
a graph, to recognize it as a consistent spatial discretization of the PDE \eqref{WassHam1}, and to show relevant properties 
of the associated flow.  The latter effort is a prelude to Section \ref{sec-disc} where also the time discretization
is examined.

\subsection{Wasserstein-Hamiltonian flows via discrete optimal transport}

Consider a graph $G=(V,E,\Omega)$ with a node set $V=\{a_i\}_{i=1}^N$, an edge set $E$, and $\omega_{jl}\in \Omega$
are the weights of the edges: 
$\omega_{jl}=\omega_{lj}>0$, if there is an edge between $a_j$ and $a_l$, and $0$ otherwise.
Below, we will write $(i,j)\in E$ to denote the edge in $E$ between the vertices $a_i$ and $a_j$. 
Finally, throughout this paper, we assume that $G$ is an undirected, strongly connected graph with 
no self loops or multiple edges.

Let us denote the set of discrete probabilities on the graph by ${\mathcal{P}}(G)$:
$$\mathcal P(G)=\{(\rho)_{j=1}^N\ :\, \sum_{j}\rho_j =1, \rho_j\ge 0,\; \text{for} \; j\in V\},$$ 
and let $\mathcal P_o(G)$ be its interior (i.e., all  $\rho_j> 0$, for $a_j\in V$).
Let $\mathbb V_j$ be a linear potential on each node $a_j$, and $\mathbb W_{jl}=\mathbb W_{lj}$ an 
interactive potential between nodes $a_j,a_l$.
We let $N(i)=\{a_j\in V: (i,j)\in E\}$ be the adjacency set of node $a_i$ and 
$\theta_{ij}(\rho)$ be the density dependent weight on the edge $(i,j)\in E$.

Now, let us define the discrete Lagrange functional on the graph by 
\begin{equation}\label{DiscLag}
\mathcal L (\rho,v)=\int_0^1 \bigl[
\frac 12\<v,v\>_{\theta(\rho)}-\mathcal V(\rho)-\mathcal W(\rho)-\beta I(\rho)\bigr] dt,
\end{equation}
where: $\rho(\cdot)\in \mathcal P_o(G)$,   the vector field $v$ is a skew-symmetric matrix on $E$. And the inner product of two vector fields  $u,v$ is defined by 
$$\<u,v\>_{\theta(\rho)}:=\frac 12\sum_{(j,l)\in E}u_{jl}v_{jl}\theta_{jl}.$$ 
The total linear potential $\mathcal V$ and interaction potential $\mathcal W$ are given by
$$\mathcal V(\rho)=\sum_{i=1}^N\mathbb V_i\rho_i,\,\,
\mathcal W(\rho)=\frac 12\sum_{i,j}\mathbb W_{ij}\rho_i\rho_j.$$ 
The parameter $\beta\ge 0$, and
the {\emph{discrete Fisher information}} is defined by 
\begin{equation}\label{DiscFisher}
I(\rho)=\frac 12\sum_{i=1}^N\sum_{j\in N(i)}\widetilde \omega_{ij}|\log(\rho_i)-\log(\rho_j)|^2\widetilde \theta_{ij}(\rho)
\end{equation}
\begin{rk}
Note that in \eqref{DiscFisher}, we are allowing use of edge weights $\widetilde \omega$ and
probability weights $\widetilde \theta$, different from $\omega$ and $\theta$; this added flexibility
may be exploited to obtain more robust space discretizations than those obtained when 
choosing $\widetilde \omega = \omega$ and $\widetilde \theta=\theta$, as done in \cite{CLZ19b}.
\end{rk}

The overall goal is to find the minimizer of $\mathcal L(\rho,v)$ subject to the constraint
\begin{align*}
\frac {d\rho_i}{d t}+div_G^{\theta}(\rho v)=0,
\end{align*}
where the discrete divergence of the flux function $\rho v$  is defined as 
$$div_G^{\theta}(\rho v):=-(\sum_{l\in N(j)}\sqrt{\omega_{jl}}v_{jl}\theta_{jl}).$$
As shown in \cite{CLZ19b}, the critical point $(\rho, v)$ of $\mathcal L$ satisfies 
$v=\nabla_G S:=\sqrt{\omega_{jl}}(S_j-S_l)_{(j,l)\in E}$ for some function $S$ on $V$.  As a consequence,
the minimization problem leads to 
the following discrete  Wasserstein-Hamiltonian vector field on the graph $G$:
\begin{equation}\label{dhs}\begin{split}
&\frac {d\rho_i}{d t}+\sum_{j\in N(i)}\omega_{ij}(S_j-S_i)\theta_{ij}(\rho)=0,\\
&\frac {d S_i}{dt}+\frac 12\sum_{j\in N(i)}\omega_{ij}(S_i-S_j)^2 \frac {\partial \theta_{ij}(\rho)}{\partial \rho_i}+\beta \frac {\partial I(\rho)}{\partial \rho_i}+\mathbb V_i+\sum_{j=1}^N\mathbb W_{ij}\rho_j=0.
\end{split}\end{equation}
With respect to the variables $\rho$ and $S$, we can rewrite \eqref{dhs} as a Hamiltonian system with
Hamiltonian function 
$\mathcal H(\rho,S)=K(S,\rho)+\mathcal F(\rho),$ where 
$ K(S,\rho):=
\frac 12 \<\nabla_G S, \nabla_G S\>_{\theta(\rho)}$ and $\mathcal F(\rho):=\beta I(\rho)+\mathcal V(\rho)+\mathcal W(\rho).$ 
In particular, if $\beta=0,$ $\mathcal V=0,$ and $\mathcal W=0$, the infimum of 
$\mathcal L(\rho,v)$ induces the Wasserstein metric on the graph, which is a discrete version of  Benamou-Brenier formula:
\begin{align*}
W(\rho^0,\rho^1):=\inf_{v}\Big\{\sqrt{\int_{0}^1\<v,v\>_{\theta(\rho)}}dt \,\ : \,
\frac{d\rho}{dt}+div_G^{\theta}(\rho v)=0, \; \rho(0)=\rho^0,\; \rho(1)=\rho^1\Big\}.
\end{align*} 

The following example illustrates the importance of adding Fisher information in order to
regularize the discrete Hamiltonian, so to avoid development of singularities when solving the
initial value problem \eqref{dhs}.

\begin{ex}\label{ex-1}
Consider a 2-point graph $G$. Let $\rho_1(0), \rho_2(0)>0$ and $S_1(0), S_2(0)$ be the corresponding 
initial values on the two nodes, 
take the weights to be constant (e.g., take them to be $1$)
and let $\mathcal F$ be some other assigned potential on the nodes.
By choosing $\theta_{12}=\theta_{21}=\frac {\rho_1+\rho_2}2,$ \eqref{dhs} becomes 
\begin{equation}\label{fds}\begin{split}
\dot \rho_1&= - (S_2-S_1)\frac{\rho_1+\rho_2}2,\,\ 
\dot \rho_2= - (S_1-S_2)\frac {\rho_1+\rho_2}2,\\
\dot S_1&=-\frac 14|S_2-S_1|^2-\frac {\delta \mathcal F} {\delta \rho_1},\,\ 
\dot S_2=-\frac 14 |S_1-S_2|^2-\frac {\delta \mathcal F} {\delta \rho_2}.
\end{split}\end{equation}
Combining the above equations and using $\rho_1+\rho_2=1$, we get 
\begin{align*}
\frac {\partial (\rho_1-\rho_2)}{\partial t}&=- (S_2-S_1) \\
\frac {\partial (S_1-S_2)}{\partial t}&=\frac {\delta \mathcal F} {\delta \rho_2}-\frac {\delta \mathcal F} {\delta \rho_1}.
\end{align*}
Now, we claim that if $\mathcal F$ has no singularity on the boundary of $\mathcal P(G)$, then
positivity of $\rho_1, \rho_2$ may fail. For example, taking 
$\mathcal F(\rho_1,\rho_2)=\frac 12\rho_1^2+\frac 12\rho_2^2$, we get 
$\rho_1(t)-\rho_2(t)=(\rho_1(0)-\rho_2(0))\cos(t) +(S_1(0)-S_2(0))\sin(t)$. 
Then, we  obtain 
\begin{align*}
\rho_1(t)&=\frac 12+ \frac 12\cos(t)(\rho_1(0)-\rho_2(0))+
\frac 12\sin(t)(S_1(0)-S_2(0)),\\
\rho_2(t)&=\frac 12+ \frac 12\cos(t)(\rho_2(0)-\rho_1(0))+
\frac 12\sin(t)(S_2(0)-S_1(0)).
\end{align*}
It is clear that one of the density value can be a negative number if $|S_1(0)-S_2(0)|>1$.
When taking $S_1(0)=S_2(0)$, the solution can be given in the following cases, 
\begin{align*}
&\rho_1(t)=\rho_2(t)=\frac 12,\; \text{if}\; \rho_1(0)=\rho_2(0),\\
&\rho_1(t)>0, \rho_2(t)>0,\; \text{if} \;|\rho_1(0)-\rho_2(0)|<1,\\
&\rho_1(n\pi)=0,\; \text{or} \; \rho_2(n\pi)=0,\; \text{if} \; |\rho_1(0)-\rho_2(0)|=1.
\qed
\end{align*}
\end{ex}

Let us denote with $T^*$
the first time for which $\lim_{t\to T^*}\rho_i(t)\le 0$ or $\lim_{t\to T^*}S_i(t)=\infty$ for some index $i$.
Following arguments similar to those in \cite{CLZ19b}, we have the following result.
\begin{prop}\label{well-dhs}
Consider \eqref{dhs} and assume that $\beta\ge 0$. Then, for any $\rho^0\in \mathcal P_o(G)$ and any
function $S^0$ on $V$, there exists a unique solution of \eqref{dhs} and it satisfies the following properties
(i)-(vi).
\begin{enumerate}[label=(\roman*)]
\item Mass is conserved: before time $T^*$,
$$\sum_{i=1}^N\rho_i(t)= \sum_{i=1}^N\rho_i^0=1.$$
\item Energy is conserved: before time $T^*$,$$\mathcal H(\rho(t),S(t))=\mathcal H(\rho^0,S^0).$$
\item The solution is time reversible:  if $(\rho(t), S(t))$ is the solution of \eqref{dhs}, then 
$(\rho(-t), -S(-t))$ also solves \eqref{dhs}.
\item It is time transverse invariant with respect to the linear potential: 
if $\mathbb V^{\mathbb \alpha}=\mathbb V-\mathbb \alpha$, then $S^{\mathbb \alpha}=S+\mathbb \alpha t$ is 
the solution of \eqref{dhs} with potential $\mathbb V^{\mathbb \alpha}$.
\item A time invariant $\rho^*\in \mathcal P_o(G)$ and $S^*(t)=-v t$ form an interior stationary solution of \eqref{dhs} 
if and only if $\rho^*$ is the critical point of $\min_{\rho\in \mathcal P_o(G)}\mathcal H(\rho,S)$ and 
$v=\mathcal H(\rho^*)+\frac 12 \sum_{i=1}^N\sum_{j=1}^N \mathbb W_{ij}\rho^*_i\rho^*_j$. 
\item Assuming that $\beta>0$ and $\widetilde \theta_{ij}(\rho)=0$ only if $\rho_i=\rho_j=0$, 
then there exists a compact set $B\subset \mathcal P_o(G)$ such that $\rho(t)\in \mathcal P_o(G)$ for all $t>0$.
\end{enumerate}
\end{prop}
\begin{proof}
The proof of properties (i)-(v) is the same (except for the use of $\theta_{ij}$ instead of $\widetilde \theta_{ij}$) 
as that of \cite[Theorem 6]{CLZ19b}, thus we omit it.  Here we only prove (vi).
Since the coefficient of \eqref{dhs} is locally Lipschitz and $\rho^0\in \mathcal P_{o}(G)$, 
it is not difficult to obtain the local existence of  a unique solution $(\rho(t),S(t))$ in $[0,T^*),$
where $T^*>0$ is the largest time for which $(\rho(t),S(t))$ exists and $\rho(t) \in  \mathcal P_{o}(G)$. 
Thus, it suffices to show that the local solution can be extended to $T^*=\infty$, i.e., 
to show that the boundary is a repeller for $\rho(t)$.
Consider $B=\{\rho\in \mathcal P_o(G) \; | \; \beta I(\rho)\le \mathcal H(\rho,S)-\mathcal F(\rho)\}.$ 
It is enough to prove that $I(\rho)$ is positive infinity on the boundary.  Denote 
$M:=\mathcal H(\rho,S)-\inf_{\rho\in  \mathcal P_o(G) }\mathcal F(\rho).$
If there exists $\rho$ such that $\min_i \rho_i=0,$ and $\beta I(\rho)\le M$, then 
$M\ge \frac \beta 2\sum_{i}\sum_{j\in N(i)}\widetilde \omega_{ij}(\log(\rho_i)-\log(\rho_j))^2\widetilde \theta_{ij}(\rho).$
For some $i$, we have that $\rho_i=0$ and that for $j\in N(i)$, 
$$\beta \widetilde \omega_{ij}(\log(\rho_i)-\log(\rho_j))^2\widetilde \theta_{ij}(\rho)\le M.$$
This implies that $\widetilde \theta_{ij}(\rho)=0$ for any $j\in N(i)$.
Since $G$ is connected and $V$ is a finite set, we get that 
$\max_i \rho_i=0$, which leads to a contradiction.
\end{proof}

From Property (vi) in Proposition \ref{well-dhs}, it is clear that the Fisher information term helps maintain
positivity of the density function in the  Wasserstein-Hamiltonian flow. 
This fact motivated us to regularize the discretized  Wasserstein-Hamiltonian system \eqref{dhs}
by adding Fisher information, 
and the details are discussed in Section \ref{sec-pde}.

There are many choices for $\theta_{ij}$ and $\widetilde \theta_{ij}$, as long as we require that  
$\widetilde \theta_{ij}(\rho)=0$ only if $\rho_i=\rho_j=0$, as this is needed in order to get the lower
bound estimate on the density in Section \ref{sec-low}.  For $\theta_{ij}$, one can choose the upwind weight,
$\theta_{ij}^{U}(\rho)=\rho_{i}$, if $S_{j}>S_{i}$, 
the average weight $\theta_{ij}^{A}(\rho)=\frac {\theta_{i}+\theta_{j}}2$, or the logarithmic 
weight $\theta_{ij}^{L}(\rho)=\frac {\rho_{i}-\rho_{j}}{\log(\rho_{i})-\log(\rho_{j})}$.

\begin{rk}\label{rk-well}
The above results hold even when $G$ is not connected, in the following sense.  Consider the decomposition of
$G$ into disjoint connected components, and let $G=\cup_{j=1}^l G_{j}$.
Then, relative to each subgraph $(G_j,V^j,\omega^j)$,  
$\sum_{a_i\in V_j}\rho_{i}(t)=\sum_{a_i\in V_j}\rho_{i}^0$ and the properties (i)-(vi)
in Proposition \ref{well-dhs} also hold.
\end{rk}


\subsection{Spatial consistency for  Wasserstein-Hamiltonian flows}
When the graph $G$ is a lattice grid  on a domain $\mathcal M$ in $\mathbb R^d$,  
\eqref{dhs} can be viewed as a consistent spatial discretization of the  Wasserstein-Hamiltonian system \eqref{WassHam1}.
We show this next.

Let us consider a Hamiltonian in the density space
\begin{align*}
\mathcal H(\rho,S)&=\int_{\mathcal M}H(x,\nabla S(x))\rho(x)d{x}\\
&=\int_{\mathcal  M}\frac 12|\nabla S(x)|^2\rho(x)dx+\mathcal F(\rho),
\end{align*}
with the potential $\mathcal F(\rho)=\int_{\mathcal M}\mathbb V(x)\rho(x)dx+
\frac 12\int_{\mathcal M}\int_{\mathcal M} \mathbb W(x,y)\rho(x)\rho(y)dxdy+\beta I(\rho),$ and 
$I(\rho)=\int_{\mathcal M} |\nabla \log (\rho)|^2\rho dx$.
The corresponding  Wasserstein-Hamiltonian vector field is 
\begin{equation}\label{hpde}\begin{split}
\frac {\partial \rho}{\partial t}-\frac {\delta \mathcal H(\rho,S)}{\delta S}=0,\; \rho(0)=\rho^0,\\
\frac {\partial S}{\partial t}+\frac {\delta \mathcal H(\rho,S)}{\delta \rho}=0,\; S(0)=S^0.
\end{split}\end{equation}
We assume that for some $T^*>0$ there exists a unique smooth solution $(\rho,S)$ of \eqref{hpde}
for all $t\le T^*$. 
In the following, we  show that the semi-discretization \eqref{dhs} 
is consistent with \eqref{hpde} for all $t \leq T^*$. 

For simplicity, we consider the lattice graph $(G,V,\Omega)$, which is a cartesian product of $d$
one dimensional lattices: $G=G_1\times\cdots\times G_d$ with $G_k=(V_k,E_k)$, $k=1,\dots, d$.  Also, let
us assume that there is no interaction potential in \eqref{dhs}. 
Denote $\omega=\frac 1{h^2}$, let $i=(i_1,i_2,\cdots,i_d)$ represents a point $x(i)$ in 
$\mathbb R^d$ and let the set of neighbors of $i$ be indicated by $N(i)$:
$$N_{k}(i)=\{(i_1,\cdots,i_{k-1},j_k,i_{k+1},\cdots,i_d)\,\ :\,\ (i_k,j_k)\in E_k\}.$$
For the probability weights $\theta_{ij}(\rho)$ and $\widetilde \theta_{ij}(\rho)$ in \eqref{dhs},
we assume that 
\begin{align*}
\theta_{ij}(\rho)=\Theta(\rho_i,\rho_j), \quad \widetilde  \theta_{ij}(\rho)=\widetilde \Theta(\rho_i,\rho_j),
\end{align*}
where $\Theta$ and $\widetilde \Theta$ are symmetric $C^{1+\epsilon}$-continuous 
functions, $\epsilon>0$. 
In order to show the spatial consistency of \eqref{dhs}, we further
assume that 
\begin{align}\label{con-wei}
\frac {\partial \Theta(x,x)}{\partial x}= \frac 12, \; \Theta(x,x)= x.
\end{align} 

\begin{prop}\label{con-gen}
Assume that  $\theta$ and $\widetilde \theta$ satisfy \eqref{con-wei}.
Then, the semi-discretization \eqref{dhs} is a consistent finite difference scheme for 
the Hamiltonian PDE \eqref{hpde}.
\end{prop}
\begin{proof}
Let $\rho_i(t)=\rho(t,x(i))$, $S_i(t)=S(t,x(i))$ and $e_1, \dots, e_d$, be the standard unit vectors. 
The  lattice graph in the $e_k$ 
direction contains two points near $i$, i.e., $x(i)-e_kh$ and $x(i)+e_kh$, which we label   
$i^+$ and $i^{-}$ for short.  At first, assume that $\Theta$ and $\widetilde \Theta$ are $C^2$ continuous.
Then, by Taylor expansion at $i$ in the $e_a$ direction, we obtain
\begin{align*}
&\sum_{k}\frac 1{h^2}(S_{i}-S_{i^+})\theta_{ii^{+}}(\rho)+\sum_{k}\frac 1{h^2}(S_{i}-S_{i^{-}})\theta_{ii^{-}}(\rho)\\\nonumber 
&=\sum_{k}\frac 1{h^2}(-\frac {\partial S} {\partial x_k}(x(i),t)h+\frac 12\frac {\partial^2 S} {\partial x_k^2}(x(i),t)h^2+\mathcal O(h^3))(\theta_{ii}(\rho)+\frac {\partial \theta_{ii}(\rho)}{\partial \rho_i}\frac {\partial \rho_{i}}{\partial x_{k}}h+\mathcal O(h^2))\\
&\quad+ \sum_{k}\frac 1{h^2}(\frac {\partial S} {\partial x_k}(x(i),t)h+\frac 12\frac {\partial^2 S} {\partial x_k^2}(x(i),t)h^2+\mathcal O(h^3))(\theta_{ii}(\rho)-\frac {\partial \theta_{ii}(\rho)}{\partial \rho_{i}}\frac {\partial \rho_{i}}{\partial x_{k}}h+\mathcal O(h^2))\\
&=\sum_{k}(\frac {\partial^2 S} {\partial x_k^2}(x(i),t)\theta_{ii}(\rho)+2\frac {\partial S} {\partial x_k}(x(i),t)\frac {\partial \theta_{ii}(\rho)}{\partial \rho_{i}}\frac {\partial \rho_{i}}{\partial x_{k}})+\mathcal O(h^2).
\end{align*}
Similarly,  
\begin{align*}
&-\frac 12\sum_{k}\frac 1{h^2}(S_{i^+}-S_i)^2 \frac {\partial \theta_{ii^+}(\rho)}{\partial \rho_i}-
\frac 12\sum_{k}\frac 1{h^2}(S_{i^{-}}-S_i)^2 \frac {\partial \theta_{ii^{-}}(\rho)}{\partial \rho_i}\\
&\quad -\beta \sum_{k}\frac {1}{h^2} |\log(\rho_{i^+})-\log(\rho_i)|^2
\frac {\partial \widetilde \theta_{ii^{+}}(\rho)}{\partial \rho_i}-\beta \sum_{k}
\frac {1}{h^2} |\log(\rho_{i^-})-\log(\rho_i)|^2\frac {\partial \widetilde \theta_{ii^{-}}(\rho)}{\partial \rho_i}\\
&=-\frac 1{h^2}\sum_{k}(\frac {\partial S} {\partial x_k}(x(i),t)h+
\mathcal O(h^2))^2(\frac {\partial \theta_{ii}(\rho)}{\partial \rho_i}+\mathcal O(h))\\
&\quad-\beta \frac 1{h^2}\sum_{k}(\frac {\partial \log(\rho_{i})}{\partial x_k}+
\mathcal O(h^2))^2(\frac {\partial \widetilde  \theta_{ii}(\rho)}{\partial \rho_i}+\mathcal O(h))\\
&=-\sum_{k}|\frac {\partial S} {\partial x_k}(x(i),t)|^2\frac {\partial \theta_{ii}(\rho)}{\partial \rho_i}
-2\beta \frac 1{h^2}\sum_{k}|\frac {\partial \log(\rho_{i})}{\partial x_k}|^2
\frac {\partial \widetilde  \theta_{ii}(\rho)}{\partial \rho_i}+\mathcal O(h^2).
\end{align*}
Thus, if $\frac {\partial \theta_{ii}(\rho)}{\partial \rho_i}=\frac {\partial \widetilde \theta_{ii}(\rho)}{\partial \rho_i}= \frac 12, \widetilde \theta_{ii}(\rho)=\theta_{ii}(\rho)=\rho_i$, we have 
\begin{align*}
&\frac {d\rho(t,x(i))}{d t}-\sum_{k}\frac 1{h^2}(S_{i}-S_{i^+})\theta_{ii^{+}}(\rho)-\sum_{k}\frac 1{h^2}(S_{i}-S_{i^{-}})\theta_{ii^{-}}(\rho)\\
&=\frac {\partial \rho(t,x(i))}{\partial t}+\nabla_{x_k}\cdot(\nabla_{x_k} S(t,x(i))\rho(t,x(i)))+\mathcal O(h^2),\\
&\frac {dS(t,x(i))}{dt}+\frac 12\sum_{k}\frac 1{h^2}(S_{i^+}-S_i)^2 \frac {\partial \theta_{ii^+}(\rho)}{\partial \rho_i}+\frac 12\sum_{k}\frac 1{h^2}(S_{i^{-}}-S_i)^2 \frac {\partial \theta_{ii^{-}}(\rho)}{\partial \rho_i}\\
&+\beta \sum_{k}\frac {1}{h^2} |\log(\rho_{i^+})-\log(\rho_i)|^2\frac {\partial \widetilde \theta_{ii^{+}}(\rho)}{\partial \rho_i}+\beta \sum_{k}\frac {1}{h^2} |\log(\rho_{i^-})-\log(\rho_i)|^2\frac {\partial \widetilde \theta_{ii^{-}}(\rho)}{\partial \rho_i}\\
&+V(x(i))=\frac {\partial S(t,x(i))}{\partial t}+\frac 12 |\nabla_{x_k} S(t,x(i))|^2+\beta\frac {\partial I}{\partial \rho_i}(\rho(t,x(i)))+V(x(i))+\mathcal O(h^2),
\end{align*}
which implies that \eqref{dhs} is a second order consistent semi-discretization scheme. By interpolation arguments, 
we complete the proof for the case that $\Theta$ and $\widetilde \Theta$ are $C^{1+\epsilon}$-continuous.
\end{proof}

As we show next, even if $\Theta$ and $\widetilde \Theta$ are not sufficiently regular, 
spatial consistency still holds as long as \eqref{con-wei} holds.
For example, one can take $\theta$ as the upwind weight,
$\theta_{ij}^{U}(\rho)=\Theta^U(\rho_i,\rho_j):=\rho_{i}$, if $S_{j}>S_{i}$, 
$\widetilde \theta$ satisfies \eqref{con-wei} and $\widetilde \Theta$ is symmetric $C^{1+\epsilon}$-continuous  .
\begin{prop}\label{con-up}
Assume that $\theta=\theta^U$, and that $\widetilde \theta$ satisfies \eqref{con-wei}. Then 
\eqref{dhs} is a consistent spatial discretization of \eqref{hpde}.
\end{prop}
\begin{proof}
We use the same notations as in the proof of Proposition \ref{con-gen}. For simplicity, we assume that $S(t,x(i)+e_kh)\le S(t,x(i)) \le S(t,x(i)-e_kh)$ and  
that  $\widetilde \Theta$ is $C^2$ continuous.  
Similarly, we can show the same results for other possible configurations. 
By Taylor expansion, we obtain 
\begin{align*}
&\sum_{k}\frac 1{h^2}(S_{i}-S_{i^+})\theta_{ii^{+}}(\rho)+\sum_{k}\frac 1{h^2}(S_{i}-S_{i^{-}})\theta_{ii^{-}}(\rho)\\\nonumber 
&=\sum_{k}\frac 1{h^2}(S(t,x(i))-S(t,x(i)+e_kh))\rho_{i^+}+\sum_{k}\frac 1{h^2}(S(t,x(i))-S(t,x(i)-e_kh))\rho_{i}\\
&=\sum_{k}\frac 1{h^2}(-\frac {\partial S} {\partial x_k}(x(i),t)h+\frac 12\frac {\partial^2 S} {\partial x_k^2}(x(i),t)h^2+\mathcal O(h^3))\rho_{i^+}\\
&\quad -\sum_{k}\frac 1{h^2}(\frac {\partial S} {\partial x_k}(x(i),t)h+\frac 12\frac {\partial^2 S} {\partial x_k^2}(x(i),t)h^2+\mathcal O(h^3))\rho_{i}\\
&=\frac 1h\frac {\partial S} {\partial x_k}(x(i),t)(\rho_{i^+}-\rho_i)+\frac 12\frac {\partial^2 S} {\partial x_k^2}(x(i),t)(\rho_{i^+}+\rho_i)+\mathcal O(h)\\
&=\nabla_{x_k}\cdot(\rho(t,x(i))\nabla_{x_k} S(t,x(i)))+\mathcal O(h),
\end{align*}
and 
\begin{align*}
&-\frac 12\sum_{k}\frac 1{h^2}(S_{i^+}-S_i)^2 \frac {\partial \theta_{ii^+}(\rho)}{\partial \rho_i}+-\frac 12\sum_{k}\frac 1{h^2}(S_{i^{-}}-S_i)^2 \frac {\partial \theta_{ii^{-}}(\rho)}{\partial \rho_i}\\
&\quad -\beta \sum_{k}\frac {1}{h^2} |\log(\rho_{i^+})-\log(\rho_i)|^2\frac {\partial \widetilde \theta_{ii^{+}}(\rho)}{\partial \rho_i}-\beta \sum_{k}\frac {1}{h^2} |\log(\rho_{i^-})-\log(\rho_i)|^2\frac {\partial \widetilde \theta_{ii^{-}}(\rho)}{\partial \rho_i}\\
&=-\frac 12\sum_{k}\frac 1{h^2}(S_{i^+}-S_i)^2
 -\beta \sum_{k}\frac {1}{h^2} |\log(\rho_{i^+})-\log(\rho_i)|^2\frac {\partial \widetilde \theta_{ii^{+}}(\rho)}{\partial \rho_i}\\
&\quad-\beta \sum_{k}\frac {1}{h^2} |\log(\rho_{i^-})-\log(\rho_i)|^2\frac {\partial \widetilde \theta_{ii^{-}}(\rho)}{\partial \rho_i}\\
&=-\frac 1{2h^2}\sum_{k}(\frac {\partial S} {\partial x_k}(x(i),t)h+\mathcal O(h^2))^2\\
&\quad-2\beta \frac 1{h^2}\sum_{k}(\frac {\partial \log(\rho_{i})}{\partial x_k}+\mathcal O(h^2))^2(\frac {\partial \widetilde  \theta_{ii}(\rho)}{\partial \rho_i}+\mathcal O(h))\\
&=\frac 12\sum_{k}|\frac {\partial S} {\partial x_k}(x(i),t)|^2-2\beta \frac 1{h^2}\sum_{k}|\frac {\partial \log(\rho_{i})}{\partial x_k}|^2\frac {\partial \widetilde  \theta_{ii}(\rho)}{\partial   \rho_i}+\mathcal O(h).
\end{align*}
Therefore, combining with the above estimate and \eqref{hpde}, we have that
\begin{align*}
&\frac {d\rho(t,x(i))}{d t}-\sum_{a}\frac 1{h^2}(S_{i}-S_{i^+})\theta_{ii^{+}}(\rho)-\sum_{k}\frac 1{h^2}(S_{i}-S_{i^{-}})\theta_{ii^{-}}(\rho)\\
&=\frac {\partial \rho(t,x(i))}{\partial t}+\sum_{k}\nabla_{x_k}\cdot(\nabla_{x_k} S(t,x(i))\rho(t,x(i)))+\mathcal O(h)=\mathcal O(h),\\
&\frac {dS(t,x(i))}{dt}+\frac 12\sum_{k}\frac 1{h^2}(S_{i^+}-S_i)^2 \frac {\partial \theta_{ii^+}(\rho)}{\partial \rho_i}+\frac 12\sum_{k}\frac 1{h^2}(S_{i^{-}}-S_i)^2 \frac {\partial \theta_{ii^{-}}(\rho)}{\partial \rho_i}\\
& +\beta \sum_{k}\frac {1}{h^2} |\log(\rho_{i^+})-\log(\rho_i)|^2\frac {\partial \widetilde \theta_{ii^{+}}(\rho)}{\partial \rho_i}+\beta \sum_{k}\frac {1}{h} |\log(\rho_{i^-})-\log(\rho_i)|^2\frac {\partial \widetilde \theta_{ii^{-}}(\rho)}{\partial \rho_i}\\
&+V(x(i))=\frac {\partial S(t,x(i))}{\partial t}+\sum_{k}\frac 12 |\nabla_{x_k} S(t,x(i))|^2+\beta\frac {\partial I}{\partial \rho_i}(\rho(t,x(i)))+V(x(i))+\mathcal O(h).
\end{align*}
\end{proof}
\begin{rk}
In \eqref{dhs}, take 
$\beta=\frac {\hbar^2}8>0$, a fixed number. By introducing the discrete Madelung transformation  $u(t)=(u_j(t))_{j=1}^N=(\sqrt{\rho_j(t)}e^{\mathbf i\frac {S_j(t)}{\hbar}})_{j=1}^N,$ 
\eqref{dhs} can be viewed as a nonlinear spatial approximation of the nonlinear 
Schr\"odinger equation and can be rewritten as 
\begin{align*}
\hbar\mathbf i \frac {du_j}{dt}=-\frac {\hbar^2}2(\Delta_Gu)_{j}+u_j\mathbb V_j+u_j\sum_{l=1}^N\mathbb W_{jl}|u_l|^2,
\end{align*}
where the Laplacian on the graph is defined by 
\begin{align*}
(\Delta_Gu)_{j}&:=-u_j\Big(\frac 1{|u_j|^2}\big[\sum_{l\in N(j)}\omega_{jl}(Im(\log(u_j))-Im(\log(u_l))) \theta_{jl}\\
&\quad +\sum_{l\in N(j)}\widetilde \omega_{jl}(Re(\log(u_j))-Re(\log(u_l)))\widetilde \theta_{jl}\big]\\
&\quad +\sum_{l\in N(j)}\omega_{jl}|Im(\log(u_j)-\log(u_l))|^2\frac {\partial \theta_{jl}}{\partial \rho_j}\\
&\quad +\sum_{l\in N(j)}\widetilde \omega_{jl}|Re(\log(u_j)-\log(u_l))|^2\frac {\partial \widetilde \theta_{jl}}{\partial \rho_j}\Big).
\end{align*}
\end{rk}

\section{Lower bound estimate of the density}\label{sec-low}

In this section, we give an explicit lower bound for the density function in \eqref{dhs} with the logarithmic weight $\widetilde\theta_{ij}(\rho)=\Theta^{L}(\rho_i,\rho_j):=\frac {\rho_{i}-\rho_{j}}{\log(\rho_{i})-\log(\rho_{j})}$. 
We take two basic graphs as structures to illustrate the derivation of the lower bound. 
With appropriate modifications, one can obtain the lower bounds for more general graphs
and different probability weights $\widetilde \theta.$

\subsection{Lower bound for periodic nearest neighbor structure}
This is the classic nearest neighbor graph, with periodic boundary conditions.
Our goal is to analyze the properties of the extreme point of the Fisher information \eqref{DiscFisher} in
the present case,
\begin{equation}\label{FishInfo2}
I(\rho)=\sum_{i=1}^N\widetilde \omega_{i,{i+1}}(\log(\rho_{i})-\log(\rho_{i+1}))(\rho_{i}-\rho_{i+1}),
\end{equation}
on the set $\mathcal P_o(G)$.  Denote the tangent space at $\rho\in \mathcal P_o(G)$ by 
$T_{\rho}\mathcal P_o(G)=\{(\sigma)_{i=1}^N\in \mathbb R^N | \sum_{i=1}^N \sigma_{i}=0\}$.  
  
\begin{lm}\label{con-per}
The function $I(\rho)$ in \eqref{FishInfo2}
is strictly convex on $\mathcal P_o(G)$ and achieves its unique minimum at the uniform distribution. 
\end{lm}

\begin{proof}
The convexity of $I$ can be obtained by directly calculating the Hessian matrix and proving 
$$\min_{\sigma \in T_{\rho}\mathcal P_o(G)}\{\sigma^T\text{Hess}( I(\rho))\sigma | \sigma^T\sigma=1\}>0.$$
Direct calculations yield that 
$$ \frac {\partial^2}{\partial \rho_i\rho_j} I(\rho)=\left\{
\begin{array}{rcl}
 &\widetilde\omega_{i,{i+1}}\frac 1{\rho_i^2}(\rho_i+\rho_{i+1})+
\widetilde\omega_{i,i-1}\frac 1{\rho_i^2}(\rho_{i}+\rho_{i-1})&\quad \text{for}\quad {j=i};\\
 &-\widetilde\omega_{i,{i+1}}\frac 1{\rho_i\rho_{i+1}} (\rho_{i}+\rho_{i+1}) &\quad  \text{for} \quad {j=i+1};\\
 &-\widetilde\omega_{i,{i-1}}\frac 1{\rho_i\rho_{i-1}} (\rho_{i}+\rho_{i-1})  &\quad  \text{for} \quad {j=i-1};\\
 &0\qquad\qquad  &\quad  \text{otherwise}.
\end{array} \right. $$
Thus we obtain 
\begin{align*}
\sigma^T\text{Hess} I(\rho)\sigma
&=\sum_{i=1}^N(\widetilde\omega_{i,{i+1}}\frac 1{\rho_i^2}(\rho_i+\rho_{i+1})+
\widetilde\omega_{i,i-1}\frac 1{\rho_i^2}(\rho_{i}+\rho_{i-1}))\sigma_i^2\\
&\quad+\sum_{i=1}^N(\widetilde\omega_{i,{i+1}}\frac 1{\rho_i\rho_{i+1}} (\rho_{i}+\rho_{i+1})\sigma_{i}\sigma_{i+1} 
+\widetilde\omega_{i,i-1}\frac 1{\rho_i\rho_{i-1}} (\rho_{i}+\rho_{i-1})\sigma_{i}\sigma_{i-1})\\
&=\sum_{i=1}^N\widetilde\omega_{i,{i+1}}(\rho_i+\rho_{i+1})(\frac {\sigma_i}{\rho_i}-\frac {\sigma_{i+1}}{\rho_{i+1}})^2\ge 0,
\end{align*}
which implies the semi-positvity of $\text{Hess}( I(\rho))$.
To show strict convexity, assume that there exists a unit vector $\sigma^*$ such that ${\sigma^*}^T\text{Hess} I(\rho)\sigma^*=0$. 
Then we have $\frac {\sigma_{1}}{\rho_1}=\frac {\sigma_{i}}{\rho_i}$ for $i=2,\cdots,N$. 
Since $\sigma \in T_{\rho}\mathcal P_o(G)$, then $\sum_{i=1}^N\sigma_{i}=\sigma_1(1+\sum_{i=2}^N\frac {\rho_{i}}{\rho_{1}})=0$.
As $\rho\in \mathcal P_o(G)$, we conclude that $\sigma_{i}=0$ for all $i$, which contradicts that ${\sigma^*}^T\sigma^*=1$. 
Strict convexity implies that there is a unique minimum point on $\mathcal P_o(G)$. 
By using the Lagrange multiplier technique to find the minimum of $I(\rho)$ under the constraint 
$\sum_{i=1}^N{\rho_i}=1$ and taking the first derivative with respect to $\rho$, 
we obtain that the extreme point satisfies
\begin{align*}
\widetilde\omega_{i,{i+1}} \phi(\frac {\rho_{i+1}}{\rho_i})+\widetilde\omega_{i-1,{i}}
\phi(\frac {\rho_{i-1}}{\rho_i})=\lambda, \quad \text{for}\quad  i\le N,
\end{align*} 
where $\phi(t)=1-t-\log(t), t\in (0,\infty)$.
It is not difficult to verify that $\phi$ is strictly decreasing, convex, and $\phi(1)=0$.
Then when $\lambda =0$, $\rho_i=\frac 1N$, the extreme point $\rho_i=\frac 1N$ is the unique minimum point such that   $I(\rho)=0$.
\end{proof}

Due to convexity of $I(\rho)$, for any $C>0$ there exists $c<1/N$, such that 
$\inf_{0<\min_i(\rho_i)\le c}I(\rho)\ge C$. On the other hand, we also know that the exact solution preserves 
energy, which means that 
$\rho(t)\in B=\{\rho\in \mathcal P_o(G) \; | \; \beta I(\rho)\le 
\mathcal H_0-\min_{\rho}(\mathcal V(\rho)+\mathcal W(\rho))\},$ 
where $\min_{\rho}(\mathcal V(\rho)+\mathcal W(\rho))<\infty$. 
Denote $M:= \mathcal H_0-\min_{\rho}(\mathcal V(\rho)+\mathcal W(\rho)).$
Thus, if we can find an upper bound $c$ such that $I(\rho)\ge \frac { M}{\beta}$, then $c$ will
be a lower bound for the exact solution $\rho(t), t\ge 0$.
Since 
$$I(\rho)\ge \min_{i\le N-1}\widetilde\omega_{i{i+1}}\sum_{i=1}^N(\log(\rho_i)-\log(\rho_{i+1}))(\rho_i-\rho_j),$$
the condition that $\sum_{i=1}^N(\log(\rho_i)-\log(\rho_{i+1}))(\rho_i-\rho_j))\ge 
\frac 1{ \min_{i\le N-1}\widetilde\omega_{i{i+1}}} \frac {M}{\beta} $ ensures $I(\rho)\ge \frac { M}{\beta}$. 
The following result gives the anticipated lower bound, and its proof is given in the Appendix at the end, where
we assume that $\widetilde \omega_{i,{i+1}}=1$ for simplicity. 

%
\begin{prop}\label{low-per}
Let $\min_i(\rho_i^0)<\frac 1N$. Then it holds that 
$$\sup\limits_{t\ge0}\min\limits_{i\le N}\rho_i(t)\ge  \min( \frac 12\min_i \rho_i^0, 
\frac 1{1+N\exp(\frac { M(N-1)([\frac {N-1}2]+1)}{\beta})}).$$
\end{prop}

\begin{proof}
See the Appendix.
\end{proof}

\subsection{Lower bound for aperiodic structure}
Here we consider the case of an aperiodic graph (e.g., as when we have Neumann boundary conditions), and look for
the extreme points of $I(\rho)$ under the constraint $\sum_{i=1}^N\rho_{i}=1$.
%
We denote the boundary point set by $V_B$, i.e., if $a \in V_B$, then there exists only one edge connecting with other points.  
The Fisher information term now is 
\begin{align}\label{FishInfoAp}
I(\rho)=\sum_{i=1}^{N-1}\widetilde \omega_{i,i+1}(\log(\rho_{i})-\log(\rho_{i+1}))(\rho_{i}-\rho_{i+1}).
\end{align}
Similarly to Lemma \ref{con-per}, we have strict convexity of $I(\rho)$.

\begin{lm}
$I(\rho)$ in \eqref{FishInfoAp}
is strictly convex on $\mathcal P_o(G)$ and achieves its unique minimum at the uniform distribution. 
\end{lm}

The proof of the following lower bound estimate is also given in the Appendix, 
where for simplicity we assume that $\widetilde \omega_{ii+1}=1$.

\begin{prop}\label{low-new}
Let $\min_i(\rho_i^0)<\frac 1N$. Assume that $\kappa\le N-1$ is the number of nodes in $V_B$, $d_{max}$ is the 
largest distance\footnote{The distance $d_{ij}$ between two nodes $a_i$ and $a_j$ is the smallest number
of edges connecting $a_i$ and $a_j$.}
between two nodes in $V_B$. Then it holds that 
$$\sup_{t}\min_{i}\rho_{i}(t)\ge 
\min\Big(\frac 12 \min_i(\rho_i(0)),\frac {1}{1+\kappa(d_{max}-1)\exp(2\frac {M(d_{max}-1)(N-1)}{\beta})}\Big),$$
where $\kappa$ is the number of nodes in $V_B$.
\end{prop}
\begin{proof}
See the Appendix.
\end{proof}


\section{Time discretization of  Wasserstein-Hamiltonian systems on graph}\label{sec-disc}
Our purpose in this section is to look at the full discretization of Wasserstein-Hamiltonian systems.  In particular,
we discuss the time discretization of the (regularized) spatial discretizations \eqref{dhs} and \eqref{hpde} and our main goal
is to devise a symplectic discretization of the  Wasserstein-Hamiltonian flow \eqref{dhs} with
$\beta>0$.
Then, we will discuss general regularization strategies for \eqref{hpde}.

Presently, 
the discrete Lagrangian functional is 
$$L(\rho,\dot \rho)=\frac 12\<\nabla_G S,\nabla_G S\>_{\theta_{\rho}}h-\mathcal F(\rho)h-\beta I(\rho)h$$
with the constraint $\frac {d\rho}{d t}+div_G^{\theta}(\rho \nabla_G S)=0$,
and $$I(\rho)=\frac 12\sum_{i=1}^N\sum_{j\in N(i)}\widetilde \omega_{ij}|\log(\rho_i)-\log(\rho_j)|^2\widetilde \theta_{ij}(\rho).$$

We assume that $c_0 \le \omega_{ij}\le C_0, c_0\le \widetilde \omega_{ij}\le C_0$, for some positive numbers $c_0,C_0$,
and that $\max_{i}\mathbb V_i+\max_{ij}\mathbb W_{ij}\le M_0$. 
For simplicity, in this part we restrict consideration to  
$\theta_{ij}(\rho)=\theta_{ij}^{A}(\rho)=\frac {\theta_{i}+\theta_{j}}2$, and
$\widetilde \theta_{ij}(\rho)=\theta_{ij}^{L}(\rho)=\frac {\rho_{i}-\rho_{j}}{\log(\rho_{i})-\log(\rho_{j})}$.
Denote the maximum numbers of edges connecting to a node with $E_{max}$, and let
$c$ be the uniform lower bound of  $\rho$ derived in Section \ref{sec-low}. Then, 
the uniform  upper bound estimate of $|{S_i-S_{j}}|$ can be obtained in the following way.

Recall $\mathcal H(\rho,S)=K(S,\rho)+\mathcal F(\rho),$ where 
$ K(S,\rho):=
\frac 12 \<\nabla_G S, \nabla_G S\>_{\theta(\rho)}$ and $\mathcal F(\rho):=\beta I(\rho)+\mathcal V(\rho)+\mathcal W(\rho).$
Due to the conservation of $\mathcal H$, we have 
\begin{align*}
K(S,\rho)+\beta I(\rho)& :=\frac 14 \sum_{i}\sum_{j\in N(i)}\omega_{ij} {|S_{i}(t)-S_{j}(t)|^2}\theta_{ij}(\rho(t))
\\
&\quad +\beta\frac 12\sum_{i,j=1}^N\widetilde \omega_{ij}(\log(\rho_{i}(t))-\log(\rho_{j}(t)))^2\widetilde \theta_{ij}(\rho(t))\\
&\le \mathcal H_0-\min_{i}(\mathbb V_i+\sum_{j=1}^N\mathbb W_{ij}\rho_{j})=:M.
\end{align*}
Then we get 
\begin{align*}
\max_{i}|{S_i-S_{j}}|^2 &\le\frac {2M}{\min_{i,j}\omega_{ij}\min_{ij}\theta_{ij}(\rho(t))}\le  
\frac {2M}{c \min_{i,j}\omega_{ij}},\\
\max_{i}|\log(\rho_i)-\log(\rho_{j})|^2 &
\le\frac {M}{\min_{i,j}\widetilde \omega_{ij}\min_{ij}\theta_{ij}(\rho(t))}\le  \frac {M}{c \min_{i,j}\widetilde \omega_{ij}},
\end{align*}
where $c\ge \min\Big(\frac 12 \min_i\rho_i(0),
\frac {1}{1+\kappa(d_{max}-1)\exp(2\frac {M(d_{max}-1)(N-1)}{\min_{i,j}\widetilde \omega_{ij}\beta})}\Big)$ and $\kappa$ is the number of nodes in $V_B$.
Since $x-y\le \log(x)-\log(y)$ for $0<y\le x<1$, we also obtain 
\begin{align*}
\max_{i}|\rho_i-\rho_{j}|^2 &\le\frac {M}{\min_{i,j}\widetilde \omega_{ij}\min_{ij}\theta_{ij}(\rho(t))}\le  \frac {M}{c \min_{i,j}\widetilde \omega_{ij}}.
\end{align*}
The Lipschitz constant of $\mathcal H$, $Lip(\mathcal H)$, satisfies 
\begin{align*}
Lip(\mathcal H)\le \max_{i\le N}\Big(\big|\sum_{j\in N(i)}(S_i-S_j)\omega_{ij}\theta_{ij}(\rho)
\big|, \big|\frac 12\sum_{j\in N(i)}\omega_{ij}(S_i-S_j)^2\frac {\partial \theta_{ij}}{\partial \rho_{i}}+\beta \frac {\partial I} {\partial \rho_i}(\rho) \big|+M_0\Big).
\end{align*}

Then on the set $B=\{(S,\rho)| K(S,\rho)+I(\rho)\le M\}$, we have 
\begin{align*}
\Big\|\frac {\partial \mathcal H}{\partial S}\Big\|_{l^{\infty}}&\le E_{max}C_0\sqrt{\frac {2M}{cc_0}},\;
\Big\|\frac {\partial \mathcal H}{\partial \rho}\Big\|_{l^{\infty}}\le E_{max}C_0\Big(\frac 12\frac {M}{cc_0}+\beta  \sqrt{\frac M{cc_0}}+\beta \frac 1c+M_0\Big),\\
\Big\|\frac {\partial^2 \mathcal H}{\partial S^2}\Big\|_{l^{\infty}}&\le C_0,\;
\Big\|\frac {\partial^2 \mathcal H}{\partial \rho\partial S}\Big\|_{l^{\infty}}\le \frac 1{\sqrt {2}} C_0 \sqrt{\frac {M}{cc_0}},\;
\Big\|\frac {\partial^2 \mathcal H}{\partial \rho^2}\Big\|_{l^{\infty}}\le \beta C_0\big(\frac{1}{c}+\frac{1}{c^2}+M_0\big).
\end{align*}
By recursive calculations, we further get 
$\frac {\partial^n \mathcal H}{\partial \rho^n}\le \beta C_0((n-2)!(\frac 1c)^{{n-1}}+(n-1)!(\frac 1 c)^{n})$ 
for $n\ge 3$ and other partial derivatives bounded by $\frac {C_0} 2$ for $n=3$ and $0$ for $n\ge 4$.

\subsection{Symplectic methods}
Based on the positivity of the probability density in \eqref{dhs}, the constraint on $\rho$ can be rewritten 
as $S(t)=(-\Delta_{\rho(t)}^{\theta})^{\dag}\dot \rho(t)$, where $(-\Delta_{\rho(t)}^{\theta})^{\dag}$ is the 
pseudo-inverse of $-div_G^{\theta}(\rho\nabla_G(\cdot))$.
Thus we have the following equivalent forms
\begin{align*}
\mathcal L(\rho,\nabla_G S)&=\frac 12\<\nabla_G S,\nabla_G S\>_{\theta_{\rho}}-\mathcal F(\rho)= \frac 12 \<S, \Delta_{\rho(t)}^{\theta} S\>h-\mathcal F(\rho)\\
&= \frac 12 \<\nabla_G ((-\Delta_{\rho(t)}^{\theta})^{\dag}\dot \rho(t)),\nabla_G ((-\Delta_{\rho(t)}^{\theta})^{\dag}\dot \rho(t))\>_{\theta_{\rho}}-\mathcal F(\rho)\\
&= \frac 12 \<(-\Delta_{\rho(t)}^{\theta})^{\dag}\dot \rho(t)), (-\Delta_{\rho(t)}^{\theta})(-\Delta_{\rho(t)}^{\theta})^{\dag}\dot \rho(t)\>-\mathcal F(\rho)=: L(\rho,\dot \rho),
\end{align*}
where $\mathcal F(\rho)=\beta I(\rho)+\mathcal V(\rho)+\mathcal W(\rho), \beta>0.$

Consider the action integral $\mathcal S(\rho)=\int_{t_0}^{t_1}L(\rho(t),\dot \rho(t))dt$ 
among all curves $\rho(t)$ connecting two given probability densities 
$\rho(t_0)=\rho^0$ and $\rho(t_1)=\rho^1$, and 
let us consider the approximation of the action integral between $0$ and $T$, connecting
$\rho(0)$ and $\rho(T)$:
\begin{align*}
\mathcal S_\tau(\{\rho^n\}_{n=0}^N)=\sum_{n=0}^{N-1}L_\tau(\rho^n,\rho^{n+1}),
\end{align*} 
where $L_\tau (\rho^n,\rho^{n+1})$ is an approximation of 
$\int_{t_n}^{t_{n+1}}L(\rho(s),\dot \rho(s))ds$ with given 
$T=t_N$ and $\tau=t_{n+1}-t_n$. 
Then, letting $\frac {\partial \mathcal S_\tau }{\partial \rho^n}=0$, 
for $n=1,\cdots,N-1$, we get the discrete Euler-Lagrange equation
\begin{align*}
\frac {\partial  L_\tau}{\partial x}(\rho^n,\rho^{n+1})+\frac {\partial   L_\tau}{\partial y}(\rho^{n-1},\rho^{n})=0,
\end{align*}
where  $\frac {\partial  L_\tau}{\partial x}$ and $\frac {\partial  L_\tau}{\partial y}$ refer to the partial derivatives 
with respect  to the first and second argument. 

By introducing the discrete momenta via the discrete 
Legendre transformation
$p^n=-\frac {\partial L_\tau}{\partial x}(\rho^n,\rho^{n+1})$,
we can get $d\mathcal S_\tau=p^Nd\rho^N-p^0d\rho^0$. $\mathcal S_\tau$ is also called symplecticity generating function.
This implies the symplecticity of the map $(p^0,\rho^0)\to (p^N,\rho^N)$ (see e.g. \cite[Chapter VI]{HLW06}).
Indeed, we get 
\begin{align*}
p^n=-\frac {\partial L_\tau}{\partial x}(\rho^n,\rho^{n+1}),\;
p^{n+1}=\frac {\partial L_\tau}{\partial y}(\rho^n,\rho^{n+1}).
\end{align*}

Let us consider the first time step approximation.
Assume that we use some numerical integration formula, and get
$L_{\tau}(\rho^0,\rho^1)=\tau \sum_{i=1}^sb_iL(u(c_i\tau),\dot u(c_i\tau)),$
where $0\le c_1<\dots < c_s\le 1$ and $u(t)$ is the collocation polynomial of degree $s$ with $u(0)=\rho^0$ and $u(\tau)=\rho^1$.
Then we can rewrite the above approximation as 
\begin{align*}
L_{\tau}(\rho^0,\rho^1)=\tau \sum_{i=1}^sb_iL(\Phi^i,\dot \Phi^i),\\
\Phi^i= \rho^0+h\sum_{j=1}^s a_{ij}\dot \Phi^j,
\end{align*} 
subject to the constraint $\rho_1=\rho_0+h\sum_{i=1}^s b_i \dot \Phi^i$. 
We assume that all the $b_i$ are non-zero and that their sum equals 1.
By the Lagrange multiplier method, the extremum point satisfies 
\begin{equation}\label{sym}\begin{split}
S^1&= S^0-\tau \sum_{i=1}^s b_i \frac {\partial \mathcal H(\Xi^i,\Phi^i)}{\partial \rho},\\\nonumber 
\rho^1&= \rho^0+\tau \sum_{i=1}^s b_i \frac {\partial \mathcal H(\Xi^i,\Phi^i)}{\partial S},\\\nonumber 
\Xi^i&= S^0-\tau \sum_{j=1}^s \widetilde a_{ij} \frac {\partial \mathcal H(\Xi^j,\Phi^j)}{\partial \rho},
\\\nonumber 
\Phi^i&= \rho^0+\tau \sum_{ j=1}^s a_{ij} \frac {\partial \mathcal H(\Xi^j,\Phi^j)}{\partial S}
\end{split}\end{equation} 
where the coefficients satisfy the condition
$\widetilde a_{ij}b_i+a_{ji}b_{j}=b_ib_j$, of partitioned
Runge Kutta symplectic methods for the Wasserstein-Hamiltonian system \eqref{dhs}. 

\begin{ex}\label{Fis-eul}
Symplectic Euler method ($\widetilde a_{ij}=1$, $a_{ji}=0$, $b_i=b_j=1$, $s=1$)
\begin{align*}
\rho^{n+1}_i&= \rho^n_i +\frac {\partial \mathcal H(S^{n+1},\rho^n)}{\partial S}\tau,\\
&=\rho^n_i-\sum_{j\in N(i)}\omega_{ij}(S_j^{n+1}-S_i^{n+1})\theta_{ij}(\rho^n)\tau\\
S^{n+1}_i&= S^n_i-\frac {\partial \mathcal H(S^{n+1},\rho^n)}{\partial \rho}\tau ,\\
&=  S^n_i-\frac 12\sum_{j\in N(i)}\omega_{ij}(S_i^{n+1}-S_j^{n+1})^2 \frac {\partial \theta_{ij}(\rho^n)}{\partial \rho_i}\tau- \frac {\partial \mathcal F(\rho^n)}{\partial \rho_i}\tau,
\end{align*}
where $\mathcal F(\rho):=\beta I(\rho)+\mathcal V(\rho)+\mathcal W(\rho).$
\qed
\end{ex}

In the following, we focus on the case of symplectic Runge--Kutta methods, i.e., 
$\widetilde a_{ij}=a_{ij}$. With minor modifications, all results hold for the  
partitioned Runge--Kutta symplectic methods. 
\begin{tm}\label{sym-pro}
Assume that $G=(V,E,\Omega)$ is a connected weighted graph and that $\min_{i\le N}\rho^0_i>0$.
Then the symplectic Runge--Kutta scheme \eqref{sym} enjoys the following properties.
\begin{enumerate}[label=(\roman*)]
\item It preserves mass:
$$\sum_{i=1}^N\rho_i^n= \sum_{i=1}^N\rho_i^0.$$
\item It preserves symplectic structure: $d\rho^n \wedge dS^n =d\rho^0 \wedge dS^0$.
\item Assuming that the scheme is symmetric, then it is time reversible:  
if $(\rho^n, S^n)$ is the solution of the full discretization, then $(\rho^{-n}, -S^{-n})$ is also the 
solution of  the full discretization.
\item It is time transverse (gauge) invariant: if $\mathbb V^{\mathbb \alpha}=\mathbb V-\mathbb \alpha$, 
then $S^{\mathbb \alpha}=S+\mathbb \alpha t$ is the solution of the scheme with linear potential $\mathbb V^{\mathbb \alpha}$.
\item A time invariant $\rho^*\in \mathcal P_o(G)$ and ${S^*}^n=-v n\tau $ form an interior stationary 
solution of the symplectic scheme if and only if $(\rho^*,S^*)$ is the critical point of 
$\mathcal H(\rho,S)$ and $v=\mathcal H(\rho^*,S^*)+\frac 12 \sum_{i=1}^N\sum_{j=1}^N \mathbb W_{ij}\rho^*_i\rho^*_j$. 
\item When $\frac M\beta$ is small enough, the scheme almost preserves the Hamiltonian up to time 
$T=\mathcal O(\tau^{-r})$:
$$\mathcal H(S^n,\rho^n)=\mathcal H(S^0,\rho^0)+\mathcal O(\tau^r),$$
where $r$ is the order of the symplectic numerical scheme. 
\end{enumerate}
\end{tm}
\begin{proof}
Property (i) holds since this is a linear constraint.
Property (ii) can be verified by using the symplecticity condition $a_{ij}b_i+a_{ji}b_{j}=b_ib_j$. 
As far as (iii), since the exact flow of the original system $\Phi(y)=\Phi(S,\rho)$ is $g$-reversible, 
i.e., $g\circ \Phi=\Phi^{-1}\circ g$, with $g(S,\rho)=(-S,\rho)$, 
then since the one-step method $\Phi_{\tau}$ is symmetric, i.e, $\Phi_{\tau}\circ\Phi_{-\tau}=I$, then 
$\Phi_\tau$ is $g$-reversible, i.e., $g\circ \Phi_{\tau}=\Phi^{-1}_{\tau}\circ g$, and (iii) holds. 
Property (iv) holds because $K(\rho,S)$ is an even function of $S$ and the potential is linear.
To show Property (v), we only need to show that $\rho^*$ satisfies the Karush-Kuhn-Tucker 
conditions of optimality for minimization of
$\min_{\rho\in \mathcal P_o(G)}\mathcal H(\rho)=\min_{\rho\in \mathcal P_o(G)}(\beta I(\rho)+\mathcal V(\rho)+\mathcal W(\rho))$, 
which is done using the Lagrange multiplier method.

We next focus on the proof of (vi).
Rewrite the $r$-th order Runge--Kutta scheme as
\begin{align*}
y^1&= y^0+\tau \sum_{i=1}^s b_i f(\tilde y^i),\\
\tilde y^i&= y^0+\tau \sum_{j=1}^s a_{ij} f(\tilde y^j).
\end{align*} 
Assume that 
$y_0\in B=\{\rho\in \mathcal P_o(G) \; | \; \beta I(\rho)\le \mathcal H_0-\min_{\rho}(\mathcal V(\rho)+\mathcal W(\rho))\}$
and let $K$ be the smallest number such that $y
^{K+1}\notin B$ and  for some $j\le N$, $y^{K+1}_{N+j}=\min_{i=1}^N|y^{K+1}_{N+i}|=\alpha c, 0<\alpha<1$. 
By Taylor expansion, using recursion, we have 
\begin{align*}
|y_{N+i}(t_{K+1})-y^{K+1}_{N+i}|&\le |y_{N+i}(t_{K})-y^{K}_{N+i}| +C_{r,M,c_0,C_0}(1+\beta)(\frac {1}{c^{2r+1}}+1)\tau^{r+1}\\
&\le K\tau C_{r,M,c_0,C_0}(1+\beta)(\frac {1}{c^{2r+1}}+1)\tau^{r},
\end{align*}
which implies that for $i=1,\cdots,N$,
\begin{align*}
y_{N+j}(t_{K+1})\le  y^{K+1}_{N+j}-K\tau C_{r,M,c_0,C_0}(1+\beta)\frac {1+c^{2r+1}}{c^{2r+1}}\tau^{r}.
\end{align*}
Thus, before the time $K\tau\ge  \frac {c^{2r+2}}{2\tau^{r}C_{r,M,c_0,C_0}(1+\beta)}(1-\alpha)$,  
the lower bound of the original system is preserved by the numerical scheme. 
After $K\tau$, we can still write the scheme until the lower bound of the density goes to $0$.  
The solvability of the scheme requires the classical condition, 
$\max\Big(C_0,\frac 1{\sqrt {2}} C_0 \sqrt{\frac {M}{cc_0}}, \beta C_0\big(\frac{1}{c}+\frac{1}{c^2}+M_0\big)\Big)\tau 
\le \text{constant}$, where the constant only depends on the numerical method.
Due to the fact that if $T=\mathcal O(\tau^{-r}),$ the lower bound of the density is uniformly controlled by $c$,
we complete the proof of (vi) by using the Taylor expansion ot the energy.  
\end{proof}

\subsubsection{Backward Error Analysis}
In spite of point (vi) in Theorem \ref{sym-pro}, symplectic methods nearly preserve the Hamiltonian for
times much longer than $\mathcal{O}(\tau^{-r})$, since the 
backward error analysis allows for an exponentially small error between the symplectic scheme and
its modified equation. 
To apply the backward error analysis, we need to verify that the coefficients of the equation 
admit an analytic extension on the complex domain, which we do next.

By choosing the principle value of the logarithm of $z$ in $ \mathbb C/ \{0\}$, denoted by 
$\Log(z):=\log|z|+ i\Arg(z)$, it is known that $\Log(z)$ is analytic except along
the negative real axis.  Since $\frac 1{\rho_i}$ and $\log(\rho_i)$ can be extended to analytic complex 
functions for $\rho_i\in \mathbb C/ \{0\}$, we extend 
\begin{align*}
f(S,\rho)&:=(-\frac {\partial \mathcal H}{\partial \rho}, \frac {\partial \mathcal H}{\partial S})\\
&=\Big(-\frac 14\sum_{j\in N(i)}\omega_{ij}(S_i-S_{j})^2-\sum_{j\in N(i)}\widetilde \omega_{ij}(1-\frac{\rho_{j}}{\rho_i}-
\log(\frac {\rho_{j}}{\rho_{i}})),\\
&\qquad \frac 12\sum_{j\in N(i)}\omega_{ij}(S_i-S_j)(\rho_{i}+\rho_j)\Big)
\end{align*}
to a complex function in $\mathbb C^{2n}$ such that for any $y^0\in B$, $f(y)$ is analytic
in the neighborhood of $y^0$ and that there exists $R>0$ such that $$\|f(y)\|\le M_{c}, \;\text{for} \; \|y-y^0\|\le 2R.$$
This is applicable since we can choose $R\le \frac 14 {\text{dist}}(y_0,B)$ such that 
$$\min_{i=1}^N|y_{N+i}|=\min_{i=1}^N|\rho_{i}|\ge c,$$ 
and that 
\begin{align*}
\|f(y)\|_{l^{\infty}}
&\le  E_{max}C_0\Big(\frac 12\frac {M}{cc_0}+\beta  \sqrt{\frac M{cc_0}}+\beta \frac 1c+M_0\Big).
\end{align*}

Thus, the backward error analysis is applicable in our case. 
We first introduce the truncated modified differential equation of \eqref{dhs} with respect to an $r$-th
order numerical scheme,
\begin{align}\label{mod-hs}
\dot{\widetilde y}=F_{\mathcal N}(\widetilde y), \; F_{\mathcal N}(\widetilde y)=f(\widetilde y)+
\tau^rf_{r+1}(\widetilde y)+\cdots+\tau^{N-1}f_{\mathcal N}(\widetilde y)
\end{align}
with $\widetilde y(0)=y(0)$. It is well-known that the above modified equation 
is also a Hamiltonian system with the modified Hamiltonian 
$\widetilde {\mathcal H}(y)=\mathcal H(y)+\tau^r\mathcal H_{r+1}(y)+\dots+\tau^{ N-1}\mathcal H_{\mathcal N}(y)$.
According to \cite[Theorem 7.2 and Theorem 7.6]{HLW06}, 
we have that for the Runge-Kutta method, if $f(y)$ is analytic and $\|f(y)\|\le M_{c}$ in the complex ball $B_{2R}(y_0)$, 
then the coefficients $d_j$ in the Taylor expansion of the numerical method
$$\Phi_{\tau}(y)=y+\tau f(y)+\tau^2d_2(y)+\dots+\tau^jd_j(y)+\dots,$$
are analytic  and satisfy $\|d_j(y)\|\le C\frac {M_c^{j}}{R}$ in $B_R(y_0).$
If $\tau\le  {\tau_0}$ with $\tau_0\le C\frac R {M_c}$ for some constant $C>0$, then there exists 
$\mathcal N=\mathcal N(\tau)$ satisfying $\tau\mathcal N\le h_0$ such that 
\begin{align*}
\|\Phi_{\tau}(y^0)-\widetilde{\phi_{N,\tau}}(y^0)\|\le C\tau M_ce^{-\frac {\tau_0} {\tau}},
\end{align*}
where $y^1=\Phi_{\tau}(y^0)$ is the numerical solution and $\widetilde{\phi_{\mathcal N,\tau}}(y^0)$ is 
the exact solution of \eqref{mod-hs} at $t=\tau$.

As a consequence of the above results, the long-time energy conservation is obtained.
Assume that the numerical solution of the symplectic method $\Phi_{\tau}(y)$ stays in the compact set $B$, 
then there exists $R$, $\tau_0$ and $N(\tau_0)$ such that 
\begin{align*}
|\widetilde{\mathcal H}(y^n)-\widetilde{\mathcal H}(y^0)|\le n\tau M_ce^{-\frac {\tau_0}{\tau}},\\
|\mathcal H(y^n)-\mathcal H(y^0)|\le C\frac {M_c^{p+1}}{R^{p}}\tau^p,
\end{align*}

\begin{cor}
Under the same condition of Theorem \ref{sym-pro}, when $\frac M\beta$ is small enough, 
there exists $\tau_0$ small enough, $C_M>0$, and a modified energy $\widetilde{\mathcal H}$, $\mathcal{O}(\tau^r)$-close
to $\mathcal{H}$, 
such that for any $\tau<\tau_0$, $n\tau<T$,  
$$|\widetilde{\mathcal H}(S^n, \rho^n)-\widetilde{\mathcal H}(S^0,\rho^0)|\le n\tau C_Me^{-\frac {\tau_0}{\tau}}.$$
\end{cor}

\subsection{Regularizations}\label{sec-pde}

Here we look at two instances of regularization for \eqref{hpde}: one based on Fisher information, and one based
on standard viscosity solution.
We assume that $\mathcal M\subset \mathbb R$ is a bounded connected domain, and 
for simplicity restrict to \eqref{hpde} subject to periodic boundary conditions without the term $\mathcal F(\rho)$. 
The initial condition $\rho(0)>0,$ and $S(0)$, are smooth and bounded functions on $\mathcal M$.
We remark that all the proposed scheme can be constructed similarly in other domain in $\mathbb R^d$. 

\subsubsection{Fisher information regularization symplectic scheme} 
For the system \eqref{hpde}, its Lagrangian formalism is equivalent to its Hamiltonian formalism. 
We  can directly apply the Fisher information regularization symplectic scheme \eqref{sym} to the semi-discretization 
of the considered Hamiltonian PDE. 
We use the mid-point scheme applied to the graph generated by the central difference  
scheme under the periodic condition as an example of a fully discrete scheme,
\begin{equation}\label{Fis-mid}\begin{split}
\rho^{n+1}_i&= \rho^n_i +\frac{\partial \mathcal H(S^{n+\frac 12},\rho^{n+\frac 12})}{\partial S_i}\tau,\\
&=\rho^n_i-\sum_{j\in N(i)}\frac \tau{h^2}(S_j^{n+\frac 12}-S_i^{n+\frac 12})\theta_{ij}(\rho^{n+\frac 12})\\
S^{n+1}_i&= S^n_i-\frac {\partial \mathcal H(S^{n+\frac 12},\rho^{n+\frac 12})}{\partial \rho_i}\tau ,\\
&=  S^n_i-\frac 12\sum_{j\in N(i)}\frac \tau{h^2}(S_i^{n+\frac 12}-S_j^{n+\frac 12})^2 \frac {\partial \theta_{ij}}{\partial \rho_i}(\rho^{n+\frac 12})-\beta \frac {\partial I(\rho^{n+\frac 12})}{\partial \rho_i}\tau.
\end{split}\end{equation}
Then all the properties in Theorem \ref{sym-pro} hold. 
According to the priori estimate on the coefficients of discrete Hamiltonian PDEs, we have the following space-time step size restriction,
\begin{align*}
\tau\le C \min\Big(\frac 1{C_0}, \frac 1{C_0} \sqrt{\frac {cc_0}{M}},\frac 1{C_0}  \frac {c^2}{\beta(1+c+M_0c^2)}\Big),
\end{align*}
where
\begin{align*}
c\ge  \min( \frac 12\min_i \rho_i(0), \frac 1{1+N\exp(\frac {M(N-1)([\frac {N-1}2]+1)h^2}{\beta})}),\;\text{and}\; c_0=C_0=\frac 1{h^2}.
\end{align*}

%

If we do not add a regularization term, like Fisher information, to the numerical scheme of \eqref{dhs}, then the numerical
scheme may develop singularities and produce unstable behavior. 
The following example  indicates that even the structure-preserving numerical scheme which uses
the upwind weight $\theta^U$ without regularization will fail --at a finite step $n$-- 
to maintain positivity for $\rho_i^n$, and will lead to blow up for $S_i^n$.

\begin{ex}\label{ex-non}
Assume that  the graph has only two points. Assume that $\rho_1(0), \rho_2(0)>0$ and $S_1(0), S_2(0)$ are the corresponding initial densities and potentials of the two points. 
We choose $\theta_{ij}=\widetilde \theta_{ij}$ as the probability weight 
\begin{align*}
&\theta_{ij}(\rho)=\rho_j, \; \text{if} \; S_i>S_j,\\
&\theta_{ij}(\rho)=\rho_i,\; \text{if}\; S_i<S_j.
\end{align*}
For simplicity, assume that $S_1(0)>S_2(0)$, $\mathcal F(\rho)=0$.
Then the finite dimensional system becomes 
\begin{align*}
\dot \rho_1&=   (S_1-S_2)\rho_2,\;
\dot \rho_2=  (S_2-S_1)\rho_2,\\
\dot S_1&=0,\;
\dot S_2=-\frac 12 |{S_1-S_2}|^2.
\end{align*}
Then $S_1-S_2=\frac {S_1(0)-S_2(0)}{1- \frac 12(S_1(0)-S_2(0))t)}$. Until $t<\frac 2{S_1(0)-S_2(0)}$, $\rho_1$ and $\rho_2$ possess  the strict positivity property. When $t=\frac 2{S_1(0)-S_2(0)}$, $\rho_1=1,$ $\rho_2=0$.
\end{ex}

\subsubsection{Regularization by adding viscosity}
As alternative to adding Fisher information as regularization term,
a classical regularization procedure is obtained by adding numerical viscosity in order
to obtain monotone schemes for $S$. 
For example, by introducing the numerical viscosity  $\alpha_i(S^n):=\alpha(S^n_{i+1}-2S^n_{i}-S^n_{i-1})$, 
where $\alpha \in \R$ is used to guarantee the monotonicity of $S^{n+1}_i$. 
This is a standard way of proceeding ({\emph{elliptic regularization}}, 
which we now detail and further use in the numerical tests for comparison
purposes.  As we will see, although adding viscosity does lead to a well defined discretization \eqref{scheme-1}, 
unlike the regularization scheme \eqref{Fis-mid}, 
the numerical scheme \eqref{scheme-1} does not preserve relevant properties of the Hamiltonian
system (see Theorem \eqref{tm-vis} below). This can be easily appreciated
in the numerical tests in Section \ref{sec-test}.

Assume that $\max_{i,n}|\frac {S^{n}_{i+1}-S^n_{i}}h|\le R$.
Then, we can choose $\alpha$ $(0<\alpha <\frac 12, \alpha \ge R \frac {\tau} h)$ such that
\begin{align*}
&1-\frac {\tau} h (\frac {(S_{i+1}^n-S_{i}^n)^+}h+\frac {(S_{i-1}^n-S_{i}^n)^+}{h})
-2\alpha \ge 0,\\ 
& -\frac {\tau} h \frac {(S_{i+1}^n-S_{i}^n)^+}h+\alpha\ge 0, \\
& -\frac  {\tau} h \frac {(S_{i-1}^n-S_{i}^n)^+}h+\alpha \ge 0.
\end{align*}

Doing so, we get the following scheme:
\begin{equation}\label{scheme-1}\begin{split}
\rho_i^{n+1}&=\rho_{i}^n+\tau(\frac {S^n_{i}-S^n_{i+1}}{h^2})^+\rho_{i+1}^n+\tau(\frac {S^n_{i}-S^n_{i-1}}{h^2})^+\rho_{i-1}^n\\
&+
\tau(\frac {S^n_{i}-S^n_{i+1}}{h^2})^-\rho_{i}^n+\tau(\frac {S^n_{i}-S^n_{i-1}}{h^2})^-\rho_{i}^n\\
S_i^{n+1}&=S_i^n-\frac 12\tau|\frac {(S^n_{i}-S^n_{i+1})^-}{h}|^2-\frac 12\tau|\frac {(S^n_{i}-S^n_{i-1})^-}{h}|^2+\alpha_i(S^n).
\end{split}\end{equation}

Let $\rho^0$ and $S^0$ be the grid function of $\rho(0)$ and $S(0)$ on the grid $G$.
Then the proposed scheme \eqref{scheme-1} enjoys the following properties, which implies that
the numerical viscosity term leads to positivity of the density function and uniform boundedness of $S.$

\begin{tm}\label{tm-vis}
Assume that $\max_{i,n}|\frac {S^{n}_{i+1}-S^n_{i}}h|\le R,\ \alpha\ge R \frac {\tau} h$. 
Then there exists a unique solution $(\rho_i^n, S_i^n)_n$ of \eqref{scheme-1} and satisfies the following properties.

\begin{enumerate}[label=(\roman*)]
\item Mass is preserved: $\sum_{i}\rho_i^n=\sum_{i}\rho_i^0$.
\item It is strictly positive:
if $\min \rho_i^0 >0$, then $\min \rho_i^n >0$ for any $n$. 
\item If $\frac {\tau}h$ is sufficient small, and  $\tau,h\to 0$, then $S_i^n$ converges to the viscosity solution of the Hamilton Jacobi equation.
\item  It holds that $\lim\limits_{n\to\infty} S^n=S^{\infty}$ and $\lim\limits_{n\to\infty} \rho^n=\rho^{\infty}$, where $\rho^{\infty}\in \mathcal P_o(G)$.
\item It holds that
\begin{align*}
\|S^n\|_{l^{\infty}}\le \|S^0\|_{l^{\infty}},\;
\|\rho^n\|_{l^{\infty}}\le  \max((1+R\frac {\tau}h)^n\|\rho^0\|_{l^{\infty}},1/h).
\end{align*}
\end{enumerate}
\end{tm}

\begin{proof}
For Properties (i), (iii) and (v),
we refer to \cite{CL84} for their proof relative to the numerical approximation 
$$ S_i^{n+1}=S_i^n-\frac 12\tau|\frac {(S^n_{i+1}-S^n_{i})}{h}|^2+\alpha_i(S^n).$$
We proceed to prove (ii) and (iv).

Due to the expression of $\rho^{n+1}_{i}$,
we get 
\begin{align*}
\rho_i^{n+1}&\ge \rho_{i}^n+
\frac {\tau} h\Big((\frac {S^n_{i}-S^n_{i+1}}h)^-+(\frac {S^n_{i}-S^n_{i-1}}h)^-\Big)\rho_{i}^n
\ge (1-2R\frac {\tau} h)\rho_{i}^n,
\end{align*}
which leads to 
\begin{align*}
\rho_i^{n}\ge (1-2R\frac {\tau} h)^n\rho_{i}^0.
\end{align*}
Thus we have that $\rho_i^{n}\ge e^{-c_1\frac {\tau n}h}\min_i\rho_{i}^0$ for some $c_1>0$ and (ii) holds.

Now we are in a position to show (iv).
Since $S^n$ is uniformly bounded with respect to $n$, there exists a sub-sequence $\{S^{n_{k}}\}_k$ converging
to a constant $S^{\infty}$. By using the comparison principle, we get that for any $k,l,m\in \N^+$, 
\begin{align*}
\|S^{n_{k}+m}-S^{n_{l}+m}\|_{l^{\infty}}\le \|S^{n_{k}}-S^{n_{l}}\|_{l^{\infty}}.
\end{align*}
Thus $\{S^{n_{k}+m}\}_k$ is a Cauchy sequence in $l^{\infty}(V \times \N^+)$ and converges to the same limit $S^{\infty}$. 
On the other hand, one can also check that 
the solution of the following relation 
\begin{align}\label{eli-ode}
\frac 12\Big|\frac {(S^{\infty}_{i}-S^{\infty}_{i+1})^-}{h}\Big|^2+\frac 12\Big|
\frac {(S^{\infty}_{i}-S^{\infty}_{i-1})^-}{h}\Big|^2+\alpha_i(S^{\infty})=0
\end{align}
must be $0$. Indeed, let us
assume that there is a nonzero solution for \eqref{eli-ode}.
From the fact that  $\alpha_i(S^{\infty})>0$ if $S^{\infty}_{i}-S^{\infty}_{i+1}< 0$, $S^{\infty}_{i}-S^{\infty}_{i+1}< 0$ and $\alpha_i(S^{\infty})<0$, if  $S^{\infty}_{i}-S^{\infty}_{i+1}>0$, $S^{\infty}_{i}-S^{\infty}_{i+1}> 0$, the nonzero solution of \eqref{eli-ode} should has different signs for $S^{\infty}_{i}-S^{\infty}_{i+1}$ and $S^{\infty}_{i}-S^{\infty}_{i-1}$ at each node $a_i$.
For simplicity assume that $S^{\infty}_{i}-S^{\infty}_{i+1}<0$ and $S^{\infty}_{i}-S^{\infty}_{i-1}>0$.
Now adding all the equations together, we obtain that 
\begin{align*}
\sum_{i=1}^N \frac 12\Big|\frac {(S^{\infty}_{i}-S^{\infty}_{i+1})^-}{h}\Big|^2=0,
\end{align*}
which contradicts the fact that $S^{\infty}_{i}-S^{\infty}_{i+1}<0$ for $i=1,\cdots,N$.
Repeating this argument, it follows that for any $1\le n\le N$,
the solution of the following relation 
\begin{align*}
\sum_{i=1}^n \left[ \frac 12\Big|\frac {(S^{\infty}_{i}-S^{\infty}_{i+1})^-}{h}\Big|^2+\frac 12\Big|\frac {(S^{\infty}_{i}-S^{\infty}_{i-1})^-}{h}\Big|^2+\alpha_i(S^{\infty})\right] =0
\end{align*}
must be $0$. 
As a consequence, for any subsequence $\{S^{n_k}\}_k$, we have 
\begin{align*}
\frac {S_i^{n_k+1}-S_i^{n_k}}{\tau}&=-\frac 12\Big|\frac {(S^{n_k}_{i}-S^{n_k}_{i+1})^-}{h}\Big|^2-\frac 12\Big|\frac {(S^{n_k}_{i}-S^{n_k}_{i-1})^-}{h}\Big|^2+\alpha_i(S^{n_k})
\end{align*}
converges to 
\begin{align*}
\frac 12\Big|\frac {(S^{\infty}_{i}-S^{\infty}_{i+1})^-}{h}\Big|^2+\frac 12\Big|\frac {(S^{\infty}_{i}-S^{\infty}_{i-1})^-}{h}\Big|^2+\alpha_i(S^{\infty})=0,
\end{align*}
which only possesses the unique zero solution.
Since $\|\rho\|_{l^1}=1$, there exists a subsequence $\{\rho^{n_k}\}_{k}$ which converges  to a density probability $\rho^{\infty}$.
From \eqref{scheme-1} and the convergence of $S$, we are in a position to show that all the subsequence of $\{\rho^{n}\}_{n}$ converges to the same limit $\rho^{\infty}$.  
In the following, we show that for given $k$ sufficient large, then $\{\rho^{n_k+m}\}_{m}$ is a Cauchy sequence.
Indeed, we have 
\begin{align*}
\|\rho_i^{n_k+1}-\rho_{i}^{n_k}\|_{l^{\infty}}
&\le  {\tau}\|\rho_{i+1}^{n_k}\|_{l^{\infty}} \|\Big((\frac {S^{n_k}_{i}-S^{n_k}_{i+1}}{h^2})^+-(\frac {S^{\infty}_{i}-S^{\infty}_{i+1}}{h^2})^+
\Big)\|_{l^{\infty}}\\
&+ {\tau}\|\rho_{i+1}^{n_k}\|_{l^{\infty}} \|\Big(\frac {S^{n_k}_{i}-S^{n_k}_{i-1}}{h^2})^+-\frac {S^{\infty}_{i}-S^{\infty}_{i-1}}{h^2})^+
\Big)\|_{l^{\infty}}\\
&+\tau \|\rho_{i}^{n_k}\|_{l^{\infty}}\|\Big(\frac {S^{n_k}_{i}-S^{n_k}_{i+1}}{h^2})^--(\frac {S^{\infty}_{i}-S^{\infty}_{i+1}}{h^2})^-
\Big)\|_{l^{\infty}}\\
&+\tau \|\rho_{i}^{n_k}\|_{l^{\infty}}\|\Big(\frac {S_{i}^{n_k}-S^{n_k}_{i-1}}{h^2})^--(\frac {S^{\infty}_{i}-S^{\infty}_{i-1}}{h^2})^-
\Big)\|_{l^{\infty}},
\end{align*}
which, together with  the uniform convergence of $S$, implies that $\rho^{n_k+m}$ is a Cauchy sequence and possesses the same limit $\rho^{\infty}$.
\end{proof}

\section{Numerical examples}\label{sec-test}

Here we show performance of the numerical schemes on several examples. 
All the numerical tests are performed under periodic boundary conditions in space, 
for given initial conditions $\rho(0)=\rho^0 $ and $S(0)=S^0$, as specified below.

\begin{ex}\label{Geo}[Geodesic equations]  This is the system \eqref{GeodEqn1}:
\begin{align*}
\partial_t \rho  +\nabla\cdot( \rho \nabla S) =0,\\
\partial_t S +\frac 12|\nabla S |^2=0.
\end{align*}
\end{ex}

We report on the results of two different strategies: the upwind scheme \eqref{scheme-1} with numerical viscosity,
and the Fisher information regularization symplectic scheme \eqref{Fis-mid}. 
We choose three different initial value conditions to compare the evolution of the density function and energy. 
(The different behaviors of $S$ and $\nabla S$ for \eqref{scheme-1} and \eqref{Fis-mid} are not of interest, 
since for \eqref{scheme-1} $S$ will always converge to a constant; see Theorem \ref{tm-vis}.) 

In Figure \ref{fig-0315}, we show the behavior of \eqref{scheme-1} and \eqref{Fis-mid} with initial
value $\rho^0(x)=\frac {\exp(-10(x-0.5)^2)}{K}$ and $S^0(x)=-\frac 15\log(\cosh(5(x-0.5))).$ 
Here $K$ is a normalization constant so that $\int_0^1\rho^0(x)dx=1.$ 
We observe that for $T<0.15$ the two scheme behave quite closely to each other
and the density concentrates at the point $0.5$.  But, after $T=0.15$,  the density of 
\eqref{Fis-mid} begins 
to oscillate. 
Here, we choose spatial step-size $h=5\times 10^{-3}$, temporal step-size $\tau=10^{-4}$, 
viscosity coefficient $\alpha=1/12$ for \eqref{scheme-1}, and $\theta_{ij}(\rho)=\theta_{ij}^U(\rho),$
$\widetilde \theta_{ij}(\rho)=\theta_{ij}^L(\rho),$ $\beta=10^{-5}$ for \eqref{Fis-mid}. 
In Figure \ref{fig-osc}, we also plot the density functions computed by \eqref{Fis-mid} with different
schemes and  different temporal and spatial step sizes, and clearly the oscillations appear 
to be independent of the choice of schemes and mesh sizes; this leads us to believe that
the oscillations exists for the continuous system.

\begin{figure}
\centering
\subfigure {
\begin{minipage}[b]{0.31\linewidth}
\includegraphics[width=1.15\linewidth]{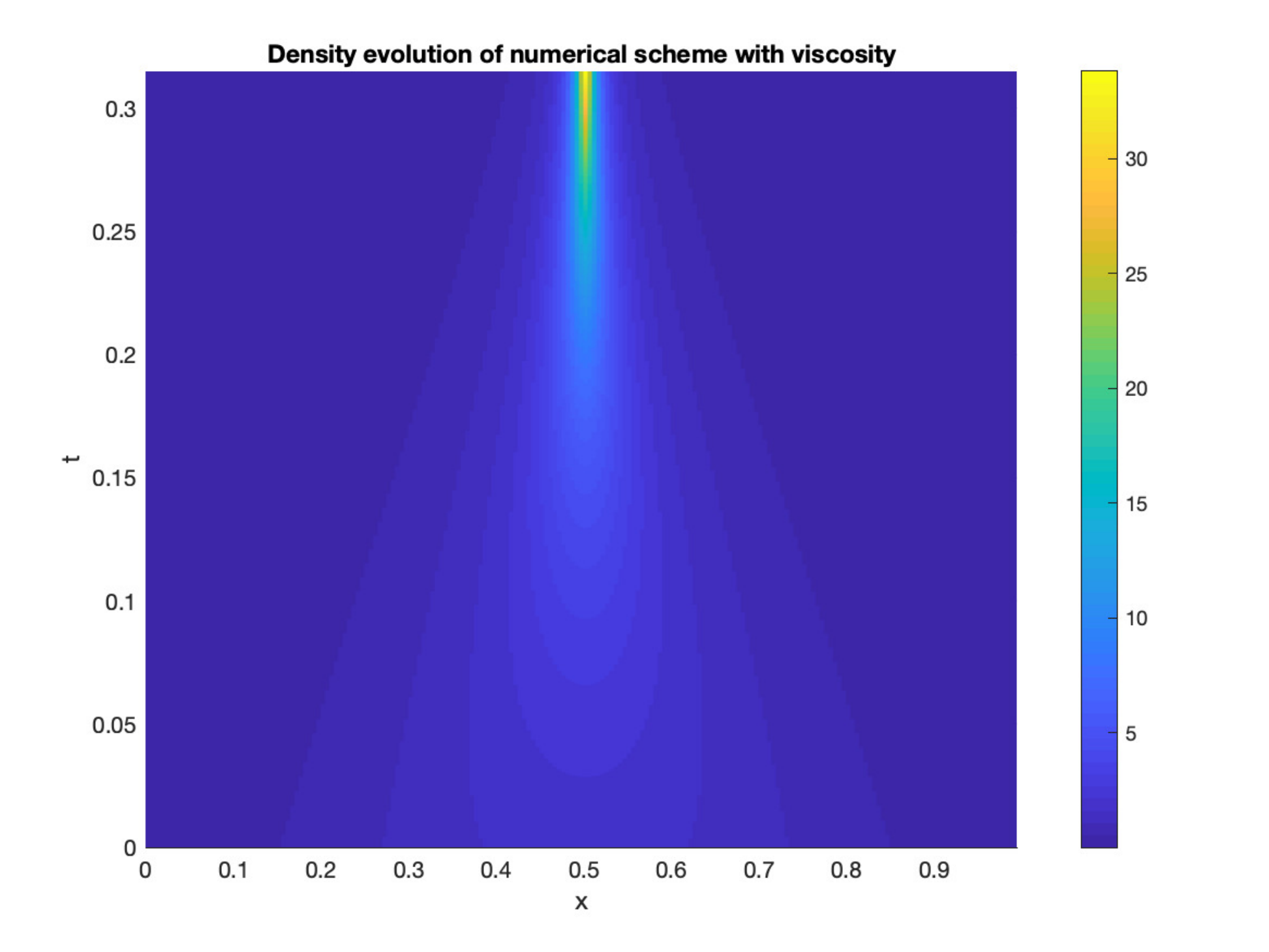}
\includegraphics[width=1.15\linewidth]{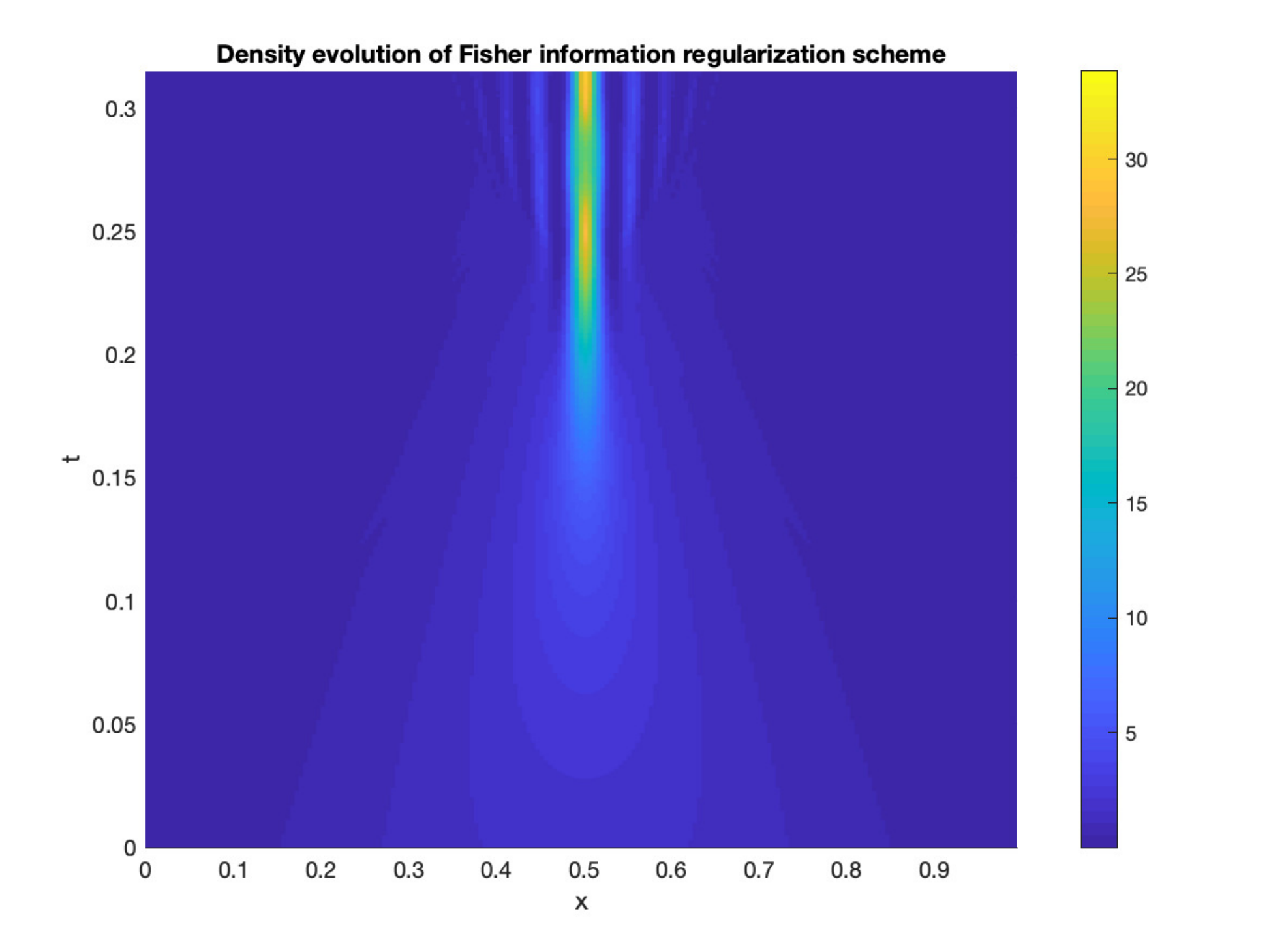}
\end{minipage}}
\subfigure{
\begin{minipage}[b]{0.31\linewidth}
\includegraphics[width=1.15\linewidth]{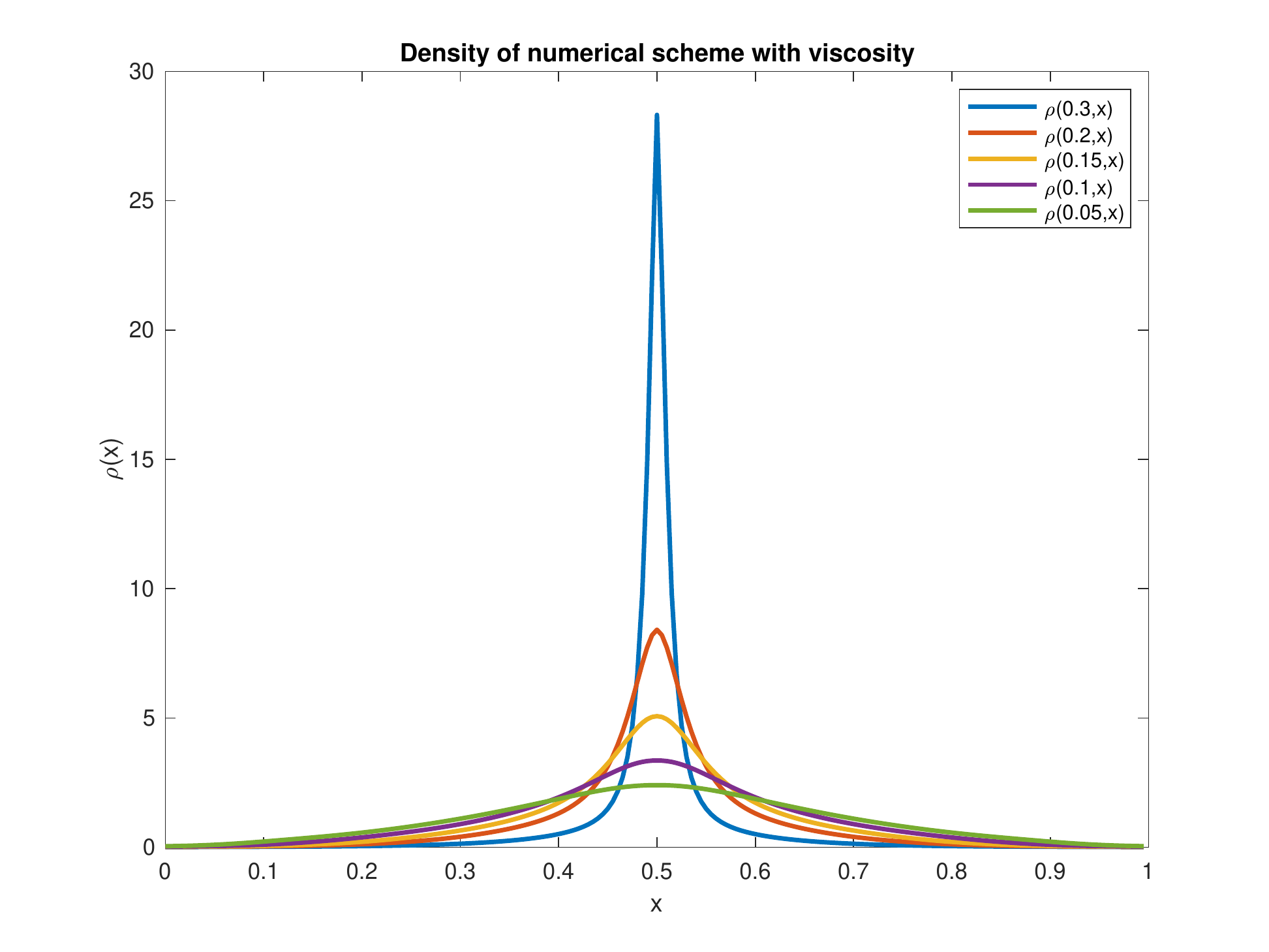}
\includegraphics[width=1.15\linewidth]{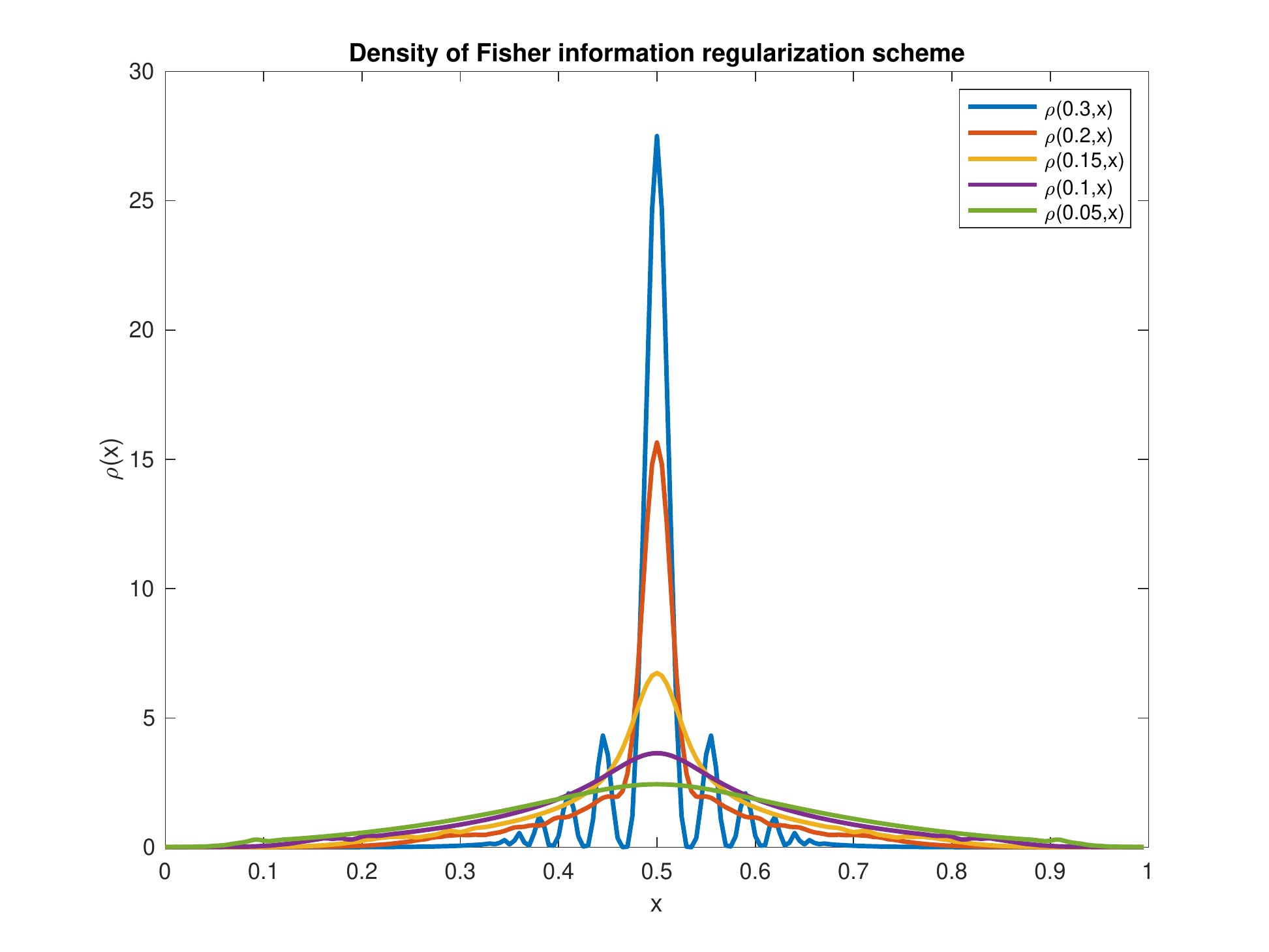}
\end{minipage}}
\subfigure{
\begin{minipage}[b]{0.31\linewidth}
\includegraphics[width=1.15\linewidth]{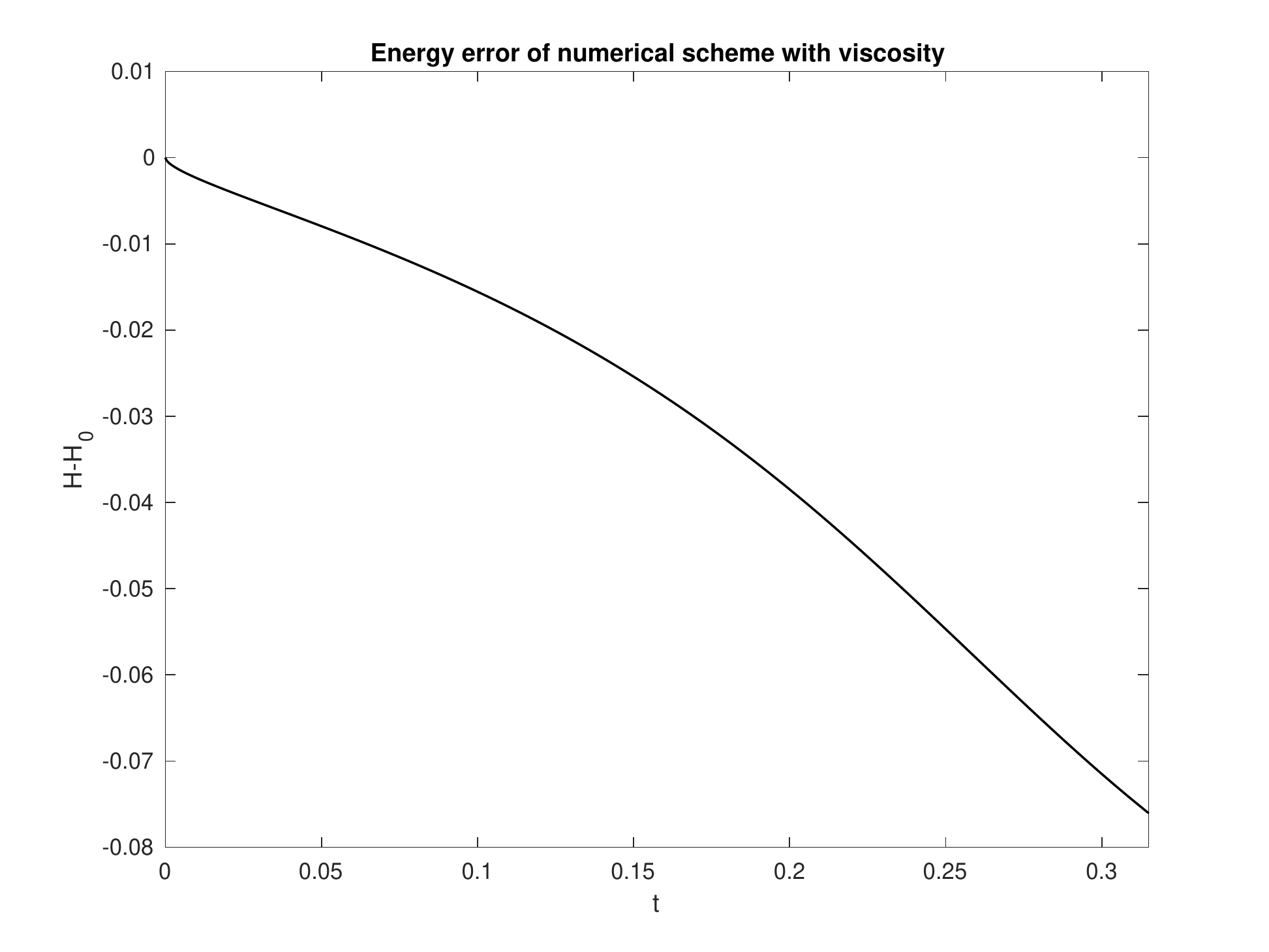}
\includegraphics[width=1.15\linewidth]{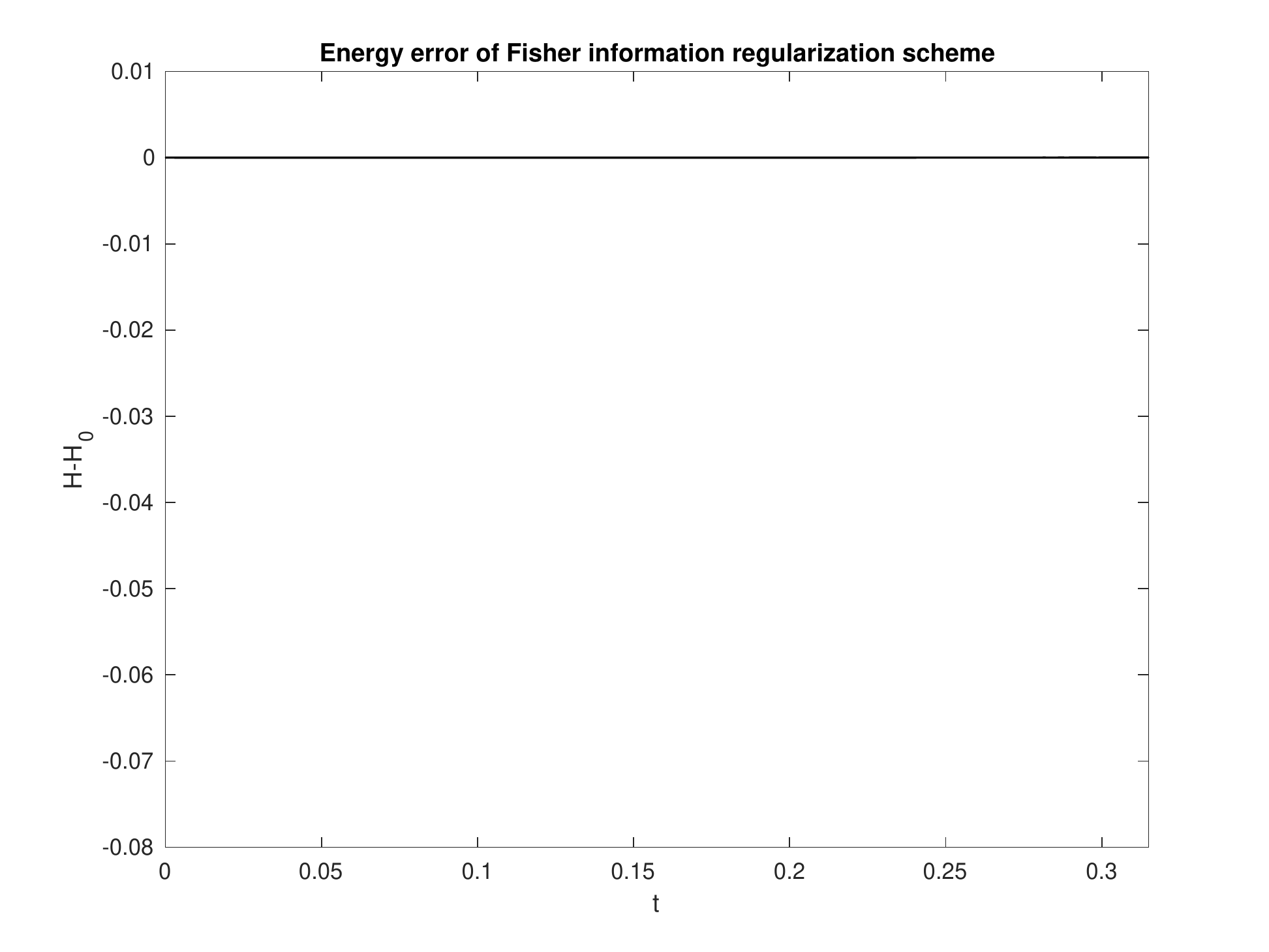}
\end{minipage}}

\centering 
\caption{The contour plot of $\rho(t,x)$ (left), snapshots 
of $\rho(t,x)$ at $t=(0.3,0.2,0.15,0.1,0.05)$ (middle) and energy error before $T=0.315$ (right) for the upwind 
scheme \eqref{scheme-1} with numerical viscosity (top) and the Fisher information regularization symplectic scheme 
\eqref{Fis-mid} (bottom).}
\label{fig-0315}
\end{figure}

\begin{figure}
\centering
 
\subfigure[Midpoint scheme]{
\includegraphics[width=2.4in,height=2.3in]{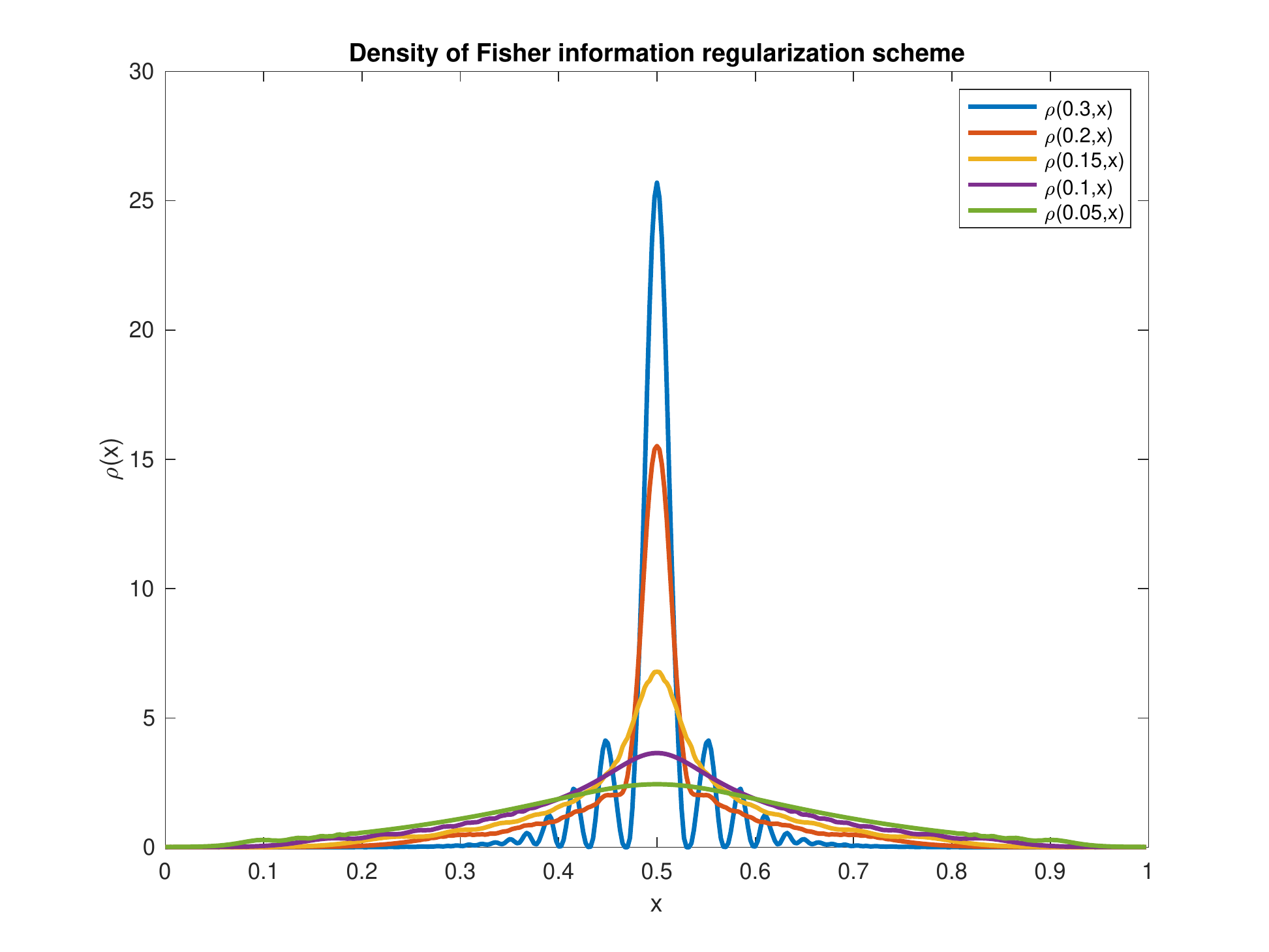}
\includegraphics[width=2.4in,height=2.3in]{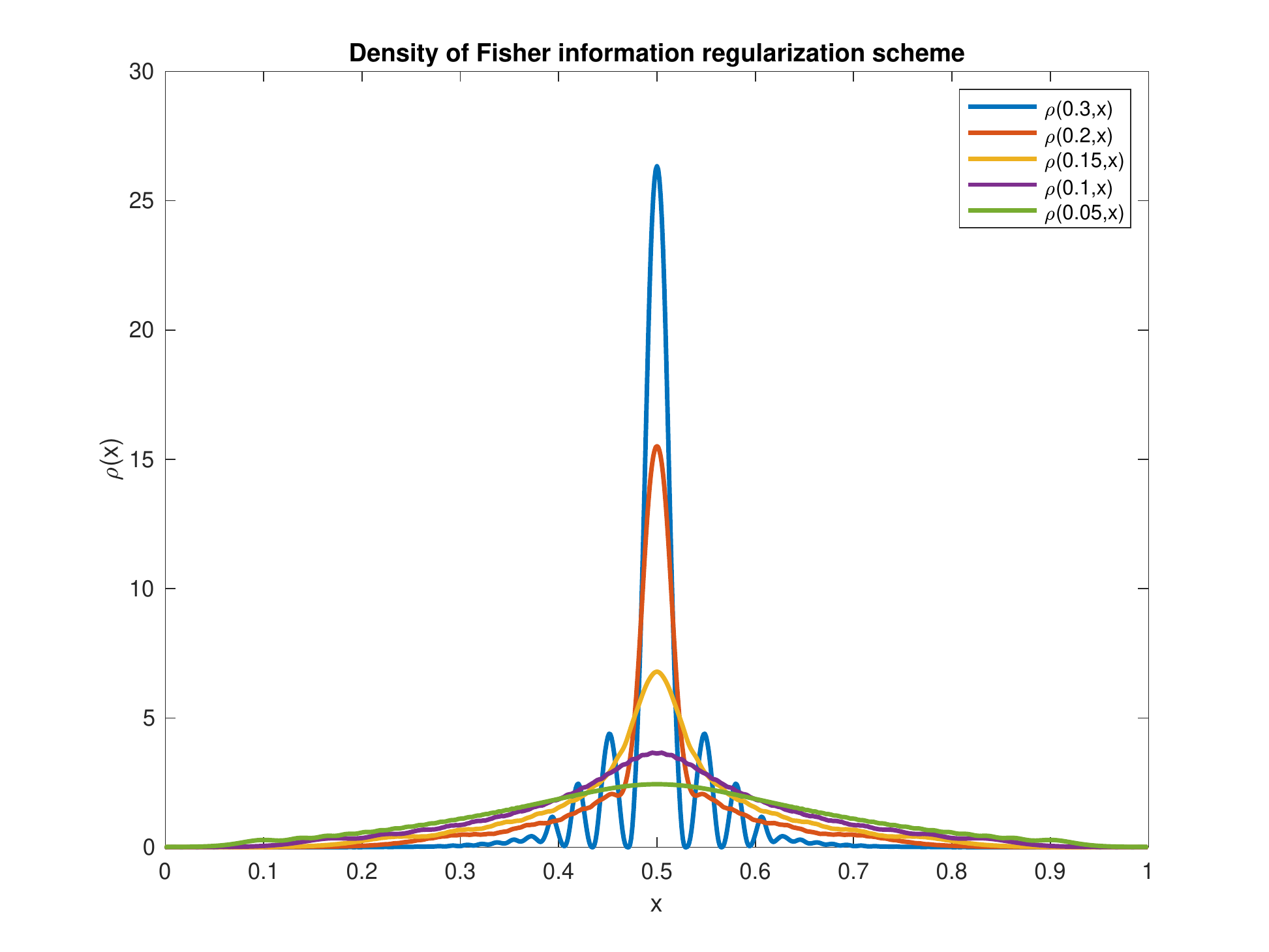}
}

\subfigure[Implicit Euler scheme]{
\includegraphics[width=2.4in,height=2.3in]{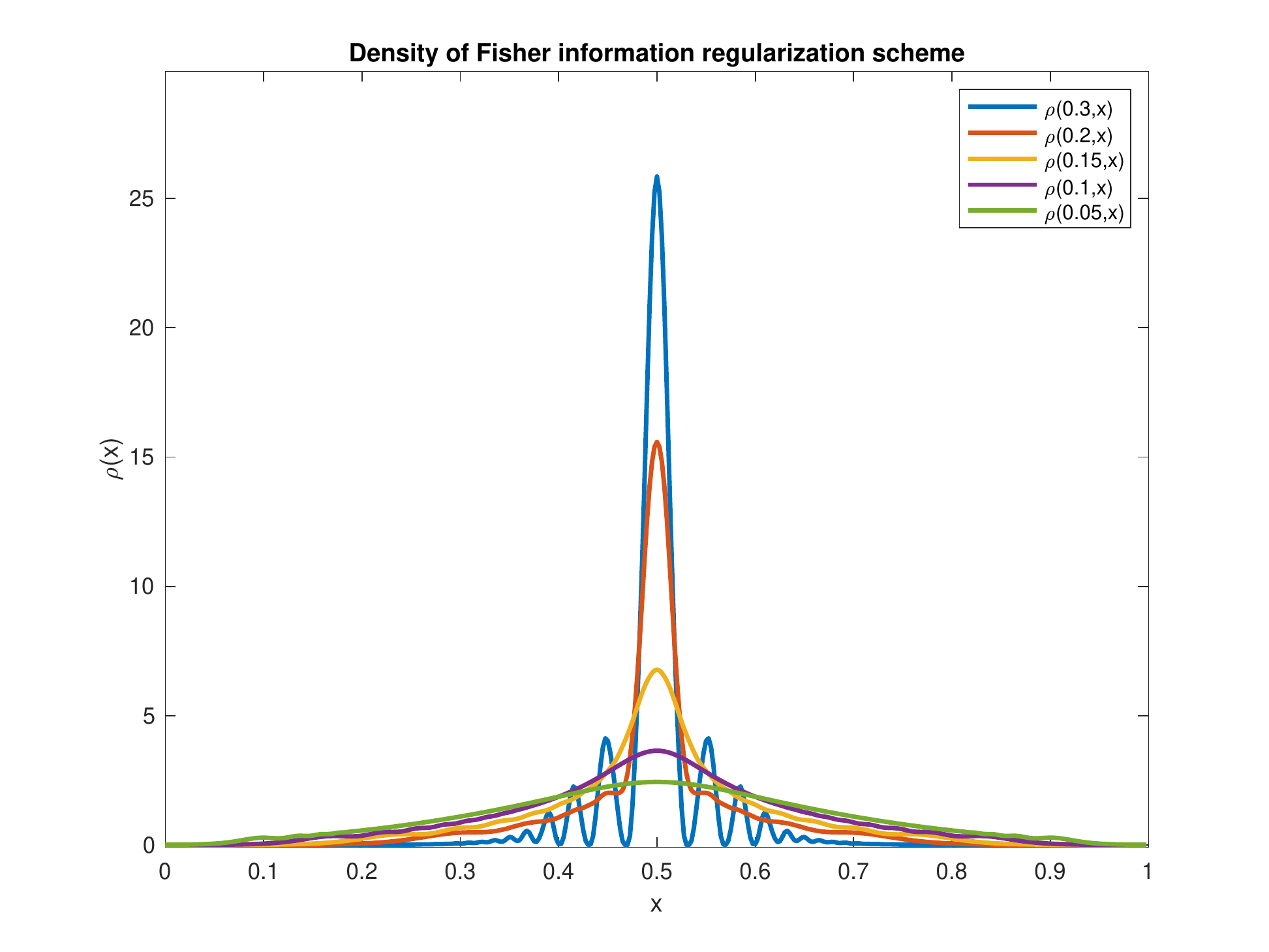}
\includegraphics[width=2.4in,height=2.3in]{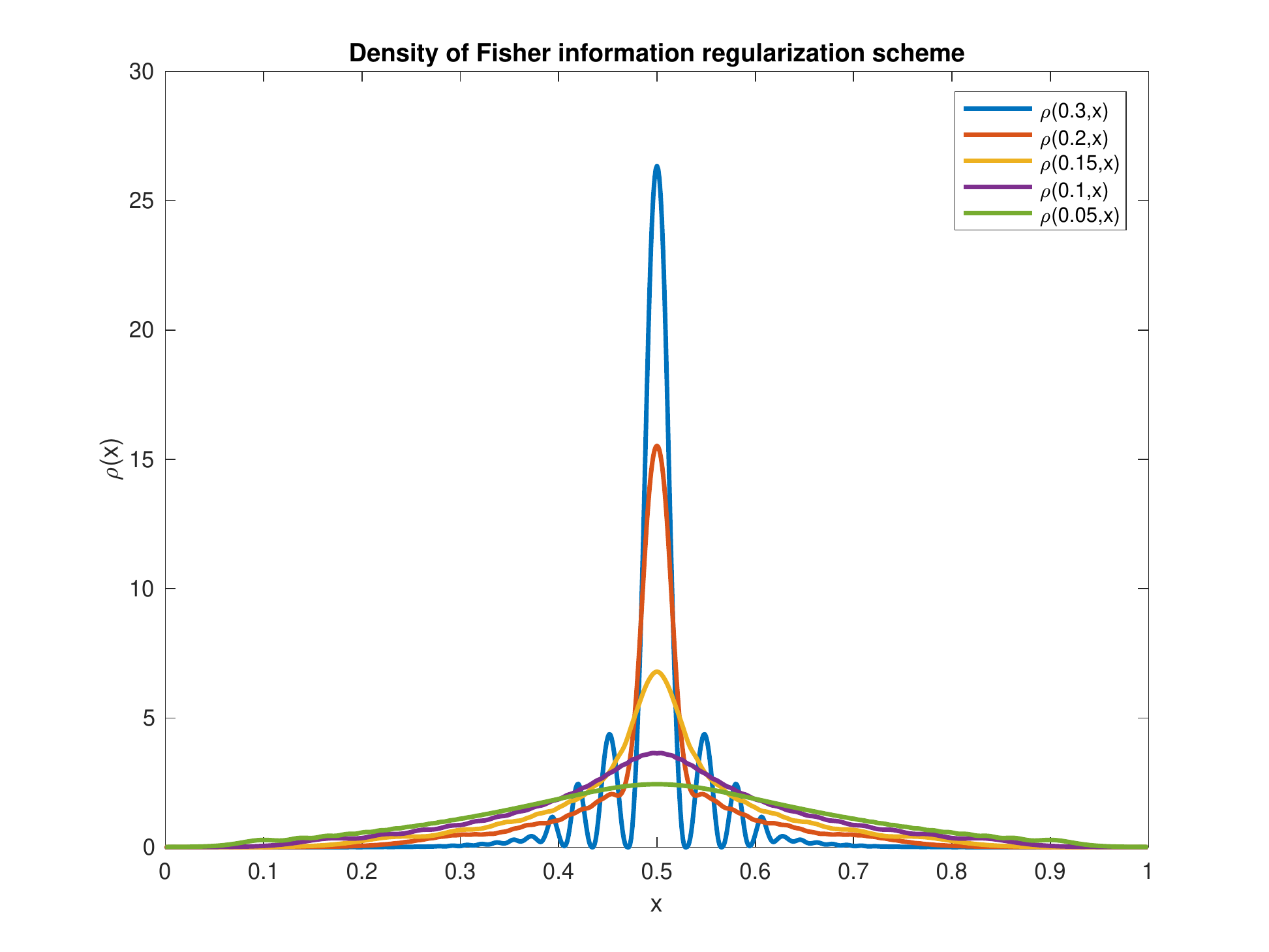}
}

\subfigure[Symplectic Euler scheme]{
\includegraphics[width=2.4in,height=2.3in]{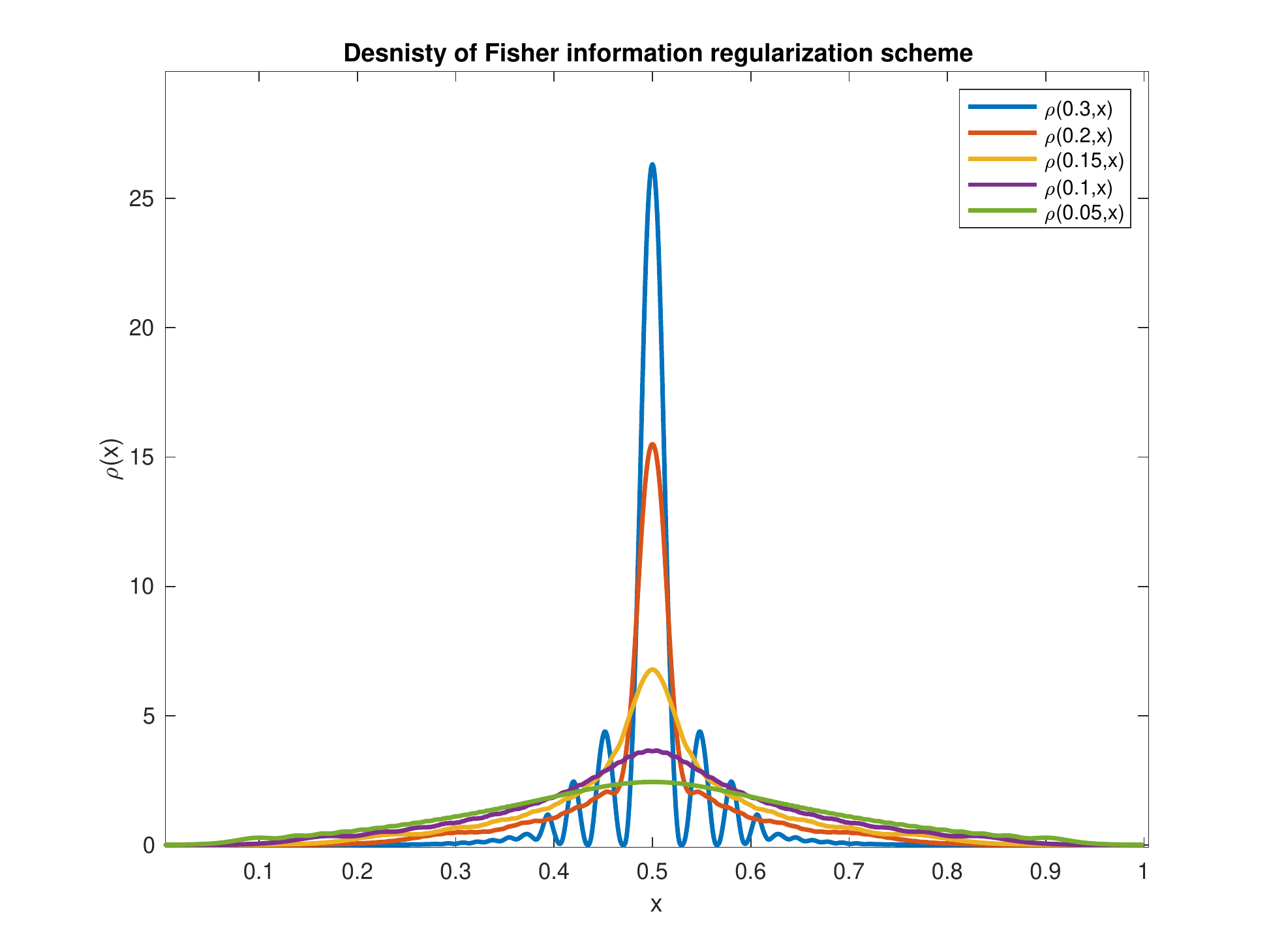}
\includegraphics[width=2.4in,height=2.3in]{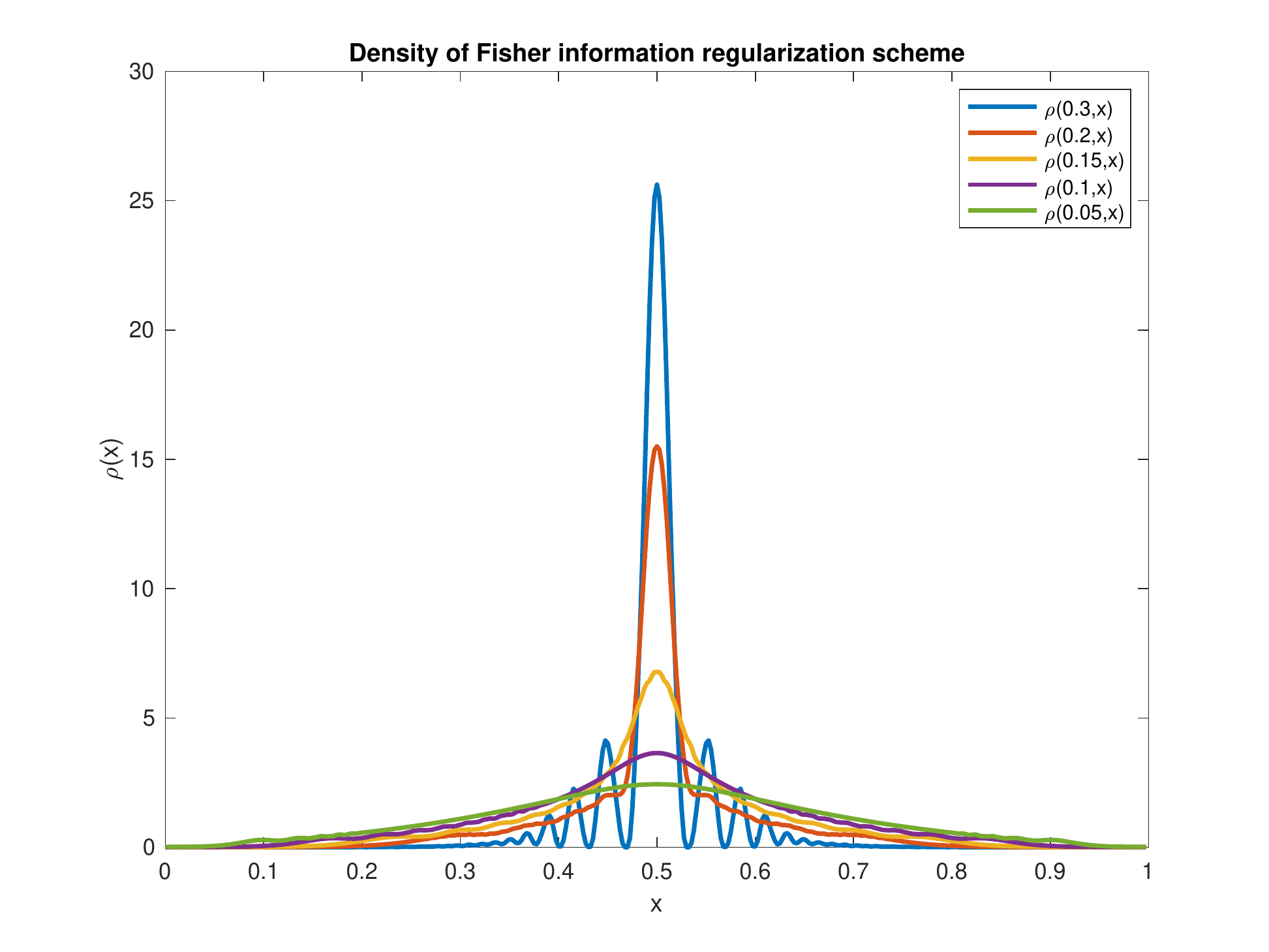}
}

\centering 
\caption{In (a) and (c), there are snapshots of $\rho(t,x)$ at $t=(0.3,0.2,0.15,0.1,0.05)$ for 
\eqref{Fis-mid} and \eqref{Fis-eul} with $h=0.25\times 10^{-2},\tau=0.25\times 10^{-4}$ (left) 
and $h=0.125\times 10^{-2},\tau=0.2\times 10^{-4}$ (right). In (b), we show snapshots of $\rho(t,x)$ at 
$t=(0.3,0.2,0.15,0.1,0.05)$ for \eqref{Fis-mid} 
with  $h=0.25\times 10^{-2},\tau=1/3\times 10^{-4}$ (left) and $h=1/8\times 10^{-2},\tau=1/2\times 10^{-5}$ (right). 
}
\label{fig-osc}
\end{figure}

In Figure \ref{fig-0315b} and Firgure \ref{fig-05}, we observe the same phenomenon for different initial conditions.
In Figure \ref{fig-0315b}, we take $\mathcal M=[0,1]$, $\rho^0(x)=1$ and $S^0(x)=-\frac 15\log(\cosh(5(x-0.5))).$ 
We choose spatial step-size $h=1.5\times 10^{-3}$, temporal step-size $\tau=1.3863\times 10^{-5}$, 
viscosity coefficient $\alpha=8\times 10^{-2}$ for \eqref{scheme-1}, and $\theta_{ij}(\rho)=\theta_{ij}^U(\rho),$
$\widetilde \theta_{ij}(\rho)=\theta_{ij}^L(\rho),$ $\beta=5\times 10^{-7}$ for \eqref{Fis-mid}. In Firgure \ref{fig-05},
we choose  $\rho^0=\frac 12$, $S^0=\frac  18\sin(2\pi x)$, $\mathcal M=[0,2]$,the spatial step-size $h=10^{-2}$, 
temporal step-size $\tau=10^{-4}$, viscosity coefficient $\alpha=5\times 10^{-2}$ for \eqref{scheme-1}, and 
$\theta_{ij}(\rho)=\theta_{ij}^U(\rho),$
$\widetilde \theta_{ij}(\rho)=\theta_{ij}^L(\rho),$ $\beta=10^{-4}$ for \eqref{Fis-mid}. 
All these numerical tests show that the Fisher information regularization scheme \eqref{Fis-mid} 
preserves more structures for \eqref{hpde}, such as the energy evolution and time transverse invariance,
compared to the numerical scheme \eqref{scheme-1}. 
Meanwhile \eqref{Fis-mid} causes oscillatory behaviors after the singularity of \eqref{hpde} is developed. 

\begin{figure}

\centering
\subfigure {
\begin{minipage}[b]{0.31\linewidth}
\includegraphics[width=1.15\linewidth]{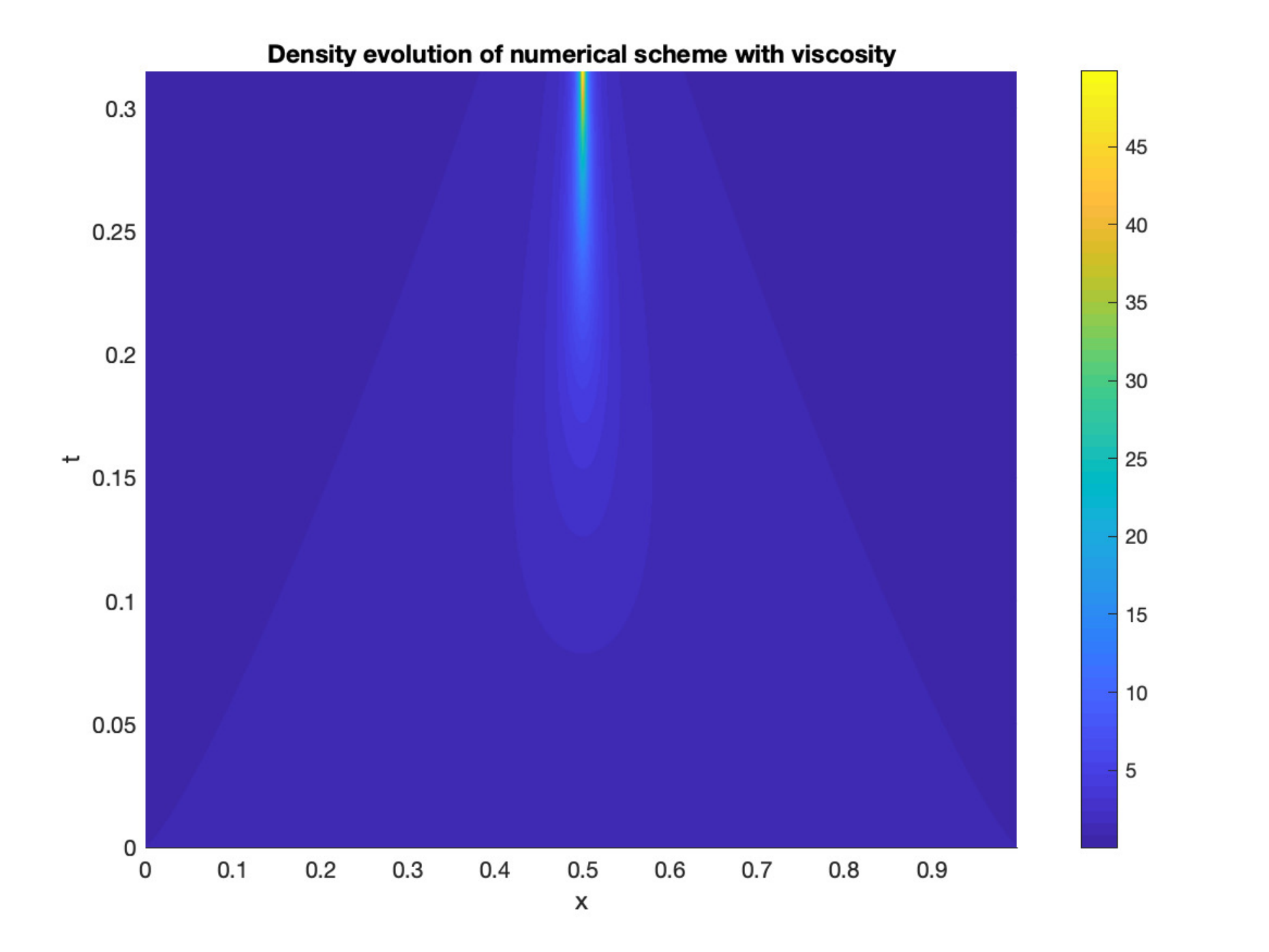}
\includegraphics[width=1.15\linewidth]{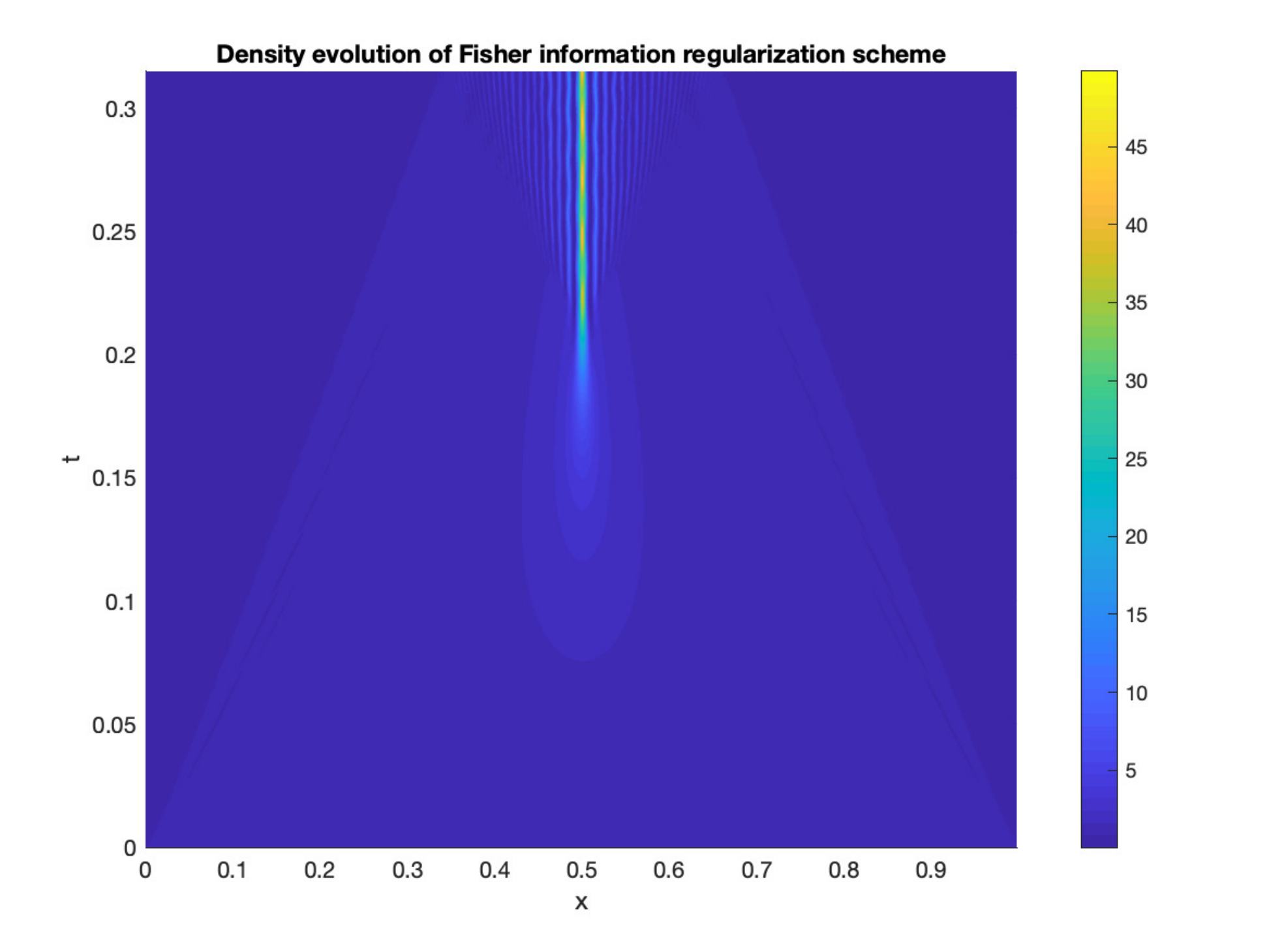}
\end{minipage}}
\subfigure{
\begin{minipage}[b]{0.31\linewidth}
\includegraphics[width=1.15\linewidth]{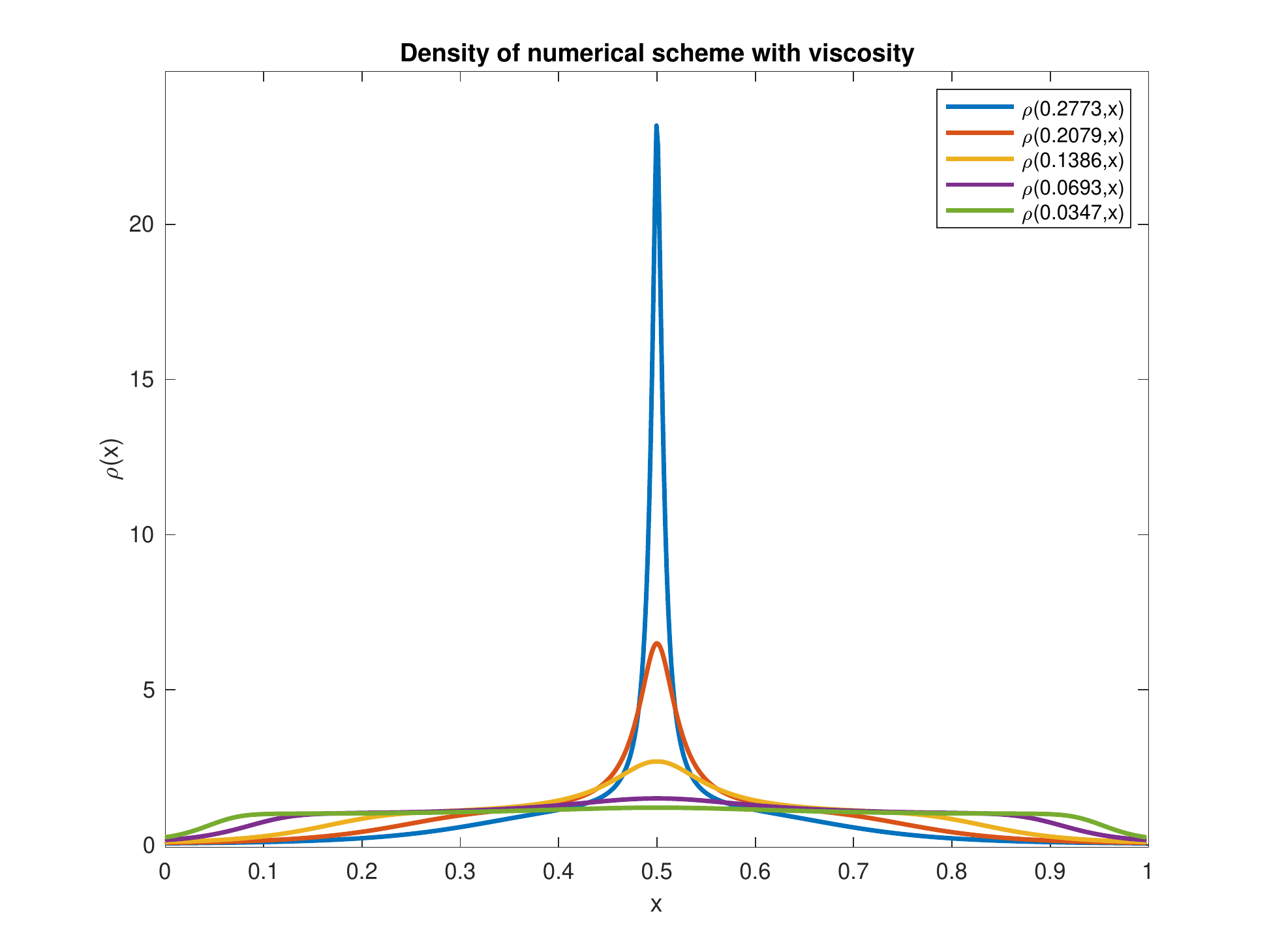}
\includegraphics[width=1.15\linewidth]{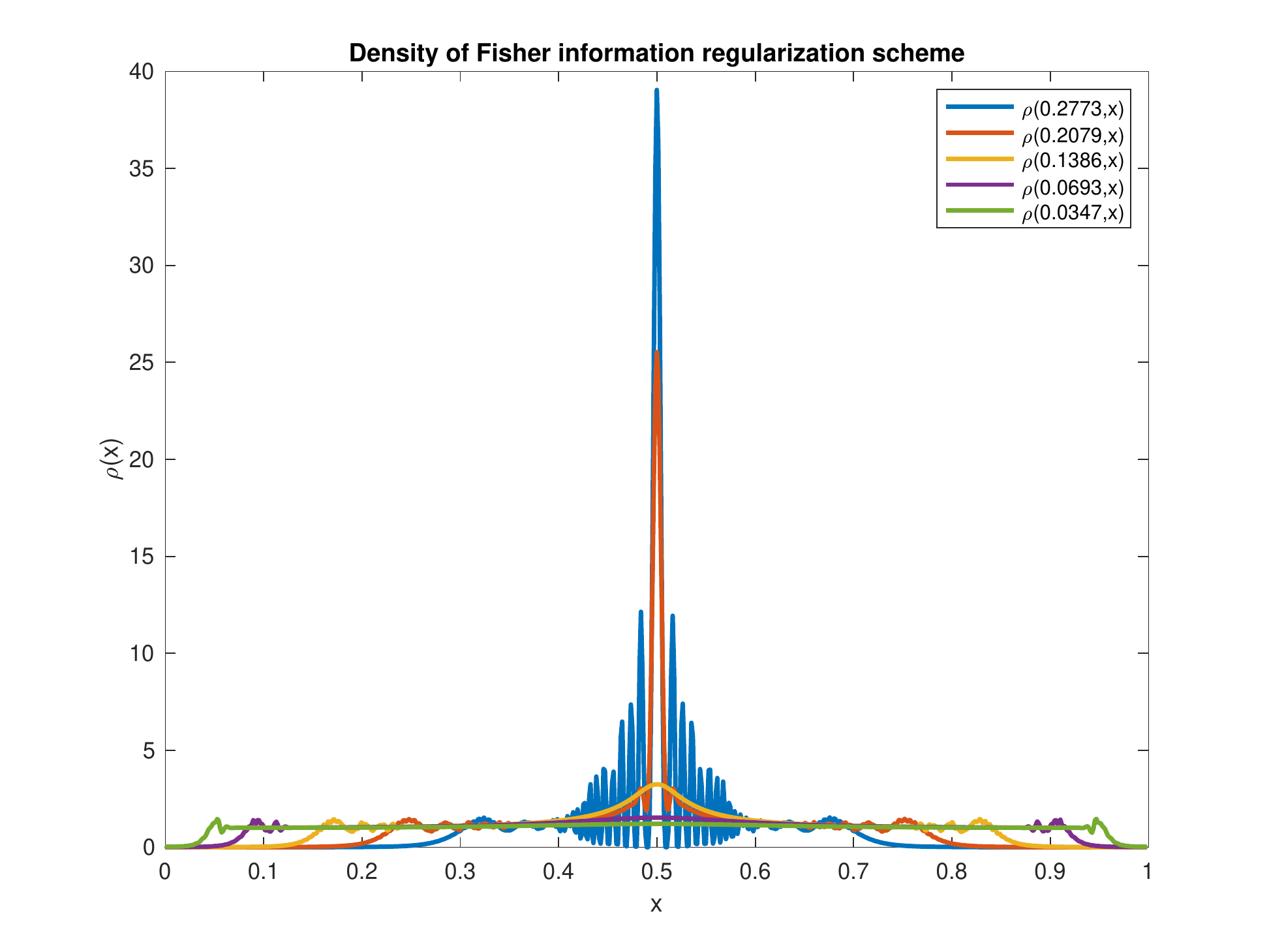}
\end{minipage}}
\subfigure{
\begin{minipage}[b]{0.31\linewidth}
\includegraphics[width=1.15\linewidth]{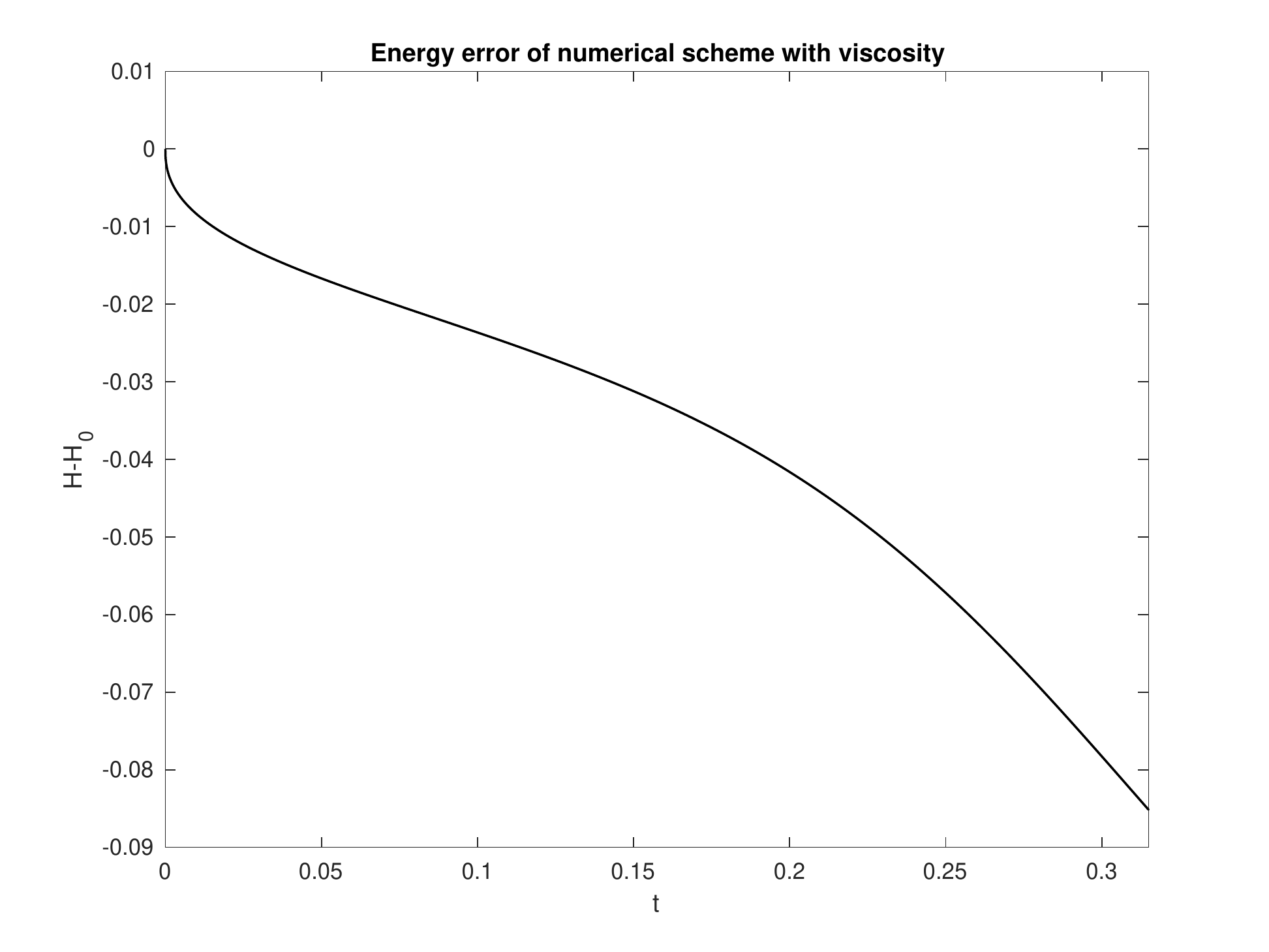}
\includegraphics[width=1.15\linewidth]{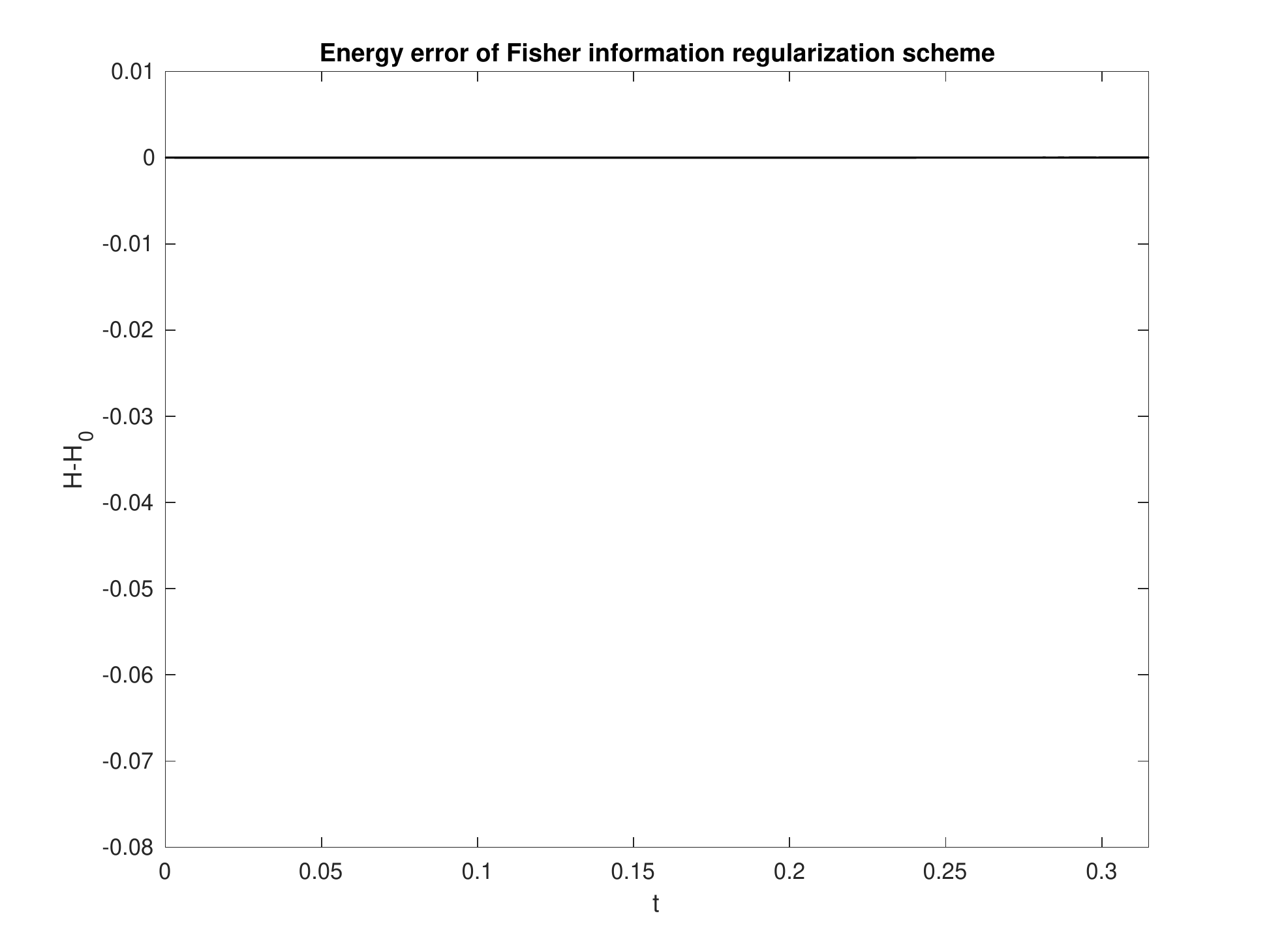}
\end{minipage}}

\centering 
\caption{Contour plot of $\rho(t,x)$ (left), snapshots of $\rho(t,x)$ at $t=(0.2773,0.2079,0.1386,0.0693,0.0347)$ (right)
and the energy error before $T=0.315$ (right) for the upwind scheme \eqref{scheme-1} with numerical viscosity (top)
and the Fisher information regularization symplectic scheme \eqref{Fis-mid} (bottom).}
\label{fig-0315b}
\end{figure}

\begin{figure}
\centering 
\subfigure {
\begin{minipage}[b]{0.31\linewidth}
\includegraphics[width=1.15\linewidth]{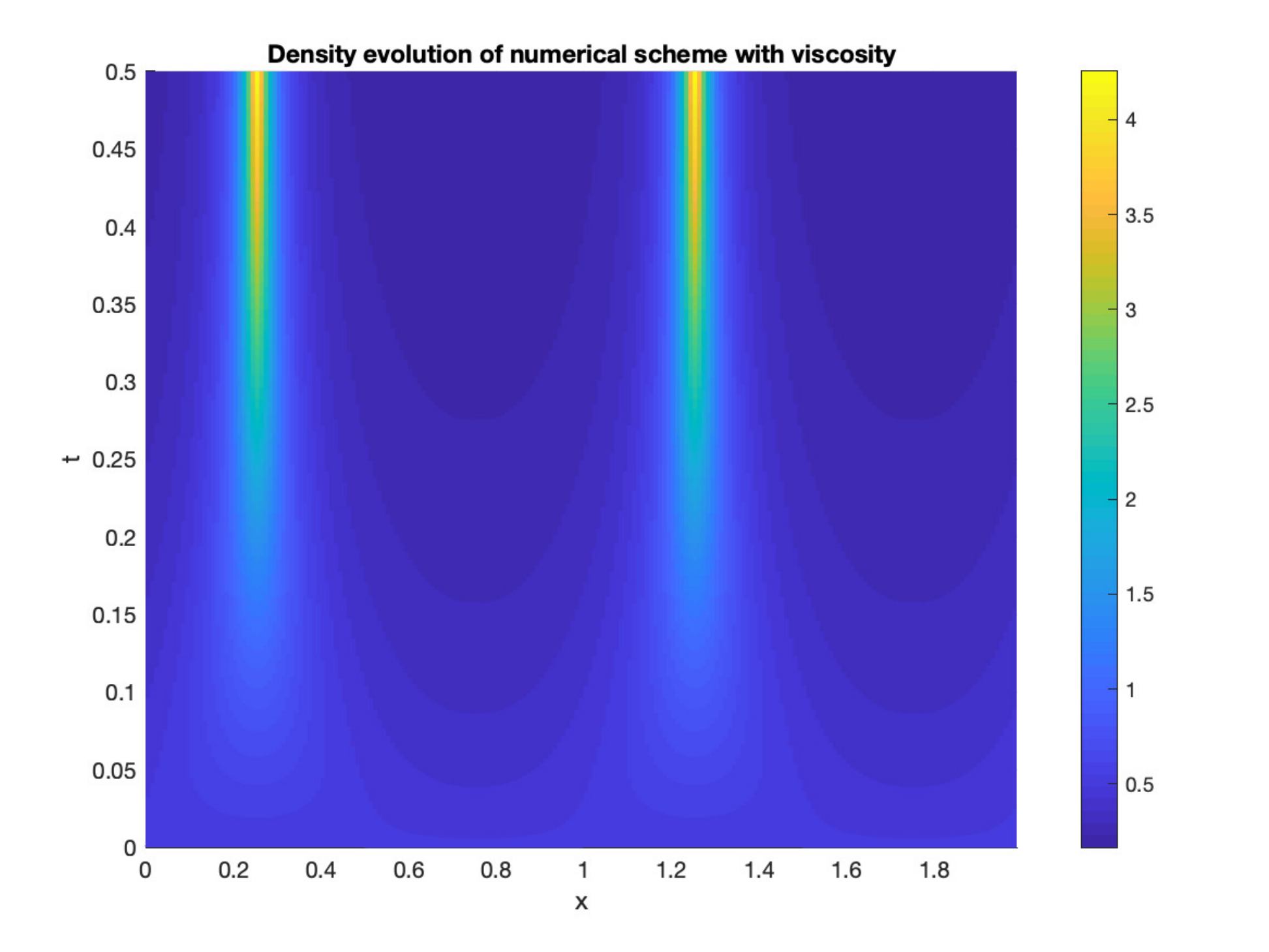}
\includegraphics[width=1.15\linewidth]{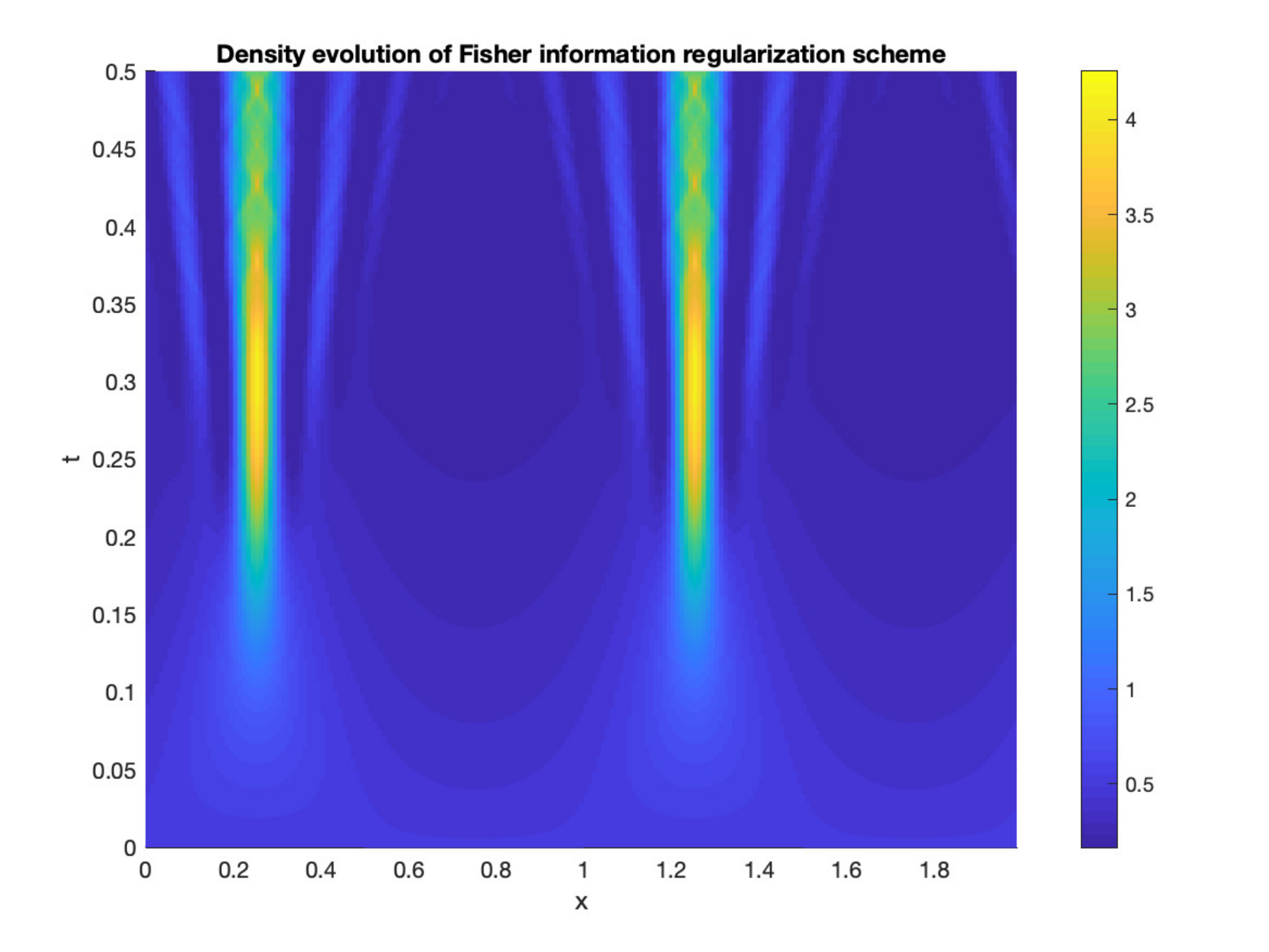}
\end{minipage}}
\subfigure{
\begin{minipage}[b]{0.31\linewidth}
\includegraphics[width=1.15\linewidth]{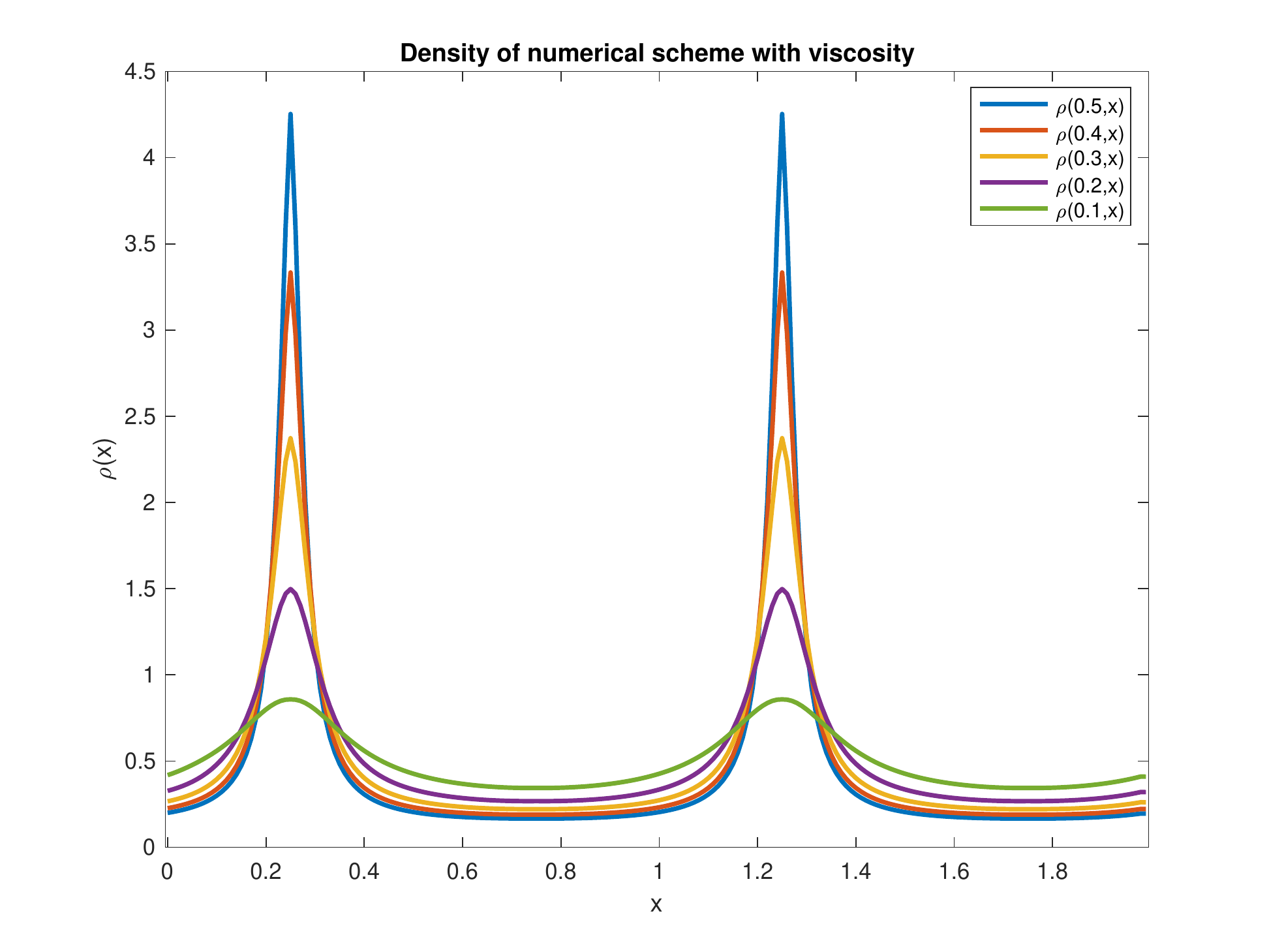}
\includegraphics[width=1.15\linewidth]{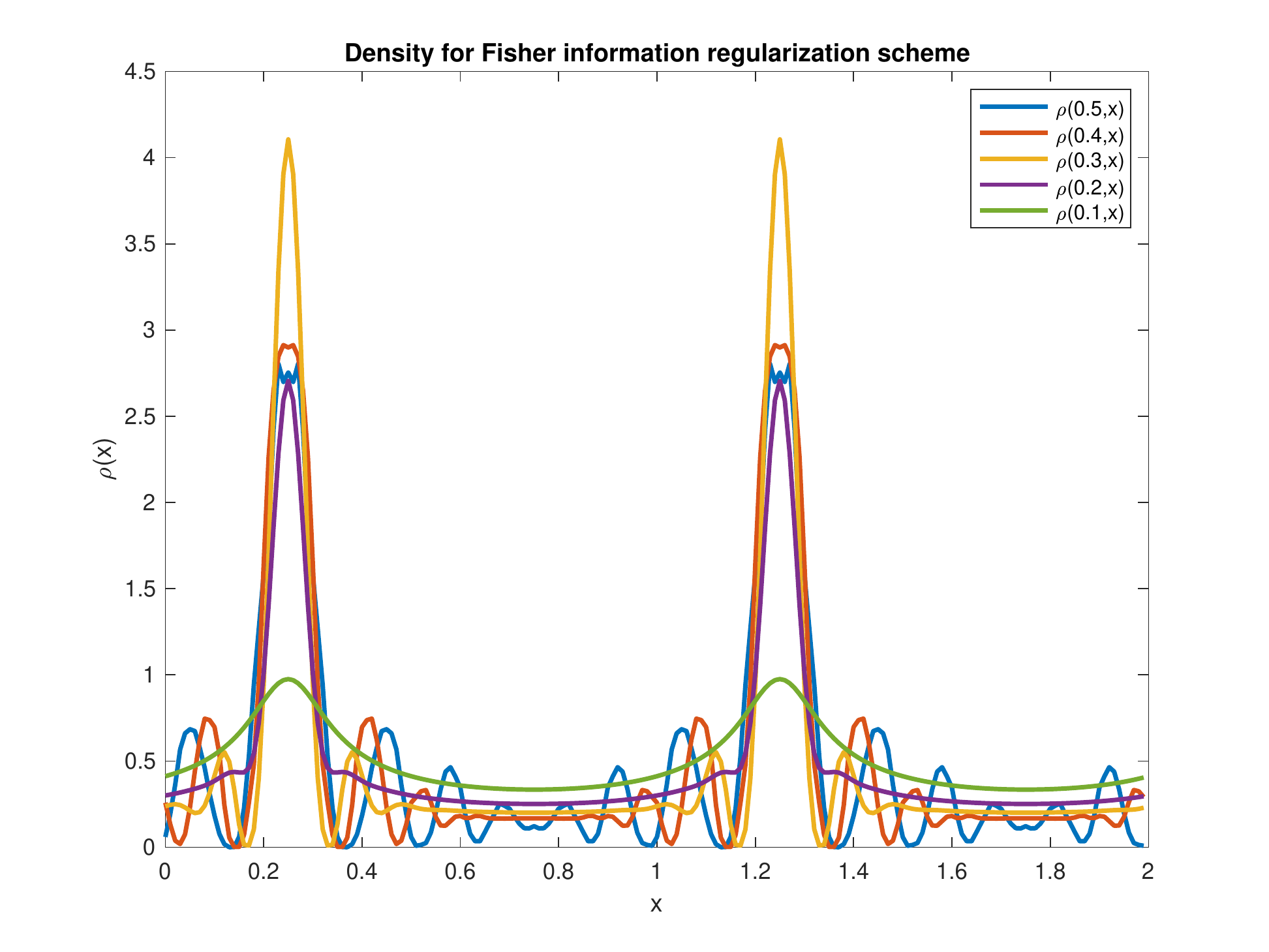}
\end{minipage}}
\subfigure{
\begin{minipage}[b]{0.31\linewidth}
\includegraphics[width=1.15\linewidth]{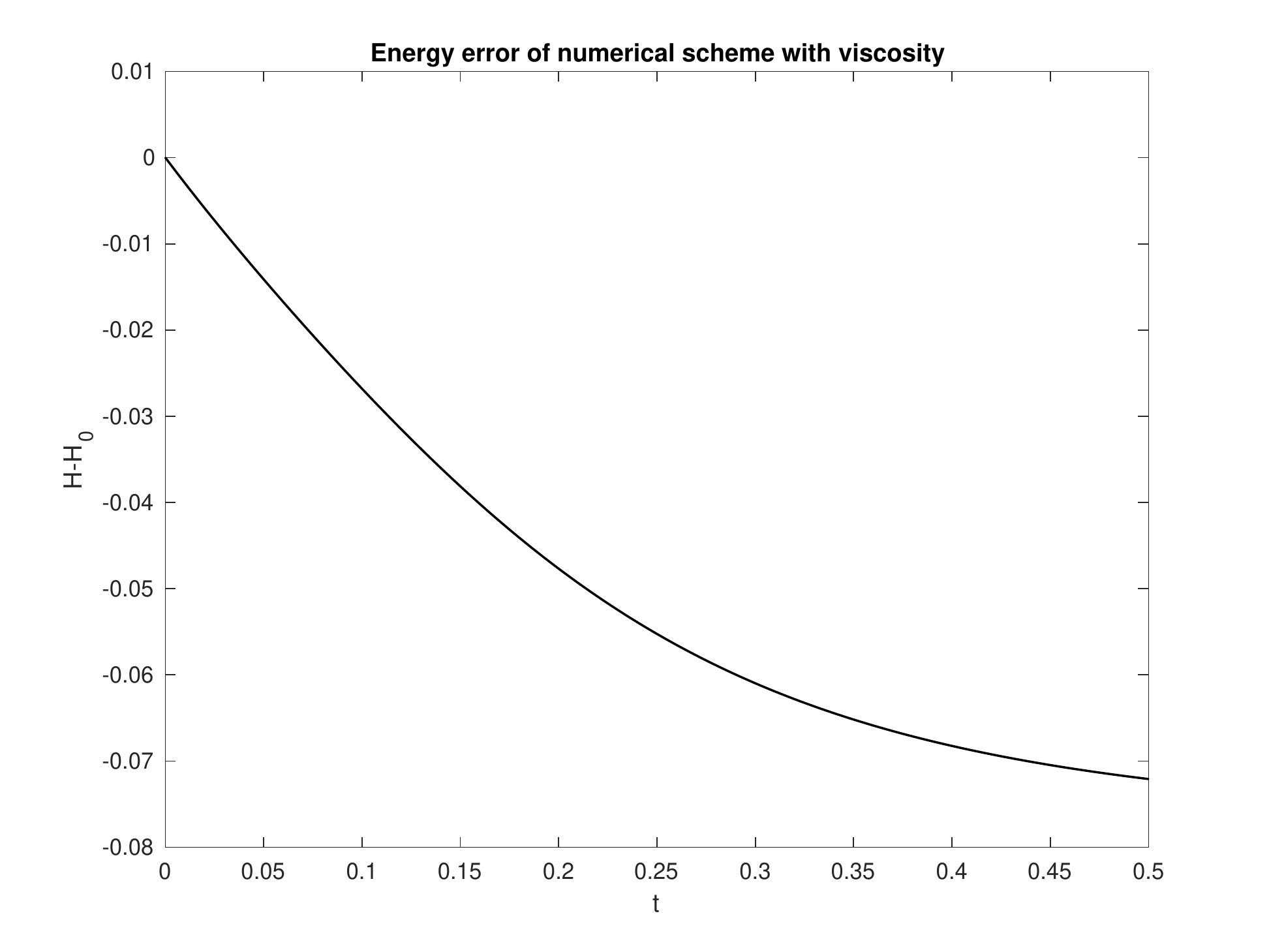}
\includegraphics[width=1.15\linewidth]{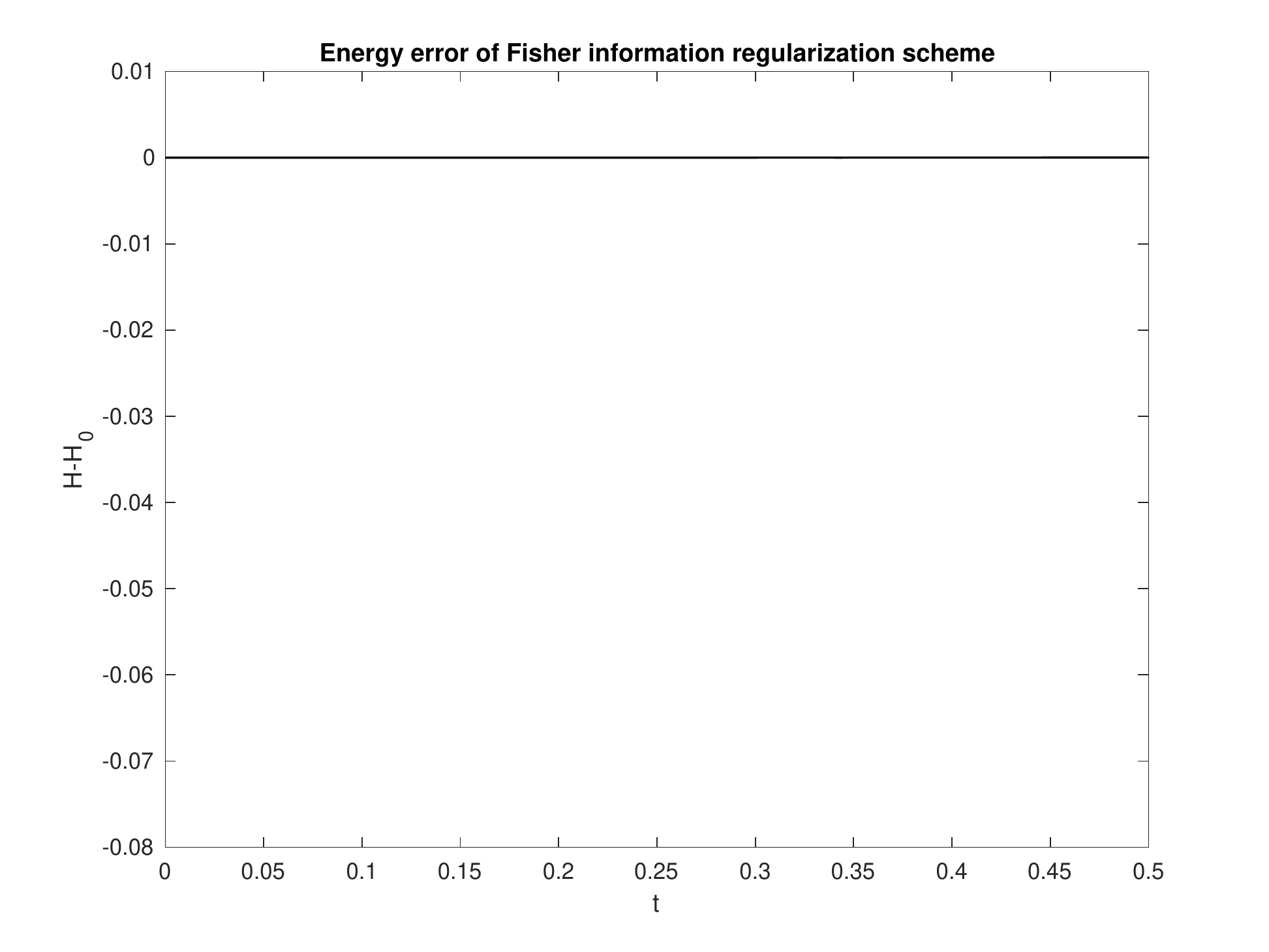}
\end{minipage}}

\centering 
\caption{Contour plot of $\rho(t,x)$ (left), snapshots of $\rho(t,x)$ at $t=(0.5,0.4,0.3,0.2,0.1)$ (right)
and  the energy error before $T=0.5$ (right) for the upwind scheme \eqref{scheme-1} with numerical viscosity (top)
and the Fisher information regularization symplectic scheme \eqref{Fis-mid} (bottom).}
\label{fig-05}
\end{figure}

Figure \ref{dtbeta} shows the relationship between $\beta$ and the largest time step-size $\tau$ in 
\eqref{Fis-mid} that still gives correct approximation to the solution. 
In this numerical test, we use $h=5\times 10^{-2},T=4,$ $\mathcal M=[0,1],$ $S_0(x)=\frac {\sin(\pi x)}{\pi}$, $\rho_0(x)=1$.
The parameter $\beta$ is chosen as five different values, $0.005788,0.005513,0.00525,0.005,0.00476,0.00454$. 
From \ref{dtbeta}, we can see that the relationship between $\frac {H_0}{\beta}$ and $\tau$ is very sensitive when 
$\frac {H_0}{\beta}$ is large.

\begin{figure}
\centering 
\includegraphics[width=2.8in,height=2.4in]{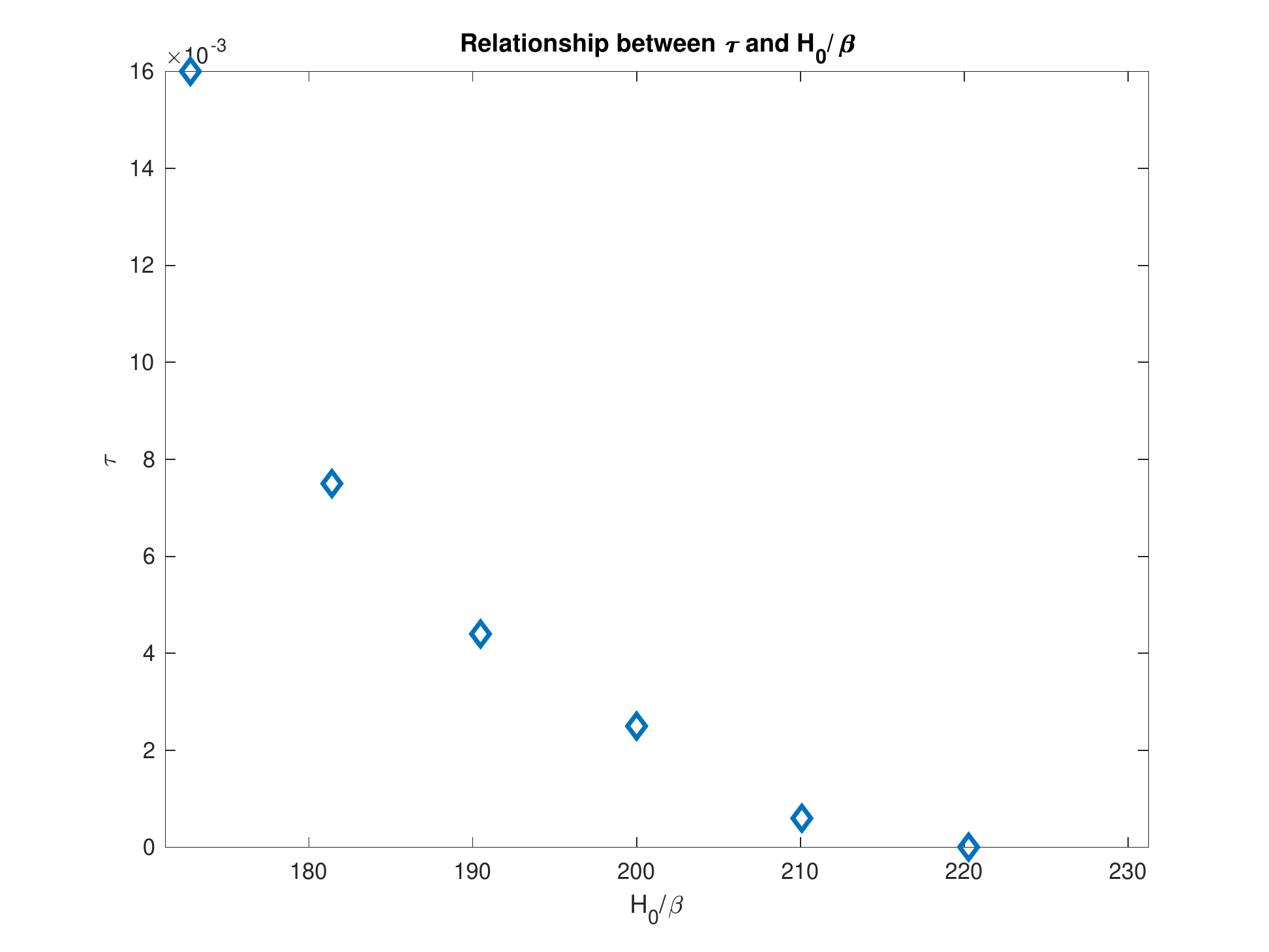}\\
\centering 
\caption{Relationship between $\frac {H_0}{\beta}$ and the largest time step-size $\tau$ that \eqref{Fis-mid} 
with parameter $\beta=0.005788,0.005513,0.00525,0.005,0.00476,0.00454$.}
\label{dtbeta}
\end{figure}

\begin{ex}\label{NLS}[Linear Madelung system]
This is the reformulation of \eqref{Madelung1} as Wasserstein-Hamiltonian system:
\begin{align*}
\partial_t \rho + \nabla\cdot( \rho \nabla S) &=0,\\
\partial_t S +\frac 12|\nabla S|^2+\beta \frac {\partial }{\partial \rho}I(\rho)&=0.
\end{align*}
\end{ex}

We use the scheme \eqref{Fis-mid} for a given $\beta>0$.
Figure \ref{fig-nls} shows the behaviors of $\rho$ and $S$, as well as the energy evolution.
Here for the evolution of $\rho$ and $S$, we choose $\beta=1$, $T=0.5,$ $\tau=10^{-3}$, $h=10^{-2},$ 
$S^0(x)=1/2\sin(2\pi x), \rho^0(x)=1$. 
We also plot the evolution of energy error $\mathcal H(t)-\mathcal H_0$ and mass error 
up to $T=400,$ which shows the good  longtime behaviors of the proposed scheme. 
 
\begin{figure}
\centering

\subfigure[]{
\includegraphics[width=2.4in,height=2.4in]{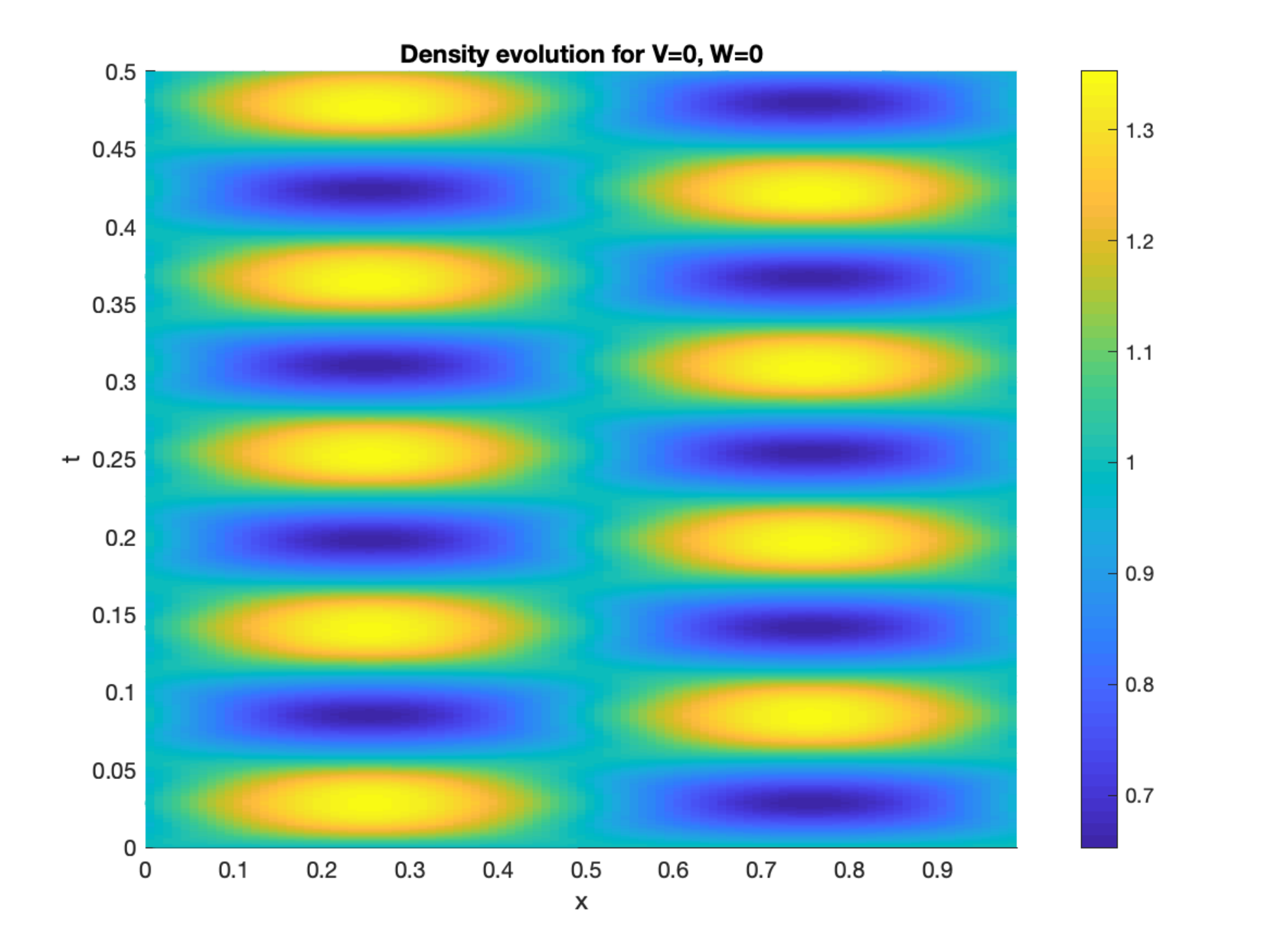}
\includegraphics[width=2.4in,height=2.4in]{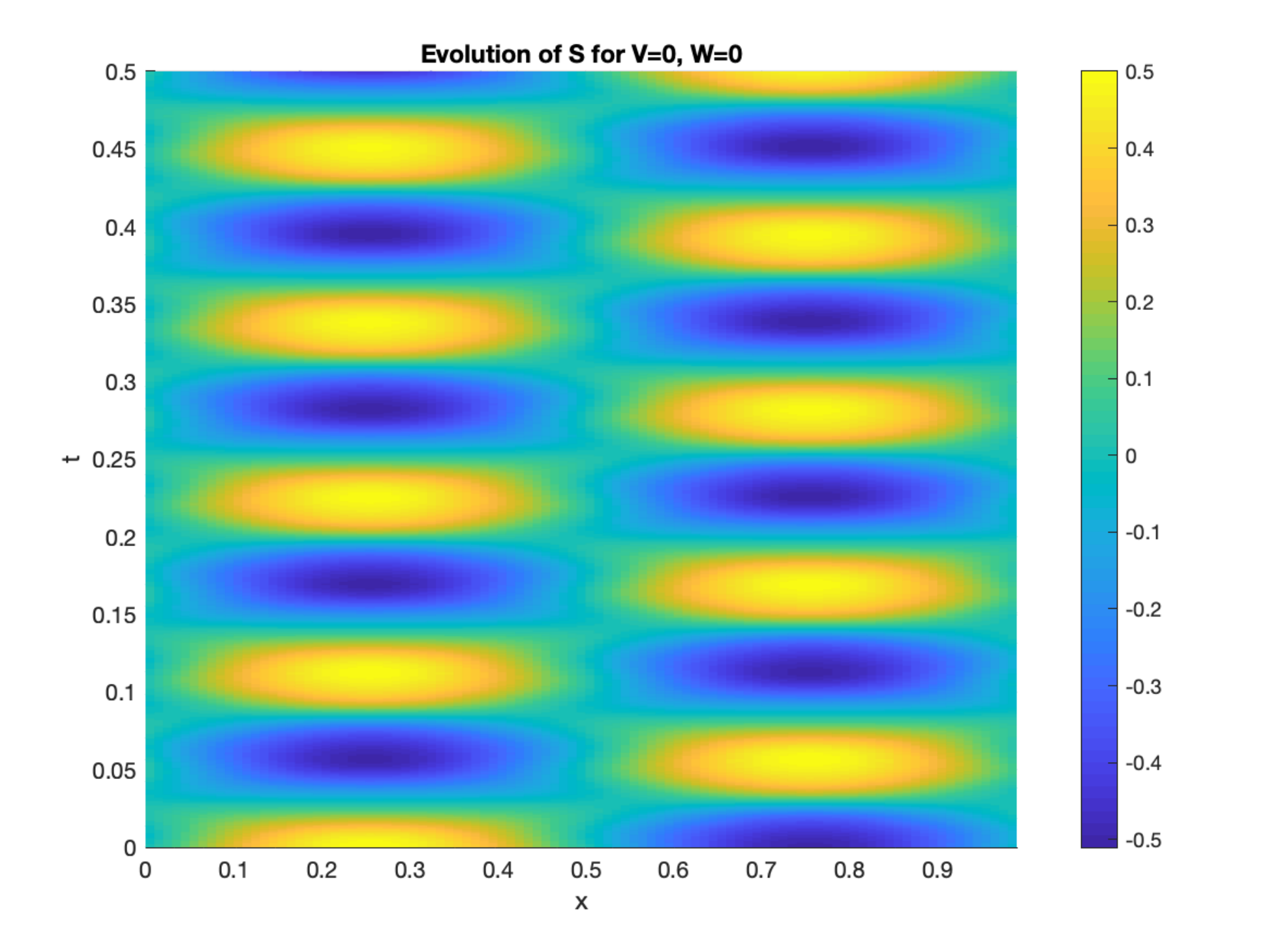}
}

\subfigure[]{
\includegraphics[width=2.4in,height=2.4in]{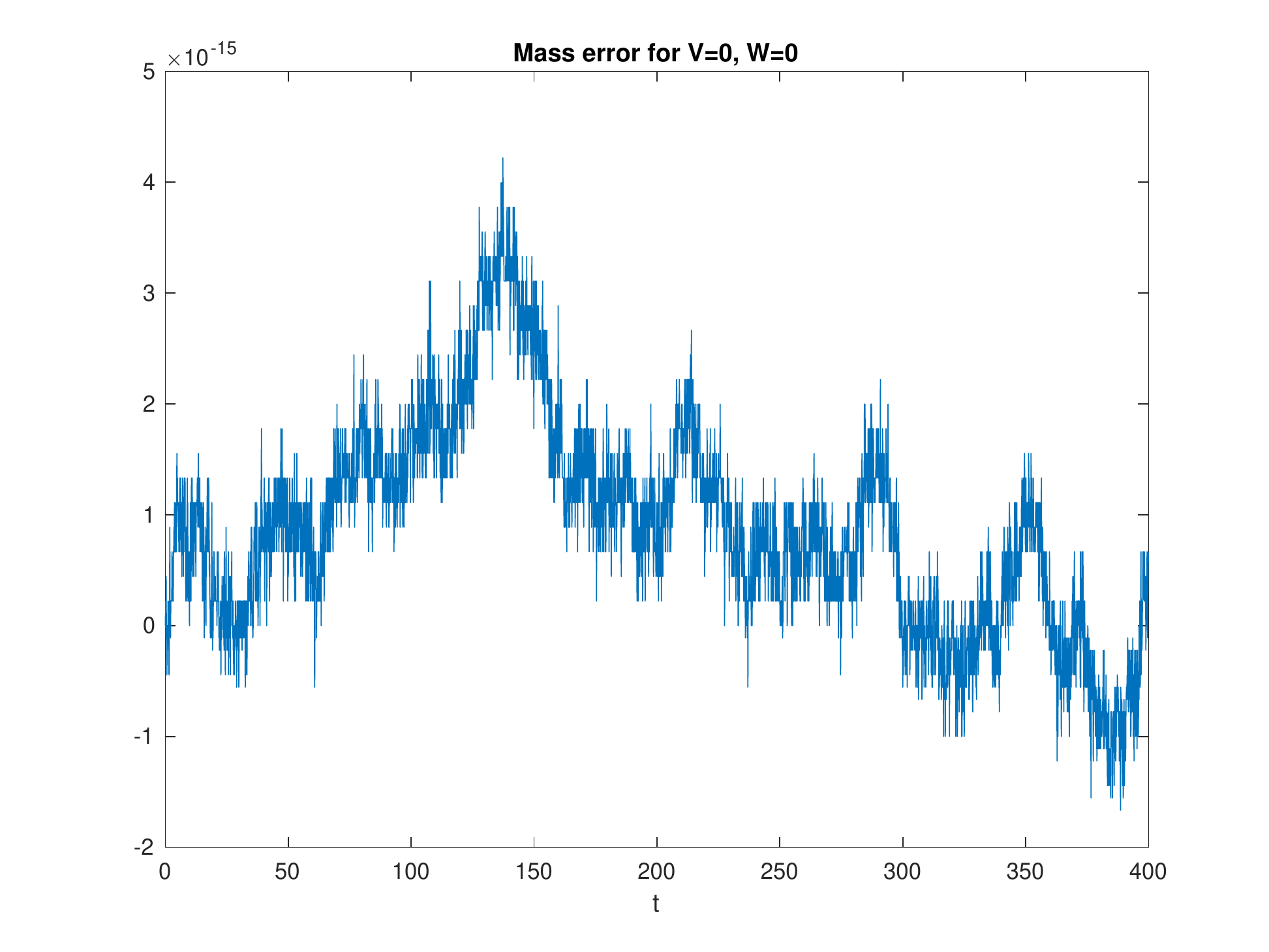}
\includegraphics[width=2.4in,height=2.4in]{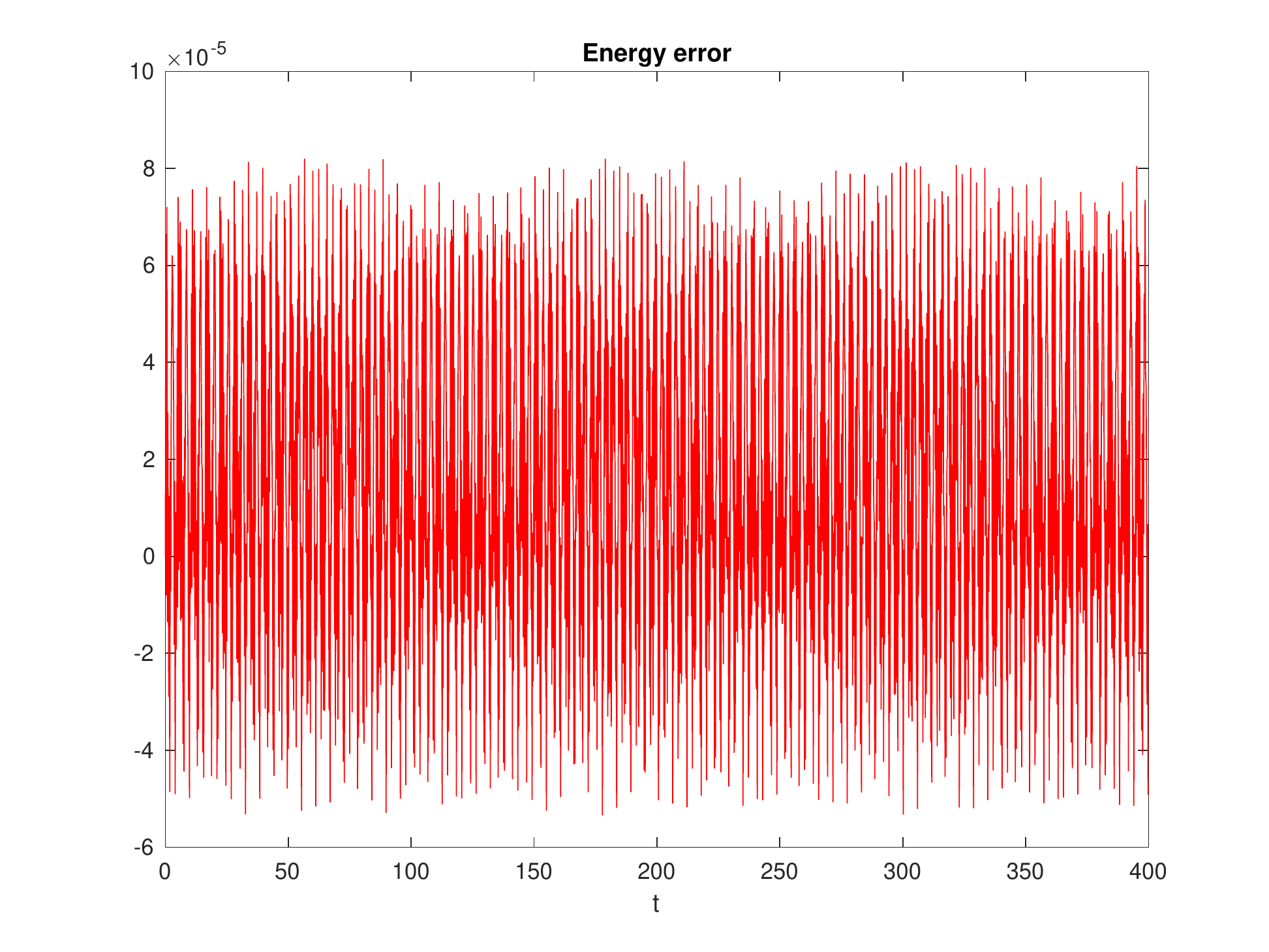}
}

\caption{The evolutions of $\rho$ and $S$ before $T=0.5$ (a), the mass conservation law and the energy error before $T=400$ (b). Note the extremely small scales in the plots.}
\label{fig-nls}
\end{figure}

\section{Acknowledgements}
The research is partially support by Georgia Tech Mathematics Application Portal (GT-MAP)  
and by research grants NSF DMS-1620345. DMS-1830225, and ONR N00014-18-1-2852. 

\bibliographystyle{plain}
\bibliography{bib}

\appendix
\section*{Appendix}
\setcounter{section}{1}

\begin{proof}[Proof of Proposition \ref{low-per}]
It suffices to find a constant $0<c<\frac 1N$ such that $\inf\limits_{0\le \min_i(\rho_i)\le c}I(\rho)\ge \frac {M_0}{\beta}$. 
Since the graph is finite, we have that 
$$\inf\limits_{0\le \min_i I(\rho_i)\le c}I(\rho)=\min_{i\le N}\inf\limits_{0\le \rho_i\le c} I(\rho)$$
Due to convexity of $I(\rho)$ on $0\le \rho_i\le c$ for a fixed $i\le N$, 
and the fact that $I(\rho)$ approaches $\infty$ when $\rho$ approaches the boundary of $\mathcal P_o(G)$, 
$I(\rho)$ takes the minimum at the boundary, i.e., 
$\inf\limits_{0\le \rho_i\le c} I(\rho)=\inf\limits_{\rho_i= c} I(\rho)$ on $\mathcal P_o(G)$.
Because of the periodic boundary condition, without loss of generality we can assume that $\rho_1=c$. 
By calculating the Hessian matrix of $I(\rho)$, we get for any $\sigma\neq 0$,
\begin{align*}
\sigma^T \text{Hess} I(\rho)\sigma
&=\sum_{i=3}^{N-1}(\frac 1{\rho_i^2}(\rho_i+\rho_{i+1}+\rho_{i-1}))\sigma_i^2\\
&\quad+\sum_{i=3}^{N-1}(\frac 1{\rho_i\rho_{i+1}} (\rho_{i}+\rho_{i+1})\sigma_{i}\sigma_{i+1} +\frac 1{\rho_i\rho_{i-1}} (\rho_{i}+\rho_{i-1})\sigma_{i}\sigma_{i-1})\\
&\quad+ \frac 1{\rho_2^2}(2\rho_i+\rho_3+c)+\frac 1{\rho_2\rho_3}(\rho_2+\rho_3)\\
&\quad +\frac 1{\rho_N^2}(2\rho_N+\rho_{N-1}+c)+\frac 1{\rho_N\rho_{N-1}}(\rho_N+\rho_{N-1})\\
&=\sum_{i=2}^{N-1}(\rho_{i}+\rho_{i+1})(\frac {\sigma_{i}}{\rho_{i}}-\frac {\sigma_{i+1}}{\rho_{i+1}})^2+
\frac 1{\rho_2^2}(\rho_2+c)\sigma_2^2+\frac 1{\rho_N^2}(\rho_N+c)\sigma_N^2
> 0,
\end{align*}
which implies strict convexity of $I(c,\cdot)$ on $\sum_{i=2}^N\rho_i=1-c$. 
Using the Lagrange multiplier technique on 
$I(c,\rho_2,\cdots,\rho_N)-\lambda(\sum_{i=2}^N\rho_i-1+c)$, we get that the unique minimum point satisfies  
\begin{equation}\begin{split}\label{Lag-mul-per}
\phi(\frac {c}{\rho_2})+\phi(\frac {\rho_3}{\rho_2})&=\lambda,\\
\phi(\frac {\rho_{i-1}}{\rho_i})+\phi(\frac {\rho_{i+1}}{\rho_i})&=\lambda, \; \text{if}\; 3\le i \le N-1,\\
\phi(\frac {\rho_{N-1}}{\rho_N})+\phi(\frac {c}{\rho_N})&=\lambda,
\end{split}\end{equation}
where $\phi(t)=1-t-\log(t)$.
We claim that $\rho_{N-i+1}=\rho_{i+1}$, for $i=1,\cdots,\frac {N-1}2,$  if $N-1$ is even number. 
When $N-1$ is odd, we have $\rho_{N-i+1}=\rho_{i+1}$, for $i=1,\cdots, [\frac {N-1}2]$, where $[s]$ is 
the largest integer smaller than or equal to $s\in \mathbb R$.

To prove this claim, it suffices to show that $\rho_2=\rho_N$. Assume that $\rho_2>\rho_N$,
Due to the monotonicity of $\phi$, we have $\frac c{\rho_2}<\frac c{\rho_N}$,
\begin{align}\label{ite-per-0}
\frac {\rho_3}{\rho_2}>\frac {\rho_{N-1}}{\rho_N},
\frac {\rho_4}{\rho_3}>\frac {\rho_{N-2}}{\rho_{N-1}},
\cdots,
\frac {\rho_{i+2}}{\rho_{i+1}}>\frac {\rho_{N-i}}{\rho_{N-i+1}}, \;\text{for}\; 1\le i\le [\frac {N-1}2].
\end{align}
If $N-1$ is even, we obtain that
\begin{align*}
\phi(\frac {\rho_{\frac {N-1}2+2}}{\rho_{\frac {N-1}2+1}})
<\phi(\frac {\rho_{\frac {N-1}2+1}}{\rho_{\frac {N-1}2+2}}),
\end{align*}
which leads to $\frac {\rho_{\frac {N-1}2+1}}{\rho_{\frac {N-1}2+2}}<\frac {\rho_{\frac {N-1}2+2}}{\rho_{\frac {N-1}2+1}},$ 
i.e., $\rho_{\frac {N-1}2+2}>\rho_{\frac {N-1}2+1}.$
Thus, we can conclude from \eqref{ite-per-0} that 
\begin{align*}
\frac {\rho_N}{\rho_2}>\frac {\rho_{N-1}}{\rho_3}
>\cdots>\frac {\rho_{\frac {N-1}2+2}}{\rho_{\frac {N-1}2+1}}>1, 
\end{align*}
which contradicts the assumption $\rho_2>\rho_N$. 
If $N-1$ is odd, similar arguments yield that 
\begin{align*}
\phi(\frac {\rho_{[\frac {N-1}2]+2}}{\rho_{[\frac {N-1}2]+1}})
<\phi(\frac {\rho_{[\frac {N-1}2]+2}}{\rho_{[\frac {N-1}2]+3}}),
\end{align*}
which implies that $\rho_{[\frac {N-1}2]+3}>\rho_{[\frac {N-1}2]+1}$.
Thus from \eqref{ite-per-0}, we have that
$$\frac {\rho_N}{\rho_2}>\frac {\rho_{N-1}}{\rho_3}>\cdots>\frac {\rho_{[\frac {N-1}2]+3}}{\rho_{[\frac {N-1}2]+1}}>1,$$ 
which contradicts the assumption $\rho_2>\rho_N$.
One can show that $\rho_2<\rho_N$ is also impossible by the same arguments. As a consequence, 
$\rho_2=\rho_N$. By further using   \eqref{Lag-mul-per}, we immediately get 
$\rho_{N-i+1}=\rho_{i+1}$, for $i=1,\cdots, [\frac {N-1}2].$ 

Now, we are going to show that the extreme point possesses the monotonicity  along the path starting from $a_1$. 
Indeed, $\rho_i$ is increasing when $d_{1,i+1}$ is increasing for 
$i\le [\frac {N-1}2]$ if $N$ is odd and  for $i\le [\frac {N-1}2]+1$ if $N$ is even. 
We use Figure \ref{fig-node} to illustrate these two different cases.

\begin{figure}
\centering

\includegraphics[width=2.45in,height=2.2in]{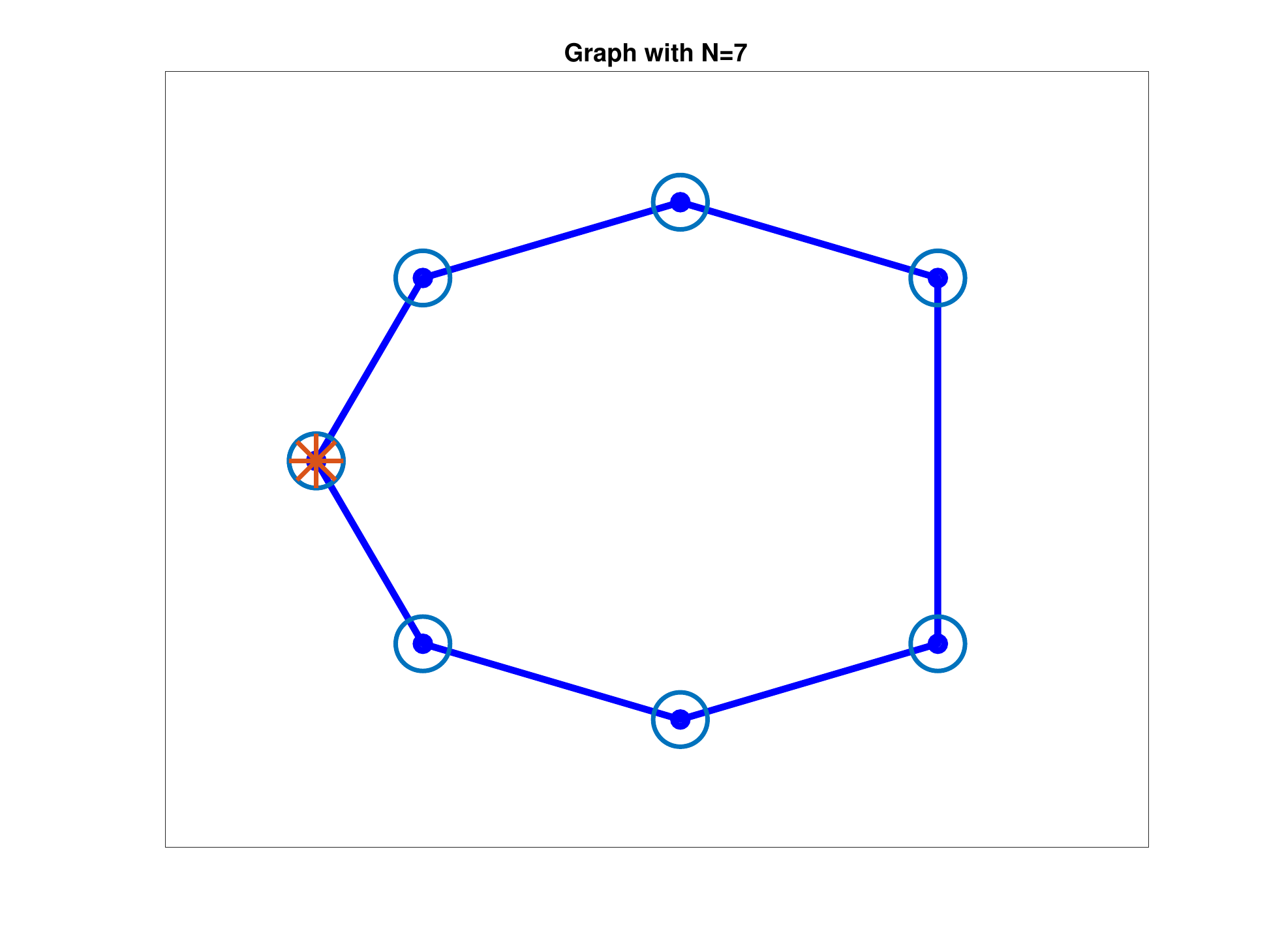}
\includegraphics[width=2.45in,height=2.2in]{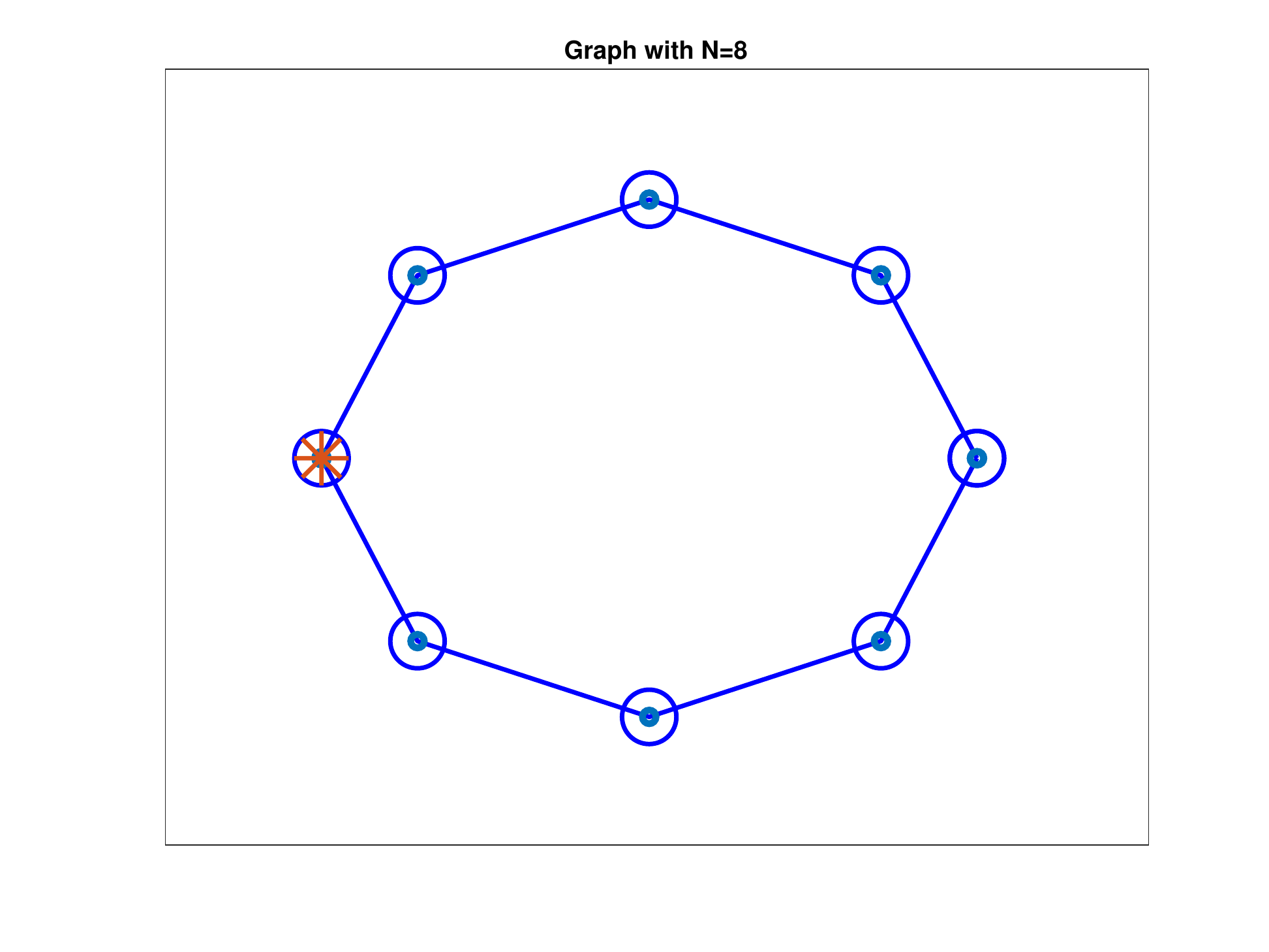}

\caption{The picture of the graph with $N=7$ (left) and with $N=8$ (right), where the red node represents $v_1$.}
\label{fig-node}
\end{figure}

Step 1: $\lambda>0$. 
Since $\lambda=0$ if and only if $\rho_i=\frac 1N$, then $I(\rho)=0$ which contradicts the
fact that $\inf_{\rho}I(\rho)>0$.
Assume that $\lambda< 0$. Then \eqref{Lag-mul-per}, together with the symmetry 
$\rho_{i+1}=\rho_{N-i+1}, i=1,\cdots,[\frac {N-1}{2}]$,  implies that 
when $N-1$ is even, it holds that 
\begin{equation}\label{Lag-mul-per-2}\begin{split}
\phi(\frac {\rho_{i-1}}{\rho_i})+\phi(\frac {\rho_{i+1}}{\rho_i})&=\lambda, \; \text{if}\; 2\le i \le \frac {N-1}2-1,\\
\phi(\frac {\rho_{\frac {N-1}2}}{\rho_{\frac {N-1}2+1}})&=\lambda.
\end{split}\end{equation}
Since $\lambda< 0$, we obtain that 
\begin{align*}
\rho_{\frac {N-1}2+1}<{\rho_{\frac {N-1}2}}<\cdots< \rho_2<\rho_1=c,
\end{align*}
which contradicts the fact that $\sum_{i=2}^{N}\rho_i=1-c$. 
When $N-1$ is odd, then \eqref{Lag-mul-per} and symmetry of $\rho_i$ imply that  
\begin{equation}\label{Lag-mul-per-3}\begin{split}
\phi(\frac {\rho_{i-1}}{\rho_i})+\phi(\frac {\rho_{i+1}}{\rho_i})&=\lambda, \; \text{if}\; 2\le i \le [\frac {N-1}2],\\
2\phi(\frac {\rho_{[\frac {N-1}2]+1}}{\rho_{[\frac {N-1}2]+2}})&=\lambda.
\end{split}\end{equation}
Then we get $\rho_{[\frac {N-1}2]+2}<{\rho_{[\frac {N-1}2]+1}}<\cdots< \rho_2<\rho_1=c,$ which is also not possible.
Thus it holds that $\lambda>0$. This indicates that 
$$\rho_{[\frac {N-1}2]+2}>{\rho_{[\frac {N-1}2]+1}}>\cdots>\rho_2>\rho_1=c.$$

Step 2: $\frac {\rho_{i+1}}{\rho_{i}}$ is strictly decreasing. 
If $N-1$ is even, $\frac {\rho_{i+1}}{\rho_{i}}$ is strictly decreasing for $1\le i\le [\frac {N-1}2]$.
According to \eqref{Lag-mul-per-2}, it holds that 
\begin{align*}
\phi(\frac {\rho_{\frac {N-1}2-1}}{\rho_{\frac {N-1}2}})&=\lambda-\phi(\frac {\rho_{\frac {N-1}2+1}}{\rho_{\frac {N-1}2}})=\phi(\frac {\rho_{\frac {N-1}2}}{\rho_{\frac {N-1}2+1}})-\phi(\frac {\rho_{\frac {N-1}2+1}}{\rho_{\frac {N-1}2}}),\\
\phi(\frac {\rho_{\frac {N-1}2-i-1}}{\rho_{\frac {N-1}2-i}})&=\lambda-\phi(\frac {\rho_{\frac {N-1}2-i+1}}{\rho_{\frac {N-1}2-i}})=\phi(\frac {\rho_{\frac {N-1}2-i}}{\rho_{\frac {N-1}2-i+1}})-\phi(\frac {\rho_{\frac {N-1}2-i+1}}{\rho_{\frac {N-1}2-i}})+\phi(\frac {\rho_{\frac {N-1}2-i+2}}{\rho_{\frac {N-1}2-i+1}}), 
\end{align*}
where $i=1,\cdots,\frac {N-5}2$.
The monotonicity of $\rho_i$, $i\le \frac {N-1}2$, together with $\lambda>0$, leads to 
\begin{align*}
\phi(\frac {\rho_{\frac {N-1}2-i-1}}{\rho_{\frac {N-1}2-i}})&>\phi(\frac {\rho_{\frac {N-1}2-i}}{\rho_{\frac {N-1}2-i+1}}), \;\text{for}\; i=0,\cdots,\frac {N-5}2.
\end{align*}
If $N-1$ is odd, $\frac {\rho_{i+1}}{\rho_{i}}$ is strictly decreasing for $1\le i\le [\frac {N-1}2]+1$.
From \eqref{Lag-mul-per-3}, it follows that 
\begin{align*}
\phi(\frac {\rho_{[\frac {N-1}2]}}{\rho_{[\frac {N-1}2]+1}})&=\lambda-\phi(\frac {\rho_{[\frac {N-1}2]+2}}{\rho_{[\frac {N-1}2]+1}})=2\phi(\frac {\rho_{[\frac {N-1}2]+1}}{\rho_{\frac {[N-1}2]+2}})-\phi(\frac {\rho_{[\frac {N-1}2]+2}}{\rho_{[\frac {N-1}2]+1}}),\\
\phi(\frac {\rho_{[\frac {N-1}2]-i-1}}{\rho_{[\frac {N-1}2]-i}})&=\lambda-\phi(\frac {\rho_{[\frac {N-1}2]-i+1}}{\rho_{[\frac {N-1}2]-i}})=\phi(\frac {\rho_{[\frac {N-1}2]-i}}{\rho_{[\frac {N-1}2]-i+1}})-\phi(\frac {\rho_{[\frac {N-1}2]-i+1}}{\rho_{[\frac {N-1}2]-i}})+\phi(\frac {\rho_{[\frac {N-1}2]-i+2}}{\rho_{[\frac {N-1}2]-i+1}}),
\end{align*}
where $i=0,\cdots,[\frac {N-1}2]+1$.
From the monotonicity of $\phi$, it follows that $\frac {\rho_{i+1}}{\rho_{i}}$ is strictly decreasing for $1\le i\le [\frac {N-1}2]$. 

Step 3: Lower bound for $\frac {\rho_{i+1}}{\rho_{i}}, i=1,\cdots, [\frac {N-1}2].$
We first deal with the case that $N-1$ is even. Due to monotonicity of $\frac {\rho_{i+1}}{\rho_{i}}$, its 
minimum is $k:=\frac {\rho_{[\frac {N-1}2]+1}}{\rho_{[\frac {N-1}2]}}$. 
Since $\sum_{i}\rho_i=1$, we have 
\begin{align*}
c+2\sum_{i=2}^{[\frac {N-1}2]}\rho_i&=c(1+2\sum_{i=2}^{[\frac {N-1}2]}\frac {\rho_i} c)=1.  
\end{align*}
To find a lower bound of $\frac {\rho_{i+1}}{\rho_{i}}$, it suffices to find an upper bound such that 
\begin{align*}
1+\sum_{i=2}^{[\frac {N-1}2]}k^{i-1}=\frac {k^{[\frac {N-1}2]}-1}{k-1}<\frac {1+c}{2c}.
\end{align*} 
Let $k\le (\frac {1-c}{2c[\frac {N-1} 2]})^{\frac 1{[\frac {N-1}2]}}$. Then it holds that 
\begin{align*}
\sum_{i=2}^{[\frac {N-1}2]+1}k^{i-1}\le  [\frac {N-1}2]k^{[\frac {N-1}2]}\le \frac {1-c}{2c}.
\end{align*}
Finally, we get that 
\begin{align*}
\inf_{\rho_1=c} I(\rho)&=2 \sum_{i=1}^{[\frac {N-1}2]}(\log(\rho_{i}^*)-\log(\rho_{i+1}^*))(\rho_{i}^*-\rho_{i+1}^*)\\
&\ge {2}(\log(\frac {\rho_{[\frac {N-1}2]+1}^*}{\rho_{[\frac {N-1}2]}^*}))(\rho_{[\frac {N-1}2]+1}^*-c).
\end{align*}
Since there exists at least $\rho_{j}^*, j\le N$ such that $\rho_j^*>\frac {1-c}{N-1}$, 
thus it holds that
\begin{align}\label{low-est}
\inf_{\rho_1=c} I(\rho)
&\ge  {2} (\log(\frac {\rho_{[\frac {N-1}2]+1}^*}{\rho_{[\frac {N-1}2]}^*}))(\frac {1-c}{N-1}-c) \ge 
2\log(k)(\frac {1-c}{N-1}-c)\ge \frac {M}{\beta}.
\end{align}
Now, we are able to show the desired lower bound estimate.
If there exists $\frac 12\min_i \rho_i(0) N$
$ \le \alpha<\min_i \rho_i(0) N, c=\alpha \frac {1}N$ such that $\inf_{\rho_1=c}I(\rho)\ge\frac {\mathcal H_0}{\beta}$, then  
$$\sup_{t\ge0}\min_{i\le N}\rho_i(t)\ge \frac 12\min_i \rho_i(0).$$
Otherwise, $c<\frac 1N\alpha,$ for $\alpha\le \frac 12\min_{i\le N}\rho_i(0) N$. 
From the estimate \eqref{low-est}, it follows that if 
$c< \frac 1{1+2[\frac {N-1}2]\exp(\frac {M_0(N-1)([\frac {N-1}2])}{2\beta(1-\alpha)})}$, then 
$\inf_{\rho_1=c}I(\rho)>\frac {M}{\beta}.$
Based on the above estimates, we have the following lower bound for $\rho$,
\begin{align*}
\sup_{t\ge0}\min_{i\le N}\rho_i(t)\ge 
\frac 1{1+2[\frac {N-1}2]\exp(\frac {M(N-1)[\frac {N-1}2]}{2\beta(1-\alpha)})}\ge 
\frac 1{1+2[\frac {N-1}2]\exp(\frac { M(N-1)[\frac {N-1}2]}{\beta})}.
\end{align*}
Thus, it holds that 
\begin{align*}
\sup_{t\ge0}\min_{i\le N}\rho_i(t)\ge \min( \frac 12\min_i \rho_i(0), 
\frac 1{1+2[\frac {N-1}2]\exp(\frac {M(N-1)[\frac {N-1}2]}{\beta})}).
\end{align*}
Similar arguments yield the estimate when $N-1$ is odd, 
\begin{align*}
\sup_{t\ge0}\min_{i\le N}\rho_i(t)\ge \min( \frac 12\min_i \rho_i(0), \frac 1{1+2([\frac {N-1}2]+1)\exp(\frac {M(N-1)([\frac {N-1}2]+1)}{\beta})}).
\end{align*}
\end{proof}

\bigskip
\begin{proof}[Proof of Proposition \ref{low-new}]
We use an induction argument and similar techniques to those used in the proof of Proposition \ref{low-per}. 
Like the proof of Proposition \ref{low-per}, it suffices to find the largest  $0<c<\frac 1N$ such that 
$\inf\limits_{0\le \min_i(\rho_i)\le c}I(\rho)\ge \frac {M}{\beta}$. Since the graph is finite and $I(\rho)$
is convex, we have that 
$$\inf\limits_{0\le \min_i I(\rho_i)\le c}I(\rho)=\min_{i\le N}\inf\limits_{\rho_i=c} I(\rho).$$
When $N=3$, then the graph only has two boundary nodes and   
we only need to consider the case that $\rho_1=c$ and $\rho_2=c$, due to the symmetry on boundary nodes. 
When $\rho_1=c$, the Lagrange multiplier method yields that the extreme point satisfies 
\begin{align*}
\phi(\frac {c}{\rho_2})+\phi(\frac {\rho_3}{\rho_2})&=\lambda,\quad
\phi(\frac {\rho_2}{\rho_3})=\lambda, \quad \phi(t)=1-t-\log(t)\ .
\end{align*}
Then it is not hard to get that $\lambda>0$, $\rho_3>\rho_2>c$ and $\frac {c}{\rho_2}<\frac {\rho_2}{\rho_3}$.
When $\rho_2=c$, the Lagrange multiplier method yields that the extreme point satisfies 
\begin{align*}
\phi(\frac {c}{\rho_1})=\lambda,\quad
\phi(\frac {c}{\rho_3})=\lambda,
\end{align*}
and so we obtain that $\lambda>0$, $\rho_3>c, \rho_1>c$.  From these, 
similarly to the proof of Proposition \ref{low-per}, we obtain 
$$\sup_{t}\min_{i}\rho_{i}(t)\ge 
\min\Big(\frac 12 \min_i\rho_i(0),\frac {1}{1+2\exp(4\frac {M}{\beta})}\Big),$$

Now we proceed with the induction steps.
Assume that for the graph with $N-1$ nodes, if $\inf_{j\le N-1}\inf_{|\rho_j|\le c}I(\rho)= \inf_{\rho_i=c}I(\rho)$ for 
some $i,$ then we get $\lambda>0$ in the Lagrange multiplier technique, and that   
for any path $a_{l_0}a_{l_1}a_{l_2} \cdots a_{l_m}$, $m\le N-1$, starting from $a_{l_0}=a_i$ to a boundary point 
$a_{l_m}$, the probability density $\rho_{l_j}$, $0\le j\le m$ is increasing and 
$\frac {\rho_{l_{j+1}}}{\rho_{l_{j}}}, 0\le j\le m-1$, is decreasing.
We are going to prove that the above statement also holds for the graph with $N$ nodes. 
Let $\inf_{j\le N}\inf_{|\rho_j|\le c}I(\rho)= \inf_{\rho_i=c}I(\rho)$ for some $i$. 
Then either $a_i$ is a boundary vertex of the the graph, or $a_i$ is an interior vertex of the graph.

Case 1: $a_i$ is an interior node of the graph. 
Assume that the numbers of edges connecting to $a_i$ is $n_i$. 
By using the Lagrange multiplier method and taking the partial derivative with respect to $\rho_j$, $j\neq i$,
we obtain $N-1$ equations. Since $v_i$ is an interior node, these $N-1$ equations
can be rewritten as $n_i$ systems of equations which are related to $n_i$ subgraphs sharing the same node $a_i$.
Notice that the number of the nodes of each subgraphs is smaller than $N-1$.
According to our induction assumption,  it holds that $\lambda>0$, 
for any path $a_{l_0}a_{l_1}a_{l_2} \cdots a_{l_m}, m\le N-1$, from $a_{l_0}=a_i$ to a boundary point $a_{l_m}$,
the probability density $\rho_{l_j}, 0\le j\le m$ 
is increasing and $\frac {\rho_{l_{j+1}}}{\rho_{l_{j}}}, 0\le j\le m-1$, is decreasing.
 
Case 2: $a_i$ is a boundary node of the graph. By the Lagrange multiplier method, with
$\phi(t)=1-t-\log(t)$, we obtain
\begin{align*}
\sum_{l\in N(j)}\phi(\frac {\rho_{l}}{\rho_{j}})=\lambda,\; \text{if} \;j\notin N(i).\\
\sum_{l\in N(j), l\neq i}\phi(\frac {\rho_{l}}{\rho_{j}})+\phi(\frac {c}{\rho_{j}})=\lambda, \;\text{if}\; j \in N(i).
\end{align*}
We first show that $\lambda>0$. Assume that $\lambda\le 0$.
If $V_B$ has only two nodes, then by the monotonicity of $\phi$,  it holds that $\rho$ is decreasing along the path
$a_{l_0}a_{l_1}a_{l_2} \cdots a_{l_m}$ from $a_{l_0}\neq a_i$ to an interior node $a_{l_m}$. 
From the connectivity of the graph, we have  $c\ge \rho_{l}, l\le N$, which
leads to the contradiction that $\sum_{l=1}^N\rho_{l}=1\le Nc<1$. 

If $V_B$ has more than two nodes, then there must exist an interior node with at least 3 outgoing edges.
 Denote $a_e$ the  farthest  interior node from $a_i$ which has 3 or more outgoing edges. Since $a_e$ is connected to $a_i$ by a road, we denote $a_{e_1}$ the point that is closet to $a_e$ and belongs to such road. 
Then at the node $a_e$, we have
$\sum_{l\in N(e), l\neq e_1}\phi(\frac {\rho_{l}}{\rho_{e}})+
\phi(\frac {\rho_{e_1}}{\rho_{e}})=\lambda\le 0$.  
Denote $a_{b_l}, l\in N(e), l\neq e_1$ as the corresponding boundary node which contains the edge $a_la_e$.
Due to the monotonicity of $\phi$, $\lambda\le 0$ and the fact that $a_l$ belongs to the road only connecting $a_{b_l}$ and $a_e$, we have $\phi(\frac {\rho_{l}}{\rho_{e}})\ge 0$ for $l\in N(e), l\neq e_1$. This implies that the density along the road from $a_{b_l}$ to $a_e$ is decreasing and that $\phi(\frac {\rho_{e_1}}{\rho_{e}})\le 0$. Then we can view $a_e$ as a new boundary node of the left subgraph which is obtained by ignoring all the roads from $a_{b_l}$ to $a_e$ and repeat the above procedures until we get a subgraph which satisfying $a_i\in V_B$ and $V_B$ has only two nodes. 
And on the graph with two boundary nodes, the density is decreasing from another boundary point to $a_i$. This will leads to the contradiction that $\sum_{l=1}^N\rho_{l}=1\le Nc<1$. Thus  we conclude that $\lambda>0.$ 
 Following similar arguments, we obtain the increasing property of $\rho_{l_j}$  along the path $a_{l_0}a_{l_1}a_{l_2} \cdots a_{l_m}$, $m\le N-1$ from $a_{l_0}=a_i$ to any boundary node $a_{l_m}\in V_B$. 

Next, we show the decreasing property of $\frac {\rho_{l_{j+1}}}{\rho_{l_j}}$.
Since 
\begin{align*}
\sum_{l\in N(l_1),l\neq i, l_{2}}\phi(\frac {\rho_{l}}{\rho_{l_1}})+\phi(\frac {c}{\rho_{l_1}})+\phi(\frac {\rho_{l_2}}{\rho_{l_1}})&=\lambda>0,\\
\sum_{l\in N(l_j),l\neq l_{j-1},l_{j}}\phi(\frac {\rho_{l}}{\rho_{l_j}})+\phi(\frac {\rho_{l_{j-1}}}{\rho_{l_j}})+\phi(\frac {\rho_{l_{j+1}}}{\rho_{l_j}})&=\lambda> 0,\; 2\le j \le m-1,\\
\phi(\frac {\rho_{l_{m-1}}}{\rho_{l_m}})&=\lambda>0.
\end{align*}
The increasing property of $\rho$  along any path from $a_{i}$ to the node in $V_B$ yields that 
\begin{align*}
\phi(\frac {\rho_{l_{m-2}}}{\rho_{l_{m-1}}})&=\lambda-\sum_{l\in N(l_j),l\neq l_{m},l_{m-2}}\phi(\frac {\rho_{l}}{\rho_{l_{m-1}}})-\phi(\frac {\rho_{l_{m}}}{\rho_{l_{m-1}}})>\phi(\frac {\rho_{l_{m-1}}}{\rho_{l_m}}).
\end{align*}
The monotonicity of $\phi$ leads to $\frac {\rho_{l_{m-2}}}{\rho_{l_{m-1}}}<\frac {\rho_{l_{m-1}}}{\rho_{l_m}}$. By repeating the above procedures on $a_{l_j}$, $1\le j \le m-2$, we obtain that
\begin{align*}
&\phi(\frac {\rho_{l_{j-1}}}{\rho_{l_j}})+\phi(\frac {\rho_{l_{j}}}{\rho_{l_{j+1}}})+\sum_{l\in N(l_j),l\neq l_{j-1},l_{j+1}}\phi(\frac {\rho_{l}}{\rho_{l_{j}}})+\phi(\frac {\rho_{l_{j+1}}}{\rho_{l_{j}}})=\lambda+\phi(\frac {\rho_{l_{j}}}{\rho_{l_{j+1}}}).
\end{align*}
Notice that $\phi(t)+\phi(1/t)\le 0, t>0$ and that $\phi(\frac {\rho_{l}}{\rho_{l_{j}}})<0$ when $l\neq j-1$. As a consequence, we get that 
\begin{align*}
\phi(\frac {\rho_{l_{j-1}}}{\rho_{l_j}})\ge \lambda+\phi(\frac {\rho_{l_{j}}}{\rho_{l_{j+1}}}),
\end{align*}
which implies that  $\frac {\rho_{l_{j+1}}}{\rho_{l_{j}}}, 0\le j\le m-1$ is decreasing along the path from $a_i$ to any node in $V_B$.
Thus the results holds for the graph with $N$ nodes.

Now, we are going to derive the desired lower bound of the $\rho_{t}$.
Assume that $\kappa\le N-1$ is the numbers of nodes in $V_B$ and that $d_{max}$ is largest distance $d(a_i, a_{l_m})\le N-\kappa+1$ from $a_i$ to  $a_{l_m}$. Since $\sum_{i=1}^N\rho_{i}=1,$ there exists at least a node $a_{n}$ such that the density at $a_{n}>\frac {1-c}{N-1}$.
Then for the path $a_{l_0}a_{l_1}\cdots a_{l_j}\cdots a_{l_m}, a_{l_0}=v_i, a_{l_m}\in V_B,a_{l_j}=a_n$, $m\le d_{max}-1$, we have 
\begin{align*}
\sum_{r=0}^m\rho_{l_{r}}=c(1+\sum_{r=1}^m\frac {\rho_{l_{r}}}{c})
\end{align*}
Adding all the paths, which have $a_i$ as a common node, together, we obtain
\begin{align*}
c(1+\sum_{s=1}^{\kappa}\sum_{r=1}^{m_s}\frac {\rho_{l_{r}^s}}{c})\ge 1
\end{align*}
To find a lower bound of the ratio of $\frac {\rho_{l_{r+1}^s}}{\rho_{l_r^s}}$ for all the paths, we denote $k=\min_{s\le \kappa}\frac {\rho_{l_{m_s}^s}}{\rho_{l_{m_s-1}^s}}$ and let 
$
c(1+\sum_{s=1}^{\kappa}\sum_{r=1}^{m_s}k^{r})<1.
$
It suffices to require that $1+\kappa(d_{max}-1)k^{d_{max-1}}<\frac 1c$, i.e.,
$k\le (\frac {1-c}{c\kappa(d_{max}-1)})^{\frac 1{d_{max}-1}}.$
Thus it holds that $\min_{s\le \kappa}\frac {\rho_{l_{m_s}^s}}{\rho_{l_{m_s-1}^s}}\ge (\frac {1-c}{c\kappa(d_{max}-1)})^{\frac 1{d_{max}-1}}, $ if $c\le \frac 1{\kappa (d_{max}-1)+1}$.

When $c\le \frac 1{\kappa (d_{max}-1)+1}$, we get that for some path which contains the node $a_{l_j}$ whose density is large than $\frac {1-c}{N-1}$,
\begin{align*}
\min_{i\le N}\inf_{\rho_i=c} I(\rho)&\ge
\min_{i\le N}\inf_{\rho_i=c} \sum_{r=0}^{m_s^i-1}(\log(\rho_{l_{r}^i}^*)-\log(\rho_{l_{r+1}^i}^*))(\rho_{l_{r}^i}^*-\rho_{l_{r+1}^i}^*)\\
&\ge \min_{i\le N}\inf_{\rho_i=c} \log(\frac {\rho_{m_s^i}^*}{\rho_{m_{s}^i-1}^*})(\rho_{m_s^i}-c).\\
&\ge {\frac 1{d_{max}-1}} \log(\frac {1-c}{c\kappa(d_{max}-1)})(\frac {1-c}{N-1}-c).
\end{align*}
If there exists $\frac 12\min_i \rho_i(0) N \le \alpha<\min_i \rho_i(0) N, c=\alpha \frac {1}N$ such that $\inf_{\rho_1=c}I(\rho)\ge\frac {M_0}{\beta}$, then  $$\sup_{t\ge0}\min_{i\le N}\rho_i(t)\ge \frac 12\min_i \rho_i(0).$$
Otherwise,
taking ${\frac 1{d_{max}-1}} \log(\frac {1-c}{c\kappa(d_{max}-1)})(\frac {1-c}{N-1}-c)\ge \frac {\mathcal H_0}{\beta}$ and $c=\alpha\min(\frac 1N,\frac 1{(d_{max}-1)\kappa+1})$, where $\alpha<\frac 12N\min_i\rho_{i}(0)$, we obtain the lower bound as 
\begin{align*}
\sup_{t}\min_{i}\rho_{i}(t)\ge \frac {1}{1+\kappa(d_{max}-1)\exp(2\frac {M(d_{max}-1)(N-1)}{\beta})}.
\end{align*}
Combining all cases above, we have the following lower bound estimate 
\begin{align*}
\sup_{t}\min_{i}\rho_{i}(t)\ge 
\min\Big(\frac 12 \min_i(\rho_i(0)),\frac {1}{1+\kappa(d_{max}-1)\exp(2\frac {M(d_{max}-1)(N-1)}{\beta})}\Big).
\end{align*}

\end{proof}

\end{document}